\documentclass[reqno]{amsart}
\usepackage{amsmath,amsthm,amssymb,bm}
\usepackage{hyperref}
\usepackage{a4wide}
\usepackage{cleveref}
\usepackage{graphicx,color}
\usepackage{tikz}
\usepackage{ytableau}
\usepackage{kotex}
\numberwithin{equation}{section}
\usepackage{mathdots}

\newtheorem{thm}{Theorem}[section]
\newtheorem{lem}[thm]{Lemma}
\newtheorem{prop}[thm]{Proposition}
\newtheorem{cor}[thm]{Corollary}
\theoremstyle{definition}

\newtheorem{defn}[thm]{Definition}
\newtheorem{conj}[thm]{Conjecture}
\newtheorem{algo}[thm]{Algorithm}
\newtheorem{question}[thm]{Question}
\newtheorem{problem}[thm]{Problem}
\newtheorem{remark}[thm]{Remark}

\crefname{claim}{Claim}{Claims}
\crefname{lem}{Lemma}{Lemmas}
\crefname{thm}{Theorem}{Theorems}
\crefname{prop}{Proposition}{Propositions}
\crefname{question}{Question}{Questions}
\crefname{conj}{Conjecture}{Conjectures}
\crefname{figure}{Figure}{Figures}
\crefname{cor}{Corollary}{Corollaries} 
\crefformat{equation}{(#2#1#3)}
\Crefformat{equation}{Equation #2(#1)#3}
\crefname{defn}{Definition}{Definitions}

\newcommand\SW{\operatorname{SW}}
\newcommand\PA{\mathcal{A}}
\newcommand\area{\operatorname{area}}

\newcommand\asc{\operatorname{asc}}
\newcommand\PT{\mathcal{T}}

\newcommand\Sort{\operatorname{Sort}}

\newcommand\inv{\operatorname{inv}}

\newcommand\QQ{\mathbb{Q}}
\newcommand{\CC}{\mathbb{C}}
\newcommand{\ZZ}{\mathbb{Z}}

\newcommand{\HH}{\mathbb{H}}

\newcommand\mm{\mathbf{m}}
\newcommand\pp{\mathbf{p}}
\renewcommand\vec[1]{\boldsymbol{#1}}

\newcommand\pat{\operatorname{pat}}
\newcommand\pr{\operatorname{pr}}

\newcommand\SYT{\operatorname{SYT}}
\newcommand\wt{\operatorname{wt}}
\newcommand\AO{\mathcal{O}}
\newcommand\sink{\operatorname{sink}}
\newcommand\ch{\operatorname{ch}}
\newcommand\Hess{\operatorname{Hess}}
\newcommand\calC{\mathcal{C}}
\newcommand\qand{\quad\mbox{and}\quad}

\newcommand\board[2]{\foreach \i in {0,1,...,5}
  \draw[color=gray!70] (#1+\i,#2) -- (#1+\i,#2+5) (#1,\i+#2) -- (#1+5,\i+#2)
  (#1,#2) -- (#1+5,#2+5);
  \draw (#1,#2+5) -- (#1,#2) -- (#1+5,#2);
}  

\renewcommand\emph[1]{\textcolor{blue}{\it #1}}

\title[Refinement of Hikita's $e$-positivity theorem]{Refinement of Hikita's $e$-positivity theorem via Abreu--Nigro's $g$-functions and restricted modular law}
\author{JiSun Huh}
\thanks{J. Huh was supported by the National Research Foundation of Korea (No. NRF-2020R1A2C1A01011045).}
\address{Department of Mathematics, University of Seoul}
\email{hyunyjia@yonsei.ac.kr}

\author{Byung-Hak Hwang}
\thanks{B.-H. Hwang was supported by KIAS Individual Grant (MG098201) at Korea Institute for Advanced Study.}
\address{School of Mathematics, Korea Institute for Advanced Study, Seoul, South Korea}
\email{byunghakhwang@gmail.com}

\author{Donghyun Kim}
\thanks{D. Kim  was supported by the National Research Foundation of Korea (NRF) grant funded by the Korean government (MEST) (No. 2019R1A6A1A10073437) and individual NRF grants 2022R1I1A1A01070620.}
\address{Department of Mathematical Sciences, Seoul National University}
\email{hyun920310@snu.ac.kr}

\author{Jang Soo Kim}
\address{Department of Mathematics, Sungkyunkwan University, Suwon, Korea}
\email{jangsookim@skku.edu}

\author{Jaeseong Oh}
\thanks{J. Oh was supported by KIAS Individual Grant (HP083401) and the National Research Foundation of Korea (NRF) grant MSIT NRF-2022R1C1C1010300} \address{June E Huh Center for Mathematical Challenges,
Korea Institute for Advanced Study}
\email{jsoh@kias.re.kr}


\begin{document}

\begin{abstract}
  We study the symmetric functions \( g_{\mm,k}(x;q) \), introduced by
  Abreu and Nigro for a Hessenberg function \( \mm \) and a positive
  integer \( k \), which refine the chromatic symmetric function.
  Building on Hikita's recent breakthrough on the Stanley--Stembridge
  conjecture, we prove the \( e \)-positivity of \( g_{\mm,k}(x;1) \),
  refining Hikita's result. We also provide a Schur expansion of the
  sum \( \sum_{k=1}^n e_k(x) g_{\mm,n-k}(x;q) \) in terms of
  \( P \)-tableaux with 1 in the upper-left corner. We introduce a
  restricted version of the modular law as our main tool. Then, we
  show that any function satisfying the restricted modular law is
  determined by its values on disjoint unions of path graphs.
\end{abstract}

\maketitle


\section{Introduction}
\label{sec:introduction}

Since its introduction, chromatic symmetric functions have been central to algebraic combinatorics.
In his seminal paper \cite{Stanley1995}, Stanley introduced a generalization of the chromatic
polynomial of a graph. Given a graph \( G = (V,E) \), a \emph{proper coloring} \( \kappa: 
V\rightarrow \ZZ_{\ge 1}\) is an assignment of a positive integer to each vertex of \( G \) such
that if \( v \) and \( u \) are joined by an edge, then \(\kappa(v) \neq \kappa(u)\).
Let \(\calC(G)\) be the set of all proper colorings of $G$.
The \emph{chromatic symmetric function} \( X_G(x)\) is the formal power series in 
infinitely many variables \( x_1, x_2, \dots \) defined by  
\[
  X_G(x) := \sum_{\kappa\in\calC(G)} x_\kappa,
\]
where \( x_\kappa = \prod_{v\in V} x_{\kappa(v)} \).
By definition, \( X_G(x) \) is a symmetric function.

While this function possesses rich combinatorial properties, its significance grew further when deep 
connections to algebra and geometry were revealed.
In what follows, we review some key developments in the study of chromatic symmetric functions.
See \Cref{sec:preliminary} for any undefined terms and notations.

A notable extension of \( X_G(x) \) was introduced by Shareshian and Wachs in \cite{Shareshian2016}.
For a graph \( G=(V,E)\) with \( V=[n]:=\{ 1,\dots,n\} \), the \emph{chromatic quasisymmetric function} is defined by  
\[
  X_G(x;q) := \sum_{\kappa\in\calC(G)} q^{\asc_G(\kappa)} x_\kappa,
\]
where \( \asc_G(\kappa) \) is the number of edges \( (i,j) \) such that \( i<j \) and \( \kappa(i)
< \kappa(j) \). By definition, \( X_G(x) = X_G(x;1) \).
In general, \( X_G(x;q) \) is not a symmetric function, but Shareshian and
Wachs~\cite{Shareshian2016} proved that for a specific class of graphs, namely the incomparability 
graphs of natural unit interval orders, \( X_G(x;q) \) is symmetric.
From now on, we denote by \( X_\mm(x;q) \) the chromatic quasisymmetric function of the
incomparability graph of a natural unit interval order \( \mm \).

A fundamental problem in the study of symmetric functions is finding their expansions in terms of 
various symmetric function bases.
When Stanley introduced chromatic symmetric functions, he also proposed a conjecture regarding their 
expansion in terms of elementary symmetric functions. This conjecture first appeared in
\cite{Stanley1993} in a different form. Later, Shareshian and Wachs
refined the conjecture as follows.

\begin{conj}[\cite{Stanley1993,Stanley1995,Shareshian2016}]\label{conj:e-positivity_intro}
  Let \( \mm \) be a natural unit interval order on \( [n] \). Suppose that
  \[
    X_\mm(x;q) = \sum_{\lambda\vdash n} c_\lambda(q) e_\lambda(x).
  \]
  Then \( X_\mm(x;q) \) is \( e \)-positive, that is,
  \( c_\lambda(q) \) is a polynomial in \( q \) with nonnegative integer coefficients.
\end{conj}

Since its proposal, this conjecture has become one of the most famous open problems in algebraic
combinatorics, and it has been the subject of extensive research.
Although it is still open, a major breakthrough was recently announced by
Hikita \cite{Hikita2024}.

\begin{thm}[\cite{Hikita2024}]\label{thm:Hikita_intro}
  The chromatic symmetric function \( X_\mm(x;1) \) of a natural unit interval order \( \mm \)
  is \( e \)-positive.
  In other words, \( c_\lambda(1) \geq 0 \) for all \( \lambda \).
\end{thm}

His approach is fundamentally different from the ones used in previous works.
We brief the approach here.
Let \( \mm \) be a natural unit interval order on \( [n] \).
To each standard Young tableau \( T \), he associated a rational function \( p_\mm(T;q) \),
and showed the following expansion\footnote{Our definition is
slightly different from Hikita's. See \Cref{sec:comp-with-hikit}.}:
\[
  X_\mm(x;q) = \sum_{\lambda\vdash n} \prod_{i=1}^{\ell(\lambda)} [\lambda_i]_q!
    \sum_{T\in\SYT(\lambda)} p_\mm(T;q) e_\lambda(x).
\]
He further provided a probabilistic interpretation of \( p_\mm(T;q) \), which ensures that its 
specialization at \( q=1 \) is nonnegative. Consequently, he established the \( e \)-positivity
of \( X_\mm(x;1) \).

Very recently, Griffin--Mellit--Romero--Weigl--Wen~\cite{Griffin2025}
independently proved \Cref{thm:Hikita_intro} by deriving the Macdonald
expansion of the chromatic quasisymmetric functions.

Another remarkable result is a connection between chromatic quasisymmetric functions and Hessenberg
varieties. There is an obvious correspondence between natural unit interval orders and
Hessenberg functions, so we identify them.
Given a Hessenberg function $\mm$ of length \( n \) and a regular semisimple \( n \times n \) matrix \( S \),
the (regular semisimple) \emph{Hessenberg variety} is defined as
\[
  \Hess_\mm(S) := \{(V_1,\dots,V_n) : SV_i \subset V_{\mm(i)} \text{ for } i \in [n]\},
\]
which lies inside the flag variety of \( \CC^n \). Since the cohomology of
\( \Hess_\mm(S) \) is independent of the choice of \( S \), we write \( \Hess_\mm \)
for brevity. In~\cite{Shareshian2016}, Shareshian and Wachs conjectured that the Frobenius
characteristic of the cohomology of \( \Hess_\mm \), equipped with Tymoczko’s dot
action~\cite{Tymoczko2008}, coincides with the chromatic quasisymmetric function
\( X_\mm(x;q) \). This conjecture was subsequently proved independently by
Brosnan--Chow~\cite{Brosnan2018} and Guay-Paquet~\cite{Guay-Paquet2016}.
\begin{thm}[\cite{Brosnan2018, Guay-Paquet2016}]\label{thm:X=Hessenberg_intro}
  Let \( \mm \) be a natural unit interval order on \( [n] \). Then
  \[
    X_\mm(x;q) = \omega (\ch H^*(\Hess_\mm)).
  \]
\end{thm}
Building on this connection, Abreu and Nigro~\cite{Abreu2023} established a deeper link between the geometry
of Hessenberg varieties and the structure of chromatic quasisymmetric functions. Consider the
forgetful map \( \Hess_\mm \to \mathbb{P}^{n-1} \), given by sending \( (V_1, \dots, V_n)
\mapsto V_1 \). Applying the decomposition theorem of Beilinson, Bernstein, and
Deligne~\cite{Beilinson} to this map,
Abreu and Nigro~\cite{Abreu2023}
expressed the cohomology of \( \Hess_\mm \) as
\begin{equation}\label{eq:decomposition of H*(Hess)}
    H^*(\Hess_\mm) = \bigoplus_{k=0}^{n-1} H^*(\tilde{H}_k) \otimes L_k,
\end{equation}
where \( \tilde{H}_k \) is a disjoint union of \( \binom{n}{k} \) copies of
\( \mathbb{P}^{n-k-1} \) inside $\mathbb{P}^{n-1}$, and \( L_k \) is a certain
\( \mathfrak{S}_k \)-module. They provided an explicit combinatorial description
of \( g_{\mm,k}(x;q) := \omega (\ch(L_k)) \), thereby obtaining the following decomposition of
\( X_\mm(x;q) \) from \eqref{eq:decomposition of H*(Hess)}.

\begin{thm}[{\cite[Theorem~1.7]{Abreu2023}}] \label{thm:X=g_intro}
  Let \( \mm \) be a natural unit interval order on \( [n] \). Then
  \begin{equation} \label{eq:g_intro}
    X_\mm(x;q) = \sum_{k=1}^n [k]_q  e_k(x)  g_{\mm,n-k}(x;q).
  \end{equation}
\end{thm}

They also proposed the following refinement of \Cref{conj:e-positivity_intro}.

\begin{conj}[{\cite[Conjecture~1.8]{Abreu2023}}] \label{conj:g_e-positive_intro}
  For a natural unit interval order \( \mm \) on \( [n] \) and $0\le k <n$,
  each \( g_{\mm,k}(x;q) \) is \( e \)-positive.
\end{conj}

Although finding a manifestly positive \( e \)-expansion of \( X_\mm(x;q) \)
is widely open, the Schur expansion has been established by
Gasharov and Shareshian--Wachs. In \cite{Gasharov1996}, Gasharov
introduced the notion of \( P \)-tableaux and used it to describe the
Schur expansion of \( X_\mm(x) \). Shareshian and Wachs
\cite{Shareshian2016} later extended Gasharov's approach to obtain the
Schur expansion of \( X_\mm(x;q) \).
\begin{thm}[\cite{Gasharov1996,Shareshian2016}] \label{thm:s-expansion_intro}
  Let \( \mm \) be a natural unit interval order on \( [n] \). Then we have  
  \[
    X_\mm(x;q) = \sum_{\lambda\vdash n} \sum_{T\in\PT_\mm(\lambda)} q^{\inv_\mm(T)} s_{\lambda}(x),
  \]
  where \( \PT_{\mm}(\lambda) \) is the set of \( P \)-tableaux of
  shape \( \lambda \) with respect to \( \mm \).
\end{thm}

Since \( g_{\mm,k}(x;q) \) is the character of the representation
\( L_k \), the symmetric function \( g_{\mm,k}(x;q) \) is Schur
positive. Therefore, the following question is natural.

\begin{question}[{\cite[Question~4.4]{Abreu2023}}] \label{ques:g_schur}
  Give a combinatorial description of the coefficients of the Schur expansion of
  \( g_{\mm,k}(x;q) \).
\end{question}

\subsection{Main Results}

As we reviewed above, the chromatic quasisymmetric function of a natural unit interval order
can be expressed in the following three ways:
\begin{align}
  X_\mm(x;q)
    &= \sum_{\lambda\vdash n} \prod_{i\ge 1} [\lambda_i]_q! \sum_{T\in\SYT(\lambda)}
      p_\mm(T;q) e_\lambda(x) \label{eq:X=Hikita_intro} \\
    &= \sum_{k=1}^{n} [k]_q e_k(x) g_{\mm,n-k}(x;q) \label{eq:X=g_intro} \\
    &= \sum_{\lambda\vdash n} \sum_{T\in\PT_\mm(\lambda)} q^{\inv_\mm(T)} s_\lambda(x).
          \label{eq:X=PTab_intro}
\end{align}
In this paper, we introduce a variant of each of the three
expressions, and show that the variants coincide. As a corollary, we
prove \Cref{conj:g_e-positive_intro} when \( q=1 \).

First we introduce a variant of \eqref{eq:X=Hikita_intro}, the \( e \)-expansion from Hikita's
work. For a partition
\( \lambda \) and \( k\ge 1 \), we denote by \( \SYT_k(\lambda) \) the set of standard Young
tableaux of shape \( \lambda \) whose largest entry lies on the \( k \)-th column.
Then for a natural unit interval order \( \mm \) on \( [n] \) and \( 1\le k\le n \), we define
\begin{equation} \label{eq:E_k_def_intro}
  E_{\mm,k}(x;q) := \sum_{\lambda\vdash n} \frac{c_{\lambda,k}(\mm;q)}{[k]_q} e_\lambda(x),
\end{equation}
and
\begin{equation} \label{eq:E_def_intro}
  E_{\mm}(x;q)
    := \sum_{k=1}^{n} E_{\mm,k}(x;q)
    = \sum_{k=1}^{n} \sum_{\lambda\vdash n} \frac{c_{\lambda,k}(\mm;q)}{[k]_q} e_\lambda(x),
\end{equation}
where
\[
  c_{\lambda,k}(\mm;q) :=
    \prod_{i\ge1} [\lambda_i]_q! \sum_{T\in\SYT_k(\lambda)} p_\mm(T;q).
\]

We next define a variant of \eqref{eq:X=g_intro} which is expressed in terms of Abreu and Nigro's
\( g \) functions: for a natural unit interval order \( \mm \) on \( [n] \) and \( 1\le k \le n \),
\begin{equation} \label{eq:G_k_def_intro}
  G_{\mm,k}(x;q) := e_k(x) g_{\mm,n-k}(x;q),
\end{equation}
and
\begin{equation} \label{eq:G_def_intro}
  G_{\mm}(x;q) := \sum_{k=1}^{n} G_{\mm,k}(x;q) = \sum_{k=1}^n e_k(x) g_{\mm,n-k}(x;q).
\end{equation}

Lastly, we give a variant of the Schur expansion~\eqref{eq:X=PTab_intro} of \( X_\mm(x;q) \).
Let \( \PT'_\mm(\lambda) \) be the set of \( P \)-tableaux \( T \) of shape \( \lambda \) with
\( T(1,1) = 1 \). Using this set, we introduce the following symmetric function expressed
in terms of Schur functions: for a natural unit interval order \( \mm \) on \( [n] \),
\begin{equation} \label{eq:S_def_intro}
  S_\mm(x;q) := \sum_{\lambda\vdash n} \sum_{T\in\PT'_\mm(\lambda)}
    q^{\inv_\mm(T)} s_{\lambda}(x).
\end{equation}
The main result of the paper is the following.
\begin{thm} \label{thm:E=G=S_intro}
  For a natural unit interval order \( \mm \) on \( [n] \), we have
  \[
    E_\mm(x;q) = G_\mm(x;q) = S_\mm(x;q).
  \]
  Furthermore, for \( 1\le k\le n \), we have
  \[
    E_{\mm,k}(x;q) = G_{\mm,k}(x;q).
  \]
\end{thm}
Our strategy to prove \Cref{thm:E=G=S_intro} is to restrict the modular law from \cite{Abreu2021}  
and show that the three symmetric functions \( G_\mm(x;q) \), \( E_\mm(x;q) \), and \( S_\mm(x;q) \)  
satisfy the restricted modular law.  

Let us first recall the modular law.  
In \cite{Guay-Paquet2013}, Guay-Paquet introduced a modular law for the chromatic symmetric  
functions of \( \mathbf{(3+1)} \)-free posets, and used it to show that the chromatic symmetric  
function of a \( \mathbf{(3+1)} \)-free poset can be expressed as a positive linear combination of 
those of natural unit interval orders.
Later, Abreu and Nigro~\cite{Abreu2021} refined this modular law to characterize chromatic  
quasisymmetric functions.

More precisely, they defined modular triples of type~I and type~II as triples  
\( (\mm,\mm',\mm'')\in\HH^3 \) satisfying certain conditions,
where \( \HH \) is the set of Hessenberg functions.  
A function \( f:\HH\rightarrow A \), where \( A \) is a \( \QQ(q) \)-algebra,  
is said to satisfy the modular law if,  
for any modular triple \( (\mm,\mm',\mm'') \) of type~I or type~II,  
we have
\[
  (1+q) f(\mm') = q f(\mm) + f(\mm'').
\]  
Using this, Abreu and Nigro~\cite{Abreu2021} showed that a function
\( f:\HH\rightarrow A \) satisfying the modular law is uniquely
determined by its values on disjoint unions of complete graphs. This
fact plays a crucial role in Hikita's
work~\cite{Hikita2024}.

While the chromatic quasisymmetric function \( X_\mm(x;q) \) satisfies
the modular law, Abreu and Nigro's \( g \)-functions, as well as the
three variants \( E_\mm(x;q) \), \( G_\mm(x;q) \), and
\( S_\mm(x;q) \), do not satisfy the modular law. To address this
issue, we introduce the restricted modular law, which can be regarded
as a weaker version of the modular law (\Cref{def:5}). Quite
surprisingly, we show that the restricted modular law alone is
sufficient to determine the entire set of values of a function
\( f \). More precisely, we prove that a function \( f \) satisfying
the restricted modular law is uniquely determined by its values on
disjoint unions of paths (\Cref{thm: master}). We then show that the
three symmetric functions \( E_\mm(x;q) \), \( G_\mm(x;q) \), and
\( S_\mm(x;q) \) satisfy the restricted modular law and that their
values on disjoint unions of paths coincide, thereby establishing
\Cref{thm:E=G=S_intro}.

Abreu and Nigro \cite[Questions~4.1 and 4.2]{Abreu2023} posed the
question of whether there exists a modification of the modular law
that the symmetric function \( g_{\mm,k}(x;q) \) satisfies, as well as
how to compute the symmetric function using the modification. Hence,
the restricted modular law provides an answer to their questions.

As direct corollaries of Theorem \ref{thm:E=G=S_intro}, we establish
\Cref{conj:g_e-positive_intro} at \( q=1 \) and the polynomiality of
the fraction \( c_{\lambda,k}(\mm;q) / [k]_q \).

\begin{cor} \label{cor:g(1)_e-positive_intro}
  For a natural unit interval order \( \mm \) on \( [n] \) and \( 0 \le k < n \),
  the symmetric function \( g_{\mm,k}(x;1) \) is \( e \)-positive.
\end{cor}
\begin{cor} \label{cor:c_lambda_k_polynomial_intro}
  For a natural unit interval order \( \mm \) on \( [n] \), a partition \( \lambda\vdash n \),
  and \( 1\le k\le n \), the quotient \( c_{\lambda,k}(\mm;q) / [k]_q \) is a polynomial
  in \( q \).
\end{cor}

Tom and Vailaya independently showed
\Cref{cor:c_lambda_k_polynomial_intro} in their recent
paper~\cite{Tom2025}.

In addition, our main theorem gives a partial answer to \Cref{ques:g_schur}.
The equality
\begin{equation} \label{eq:G_schur_expansion}
  \sum_{k=1}^n e_k(x) g_{\mm,n-k}(x;q) =\sum_{\lambda \vdash n} \sum_{T\in\PT'_\mm(\lambda)} q^{\inv_\mm(T)} s_\lambda(x).
\end{equation}
gives the Schur expansion
of the sum of \( e_k(x) g_{\mm,n-k}(x;q) \).

Notice that the two symmetric functions \( E_\mm(x;q) \) and
\( G_\mm(x;q) \) are given by the sums of their refinements
\( E_{\mm,k}(x;q) \) and \( G_{\mm,k}(x;q) \). Hence finding an
appropriate refinement of \( S_\mm(x;q) \) would be an interesting
problem, giving a complete answer to \Cref{ques:g_schur}.
\begin{problem} \label{prob:find_S_k}
  Further refine the set \( \PT'_{\mm}(\lambda) \)
  into \( \PT'_{\mm}(\lambda) = \bigsqcup_{k\ge1}\PT'_{\mm,k}(\lambda) \)
  so that
 \[
 \sum_{\lambda\vdash n}   \sum_{T\in \PT'_{\mm,k}(\lambda)} q^{\inv_\mm(T)}s_{\lambda}(x) =
 E_{\mm,k}(x;q) = G_{\mm,k}(x;q).
 \] 
\end{problem}

\subsection{Organization}
This paper is organized as follows. In Section \ref{sec:preliminary},
we set up basic notations and go over a background on chromatic
quasisymmetric functions. In \Cref{sec:restr-modular-laws}, we
introduce the restricted modular law and present an algorithm that
reduces the proof of Theorem \ref{thm:E=G=S_intro} to the base case
(disjoint unions of path graphs). In Section \ref{sec: rm for EGS}, we
prove that our functions in consideration satisfy restricted modular
law. Then we complete the proof of Theorem \ref{thm:E=G=S_intro} for
the base case in Sections \ref{sec: multiplicativity} and \ref{sec:
  path case}. Section~\ref{sec: sink} presents an analog of Stanley's
sink theorem. In Appendix \ref{sec:comp-with-hikit}, we compare our
transition probability with Hikita's.

\section{Preliminaries} \label{sec:preliminary}

For an integer \( k\ge0 \), let \( [k]_q = 1 + \cdots + q^{k-1} \) and
\( [k]_q!=\prod_{i=1}^{k}[i]_q \). We denote by \( \mathfrak{S}_n \)
the set of permutations of \( [n] \).

\subsection{Compositions, partitions, and symmetric functions}

A \emph{composition} (respectively \emph{weak composition}) is a
sequence \( \alpha=(\alpha_1,\dots,\alpha_{\ell}) \) of positive
integers (respectively \emph{nonnegative integers}). We say that
\( \alpha \) is a composition of \( n \), denoted by
\( \alpha \vDash n \), if \( |\alpha|:=\sum_{i=1}^\ell \alpha_i=n \).
The \emph{length} of \( \alpha \) is \( \ell(\alpha) = \ell \).
For a (weak) composition \( \alpha \), we associate a diagram
\[
  D(\alpha) := \{(i,j)\in \mathbb{Z}_{\geq 1}\times \mathbb{Z}_{\geq1} : j\le \alpha_i\}.
\]
Each element of \( D(\alpha) \) is called a \emph{cell}.

A \emph{partition} is a weakly decreasing composition. For a partition
\( \lambda=(\lambda_1,\dots,\lambda_\ell) \), each \( \lambda_i \) is
called a \emph{part} of \( \lambda \). We say that \(\lambda\) is a
partition of \(n\), denoted by \(\lambda \vdash n\), if
\( |\lambda| = n \). We identify a partition \( \lambda \) with its
diagram. We use the English notation to display the diagrams of
partitions. For example, let \( \lambda=(4,3,1) \), then we depict the
diagram as follows:
\[
  \lambda=
  \vcenter{\hbox{
      \scalebox{0.8}{%
  \begin{ytableau}
    ~ & & &  \\
    ~ & & \\
    ~
  \end{ytableau}
}}}
\]
We denote the empty partition by \(\emptyset\).
The \emph{conjugate partition} of \(\lambda\), denoted by
\(\lambda'=(\lambda'_1,\lambda'_2,\dots)\), is the partition whose diagram is
\( \{(j,i)\in \mathbb{Z}_{\geq 1}\times \mathbb{Z}_{\geq 1} : j\le
\lambda_i\} \).

For two partitions \( \lambda \) and \( \mu \) with \( \mu\subseteq \lambda \) (as diagrams),
we denote by \( \lambda/\mu \) the difference \( \lambda\setminus \mu \) of their diagrams,
and \( |\lambda/\mu| := |\lambda| - |\mu| \).
The diagram \( \lambda/\mu \) is called a \emph{vertical strip} if each row of \( \lambda/\mu \)
consists of at most one cell.

Consider a filling \( T \) of a partition \( \lambda \) with positive integers.
For a cell \( (i,j)\in \lambda \), the corresponding integer of \( T \)
is denoted by \( T(i,j) \). We say that \( T \) is a
\emph{semistandard Young tableau} of shape \( \lambda \) if it is weakly increasing
along each row and strictly increasing along each column. If a semistandard Young tableau \( T \)
consists of \( 1,2,\dots,|\lambda| \), we say that \( T \) is a
\emph{standard Young tableau} of shape \( \lambda \). The set of standard Young tableaux
of shape \( \lambda \) is denoted by \( \SYT(\lambda) \) and we let
\( \SYT(n) \) to be \( \bigcup_{\lambda \vdash n}\SYT(\lambda) \).
Furthermore, for a partition \( \lambda\vdash n \) and \( 1\le k\le n \), we denote
by \( \SYT_k(\lambda) \) the set of standard Young tableaux
of shape \( \lambda \) such that the entry \( n \) is in the \( k \)-th column of \( T \).

For \( n\ge 0 \), let \( \Lambda_n \) be the space of symmetric functions of
homogeneous degree \( n \) in infinitely many variables \(x_1,x_2,\dots\)
over the ground field \(\QQ(q)\) where \( q \) is a formal parameter.
Let \(\Lambda=\bigoplus_{n\ge 0} \Lambda_n \) be the graded algebra of symmetric functions.
We use the following standard notations:
\(m_\lambda(x)\) for the \emph{monomial symmetric function},
\(h_\lambda(x)\) for the \emph{complete homogeneous symmetric function},
\(e_\lambda(x)\) for the \emph{elementary symmetric function},
\(s_\lambda(x)\) for the \emph{Schur function}.
The \emph{Hall inner product} on symmetric functions,
denoted by \(\langle -,-\rangle\), is defined as
\( \langle s_\lambda(x), s_\mu(x) \rangle = \delta_{\lambda,\mu}\).
It is well known that the monomial symmetric functions and the complete homogeneous
symmetric functions are dual to each other with respect to the inner product, i.e.,
\( \langle m_\lambda(x), h_\mu(x) \rangle = \delta_{\lambda,\mu} \).
The involution \( \omega \) on \( \Lambda \) is defined by
\( \omega(h_{\lambda}(x))=e_{\lambda}(x) \).

For a partition \( \lambda\vdash n \) and a weak composition
\( \alpha=(\alpha_1,\dots, \alpha_\ell) \) of \( n \), the \emph{Kostka
  number} \( K_{\lambda,\alpha} \) is the number of semistandard Young
tableaux \( T \) of shape \( \lambda \) such that for each
\( 1\le i\le \ell \), the number of appearances of the integer \( i \)
in \( T \) is \( \alpha_i \). Then the elementary symmetric function
can be written in terms of Schur functions as follows:
\begin{equation} \label{eq:Kostka}
  e_\alpha(x) := e_{\alpha_1}(x)\cdots e_{\alpha_\ell}(x)
    = \sum_{\lambda\vdash n} K_{\lambda',\alpha} s_\lambda(x).
\end{equation}

\subsection{Hessenberg functions and associated objects}

Let \( n\ge 1 \). A \emph{Hessenberg function} of length \( n \) is a weakly increasing
function \( \mm:[n]\to[n] \) such that \( i\le \mm(i) \le n \) for each \( i\in[n] \).
We also consider the Hessenberg function \( \mm \) as an \( n \)-tuple
\( (\mm(1),\dots,\mm(n)) \). We denote by \( \HH_n \) the set
of Hessenberg functions of length \( n \) and let \( \HH=\bigcup_{n\ge 1} \HH_n \).

There is a natural correspondence between Hessenberg functions of length \( n \) and
Dyck paths of length \( 2n \), where the \( i \)-th entry of a Hessenberg function corresponds
to the height of the \( i \)-th horizontal step of the associated Dyck path.
Thus, we often depict a Hessenberg function by its corresponding Dyck path. See \Cref{fig:lolli}.
In this sense, we define the \emph{area} of \( \mm \), denoted by \( \area(\mm) \), to be
\[
  \area(\mm) := \sum_{i=1}^n (\mm(i)-i).
\]
Regarding \( \mm \) as a Dyck path, the area is the number of unit cells between the main diagonal
and the path \( \mm \).

To each Hessenberg function \( \mm \), we associate a poset \( ([n], <_\mm) \) as follows:
for \( i, j \in [n] \), we set \( i <_\mm j \) if and only if \( \mm(i) < j \).
The posets arising in this way are precisely the \emph{natural unit interval orders}.
By abuse of notation, we also write \( \mm \) for the associated poset.

For a poset \( P \), the \emph{incomparability graph} of \( P \) is
the graph with vertex set \( P \) and edge set
\( \{(u,v) : \mbox{\( u<v \), \( u \) and \( v \) are incomparable in
  \( P \)}\} \). Hence, the incomparability graph of a natural unit interval order
\( \mm \) is
the graph whose vertex set is \( [n] \) such that for \( i<j \), \( i \)
and \( j \) are joined by an edge if and only if \( j \le \mm(i) \).
By abuse of notation, we denote the incomparability graph by
\( \mm \) when there is no possible confusion. See \Cref{fig:lolli}.

For \( \mm_1\in \HH_{n_1} \) and \( \mm_2\in \HH_{n_2} \),
we define \( \mm_1+\mm_2 \) to be the Hessenberg function given by
\[
  \mm_1+\mm_2 :=(\mm_1(1),\dots, \mm_1(n_1),\mm_2(1)+n_1,\dots,\mm_2(n_2)+n_1)
  \in \HH_{n_1+n_2}.
\]
Therefore, when we regard Hessenberg functions as graphs, the graph \( \mm_1+\mm_2 \) is the disjoint 
union of the graphs \( \mm_1 \) and \( \mm_2 \).

We denote by \( \pp_n\in \HH_n \) the Hessenberg function whose
associated incomparability graph is a path graph. Specifically it is
given as \( \pp_1=(1) \) and \( \pp_n=(2,3,\dots,n,n) \) if
\( n\geq2 \).

\begin{figure}
  \centering

\begin{tikzpicture}[scale=0.5]

  \foreach \i in {0,1,...,5}
  \draw[color=gray!70] (\i,0) -- (\i,5);

  \foreach \j in {0,1,...,5}
  \draw[color=gray!70]  (0,\j) -- (5,\j);

  \draw (0,0) -- (5,0);
  \draw (0,0) -- (0,5);
  \draw[color=gray!70] (0,0) -- (5,5);

  \draw[ultra thick]
  (0,0) -- (0,2) -- (1,2) -- (1,3) -- (2,3) -- (2,5) -- (5,5);
\end{tikzpicture}
\qquad \qquad
\begin{tikzpicture}[scale=1.1]
  \draw [thick](.15,.85)--(.85, .15)
  (1, 0.2)--(1, .85)
  (1.15, .85)--(1.85, .15)
  (2, .2)--(2, .85)
  (1.15, .2) -- (1.85,0.85);
  \node () at (0,1) {\( 3 \)};
  \node () at (1,0) {\( 1 \)};
  \node () at (1,1) {\( 4 \)};
  \node () at (2,0) {\( 2 \)};
  \node () at (2,1) {\( 5 \)};
  \end{tikzpicture}
\qquad \qquad
\begin{tikzpicture}[scale=0.9]
  \foreach \i in {1,...,5}
  \filldraw (\i,2) circle (1.5pt);
  \foreach \i in {1,...,5}
  \node at (\i,1.6) {\i}; 
  \draw[thick]
  (1,2) -- (5,2);
  \draw[thick]
  (5,2) arc [start angle=0, end angle=180, radius=1];
\end{tikzpicture}

\caption{The Dyck path, the Hasse diagram of the natural unit interval order, and the incomparability graph of
the natural unit interval order associated with \( \mm=(2,3,5,5,5)\in \HH_5 \).}
\label{fig:lolli}
\end{figure}
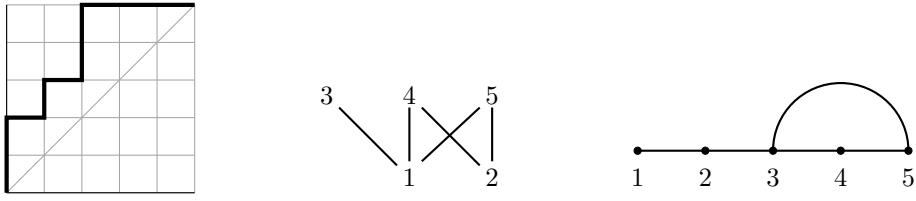 

\subsection{Chromatic symmetric functions}
Let us begin with recalling the definition of chromatic quasisymmetric functions.
While chromatic quasisymmetric functions are defined for arbitrary graphs on \( [n] \),
throughout this paper, we focus on those of the incomparability graphs of natural unit interval
orders.
For a Hessenberg function \( \mm\in\HH_n \), the \emph{chromatic quasisymmetric function}
\( X_\mm(x;q) \) of \( \mm \) is defined to be
\[
  X_\mm(x;q) = \sum_{\kappa\in\calC(\mm)} q^{\inv_\mm(\kappa)} x_\kappa,
\]
where \( x_\kappa = \prod_{i\in [n]} x_{\kappa(i)} \), and
\[
  \inv_\mm(\kappa) :=
    | \{(i,j)\in [n]\times[n] : \mbox{\( i<j\le \mm(i) \) and \( \kappa(i)>\kappa(j) \)}\} |.
\]

We note that the original definition of \( X_\mm(x;q) \) in \cite{Shareshian2016} uses the
ascent statistic on proper coloring instead of \( \inv_\mm(\kappa) \).
However our graph is the incomparability graph \( \mm \) so that \( X_\mm(x;q) \) is symmetric.
Hence, by \cite[Corolloary~2.7]{Shareshian2016}, the above definition agrees with
the original one.

The following notion is used for describing the Schur expansion of chromatic quasisymmetric functions
of natural unit interval orders \( \mm \), and also defining \( S_\mm(x;q) \).
\begin{defn}\label{def:inv}
  Let \( \mm \) be a Hessenberg function of length \( n \). A \emph{\( P \)-tableau} of
  shape \( \lambda \) (with respect to the poset \( \mm \)) is a bijection \( T:\lambda\to [n] \)
  such that \( T(i,j) <_\mm T(i,j+1) \) and \( T(i,j) \not>_\mm T(i+1,j) \) whenever
  the cells are in \( \lambda \). The poset \( \mm \) should always be clear from the context.
  We denote by \( \PT_\mm(\lambda) \) the set of \( P \)-tableaux of shape \( \lambda \).
  For a \( P \)-tableau \( T \),
  \[
    \inv_\mm(T) := | \{ (i,j) \in [n]\times[n] :
      \mbox{\( i<j\le\mm(i) \), and \( j \) lies above \( i \) in \( T \)}\} |.
  \]
  Furthermore, we denote by \( \PT'_\mm(\lambda) \) the set of \( P \)-tableaux with \( T(1,1)=1 \).
\end{defn}
In \cite{Kim2024c}, the third and last authors investigated \( P \)-tableaux subject to constraints on certain cells analogous to those in the above definition.

In \cite{Abreu2021b,Abreu2023}, Abreu and Nigro introduced novel symmetric functions \( \rho_k(x;q) \)
and \( g_{\mm,k}(x;q) \) for describing the chromatic quasisymmetric function of \( \mm \), and
the decomposition \eqref{eq:g_intro}.
For \( k\ge 1 \), the symmetric function \( \rho_k(x;q) \) is defined recursively by
\[
    [n]_q h_n(x) = \sum_{i=1}^{n} h_{n-i}(x) \rho_{i}(x;q),
\]
and by convention, we set \( \rho_0(x;q) = 1 \).
For a partition \( \lambda \), let \( \rho_\lambda(x;q) = \rho_{\lambda_1}(x;q)\cdots
\rho_{\lambda_{\ell(\lambda)}}(x;q) \).

For a permutation \( \sigma \in \mathfrak{S}_n \), we write
\( \sigma = \tau_1 \tau_2 \cdots \tau_k \) for its cycle decomposition,
where each cycle \( \tau_i \) is written with its
smallest element first, and the cycles are ordered from left to right
in increasing order of their minimal elements.
We denote by \( \sigma^c \) the permutation obtained by removing the parentheses in
the cycle decomposition of \( \sigma \). We further define
\[
    \wt_\mm(\sigma) := \inv_{\mm}((\sigma^c)^{-1}) = |\{ (i,j)\in [n]\times[n] : \mbox{\( i<j\le\mm(i) \), and $j$ precedes $i$ in \(\sigma^c\)} \}|.
\]
Abreu and Nigro's \( g \)-function, \( g_{\mm,k}(x;q) \), is then defined by
\begin{equation} \label{eq:g_def}
  g_{\mm, k}(x;q) := \sum_{\substack{\sigma = \tau_1 \cdots \tau_j \in \mathfrak{S}_{n,\mm}, j\ge1 \\ |\tau_1| \ge n - k}}
    (-1)^{|\tau_1| - n + k} q^{\wt_\mm(\sigma)} h_{|\tau_1| - n + k}(x) \, \omega(\rho_{(|\tau_2|,\dots,|\tau_j|)}(x;q))
\end{equation}
for \( 0\le k < n \).
Here, $\mathfrak{S}_{n,\mm}$ is the set of permutations \( \sigma \)
in $\mathfrak{S}_n$ such that $\sigma(i)\le \mm(i)$ for all \( i\in[n] \).

\subsection{Hikita's model}

Hikita \cite{Hikita2024} provided an interpretation of the
coefficients of the \( e \)-expansion of \( X_{\mm}(x;q) \) for any natural unit
interval orders \( \mm \).
We recall Hikita's transition probability using the English notation and
give a variation of it. We use a slight modification of Hikita's
model. See \Cref{sec:comp-with-hikit} for the connection between the
two models.

For \(T \in \SYT(n)\) and an integer \(0 \leq r \leq n \), we define
  \(\vec \delta^{(r)}(T)=\left(\delta_1, \ldots, \delta_{n}\right)\) by
  \[
    \delta_i=
    \begin{cases}
      1 & \mbox{if the \( i \)-th column of \( T \) contains an integer greater than \( r \),}\\
      0 & \mbox{otherwise.}
    \end{cases}
  \]
  Note that the sequence
  \(\vec \delta^{(r)}(T)=\left(\delta_1, \ldots, \delta_{n}\right)\)
  can be written uniquely as
\[
  \vec\delta^{(r)}(T)
  =\left(1^{b_0}, 0^{a_1}, 1^{b_1}, \ldots, 0^{a_l}, 1^{b_l}, 0^{a_{l+1}}\right),
\]
for some positive integers \( a_1,\dots,a_{l}, b_1,\dots,b_l \), and
nonnegative integers \( b_0 \) and \( a_{l+1} \). Here, \( i^j \)
denotes the sequence \( i,\dots,i \) of \( j \) \( i \)'s.

For an integer \( k \) with \(0 \leq k \leq l\), we define
\( f_k^{(r)}(T) \) to be the tableau obtained from \( T \) by adding a
cell with an entry \( n+1 \) at the end of column \( c \), where
\[
  c=\sum_{i=1}^k a_i+\sum_{i=0}^k b_i+1.
\]
Since \( \delta_c=0 \) and \( \delta_{c-1}=1 \) (or \( c=1 \)), one
can easily see that \( f_k^{(r)}(T)\in \SYT(n+1) \). The
\emph{transition probability} \( \psi_k^{(r)}(T ; q) \) from \( T \)
to \( f_k^{(r)}(T) \) is defined by
  \begin{equation}\label{eq:15}
   \psi_k^{(r)}(T ; q)= q^{\sum_{i=k+1}^l b_i} \prod_{i=1}^k
    \frac{\left[\sum_{j=i+1}^k a_j+\sum_{j=i}^k
      b_j\right]_q}{\left[\sum_{j=i}^k a_j+\sum_{j=i}^k
      b_j\right]_q} \prod_{i=k+1}^l \frac{\left[\sum_{j=k+1}^i
      a_j+\sum_{j=k+1}^{i-1} b_j\right]_q}{\left[\sum_{j=k+1}^i
      a_j+\sum_{j=k+1}^i b_j\right]_q} .
\end{equation}

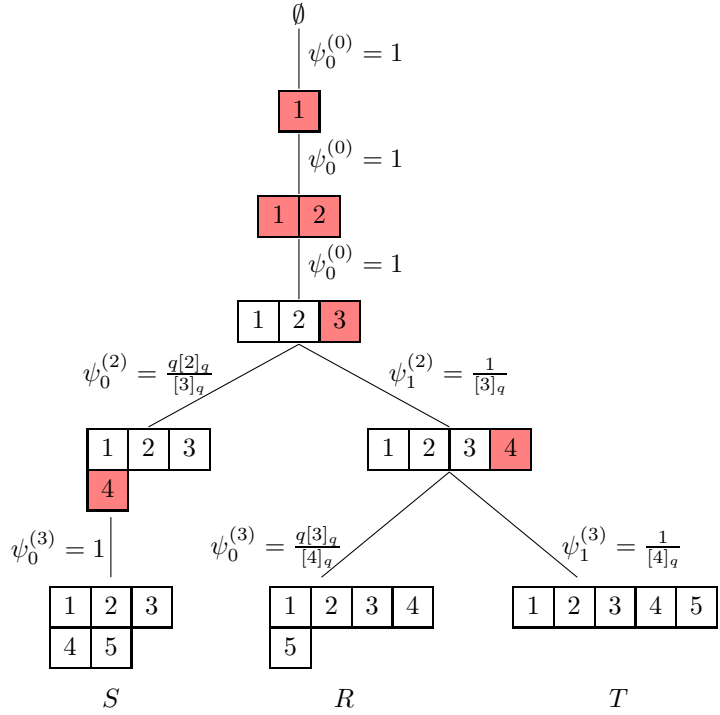
\begin{figure}
  \centering
\begin{tikzpicture}[scale=1]
  \node at (0,0) {\( \emptyset \)};
  \draw (0,-0.2) -- (0,-1);
  \node[right] at (0,-0.5) {\( \psi_0^{(0)}=1 \)};
  \node at (0,-1.3) {
    \begin{ytableau}
    *(red!50) 1
    \end{ytableau}
  };
  \draw (0,-1.6) -- (0,-2.4);
  \node[right] at (0,-1.9) {\( \psi_0^{(0)}=1 \)};
  \node at (0,-2.7) {
    \begin{ytableau}
    *(red!50) 1 & *(red!50) 2
    \end{ytableau}
  };
  \draw (0,-3) -- (0,-3.8);
  \node[right] at (0,-3.3) {\( \psi_0^{(0)}=1 \)};
  \draw (0,-4.4) -- (2,-5.5);
  \draw (0,-4.4) -- (-2,-5.5);
  \node at (-2,-4.8) {\( \psi_0^{(2)}=\frac{q[2]_q}{[3]_q} \)};
  \node at (2,-4.8) {\( \psi_1^{(2)}=\frac{1}{[3]_q} \)}; 
  \node at (0,-4.1) {
    \begin{ytableau}
    1 & 2& *(red!50) 3
    \end{ytableau}
  };
  \node[below] at (-2,-5.4) {
      \begin{ytableau}
        1 & 2 & 3 \\
        *(red!50) 4
      \end{ytableau}
    }; 
  \node[below] at (2,-5.4) {
    \begin{ytableau}
      1 & 2 & 3 & *(red!50) 4
    \end{ytableau}
    };
  \draw (-2.5,-6.7) -- (-2.5,-7.5);
  \draw (2,-6.1) -- (0.3,-7.5);
  \draw (2,-6.1) -- (3.7,-7.5);
  \node[below] at (-2.5,-7.5) {
      \begin{ytableau}
        1 & 2 & 3 \\
        4 & 5
      \end{ytableau}
    };
  \node at (-3.2,-7.1) {\( \psi_0^{(3)}=1 \)};
  \node at (-0.3,-7.1) {\( \psi_0^{(3)}=\frac{q[3]_q}{[4]_q} \)};
  \node at (4.3,-7.1) {\( \psi_1^{(3)}=\frac{1}{[4]_q} \)};
  \node[below] at (0.7,-7.5) {
      \begin{ytableau}
        1 & 2 & 3 & 4\\
        5
      \end{ytableau}
    };
  \node[below] at (4.2,-7.5) {
      \begin{ytableau}
        1 & 2 & 3 & 4 & 5
      \end{ytableau}
    };
    \node at (-2.5,-9.1) {\( S \)};
    \node at (0.6,-9.1) {\( R \)};
    \node at (4.25,-9.1) {\( T \)};
\end{tikzpicture}
  
\caption{The transition probabilities for \( \mm=(2,3,5,5,5) \). Here,
  \( p_\mm(S;q)=q[2]_q/[3]_q \), \( p_\mm(R;q)=q/[4]_q \), and
  \( p_\mm(T;q)=1/[3]_q[4]_q \). When computing \( \psi^{(r)}_k \),
  the cells containing an integer greater than \( r \) are colored
  red. }
\label{fig:exam_modified}
\end{figure}

\begin{remark}\label{rem:1}
  We note that our definition of \( \psi_k^{(r)}(T ; q) \) is a
  modification of Hikita's \( \varphi_k^{(r)}(T ; q) \), which is the
  same as \eqref{eq:15} except that the power \( \sum_{i=k+1}^l b_i \)
  of \( q \) is replaced by \( \sum_{i=1}^k a_i \).
  By making this modification, we obtain Theorem~\ref{thm:Hikita_m}, 
  a version of Hikita's formula without an adjustment by a \( q \)-factor.
\end{remark}

\begin{defn}\label{defn:psi_pT}
  Let \( \mm\in \HH_n \) and \( T\in \SYT(n) \). Let
  \( T'\in\SYT(n-1) \) be the tableau obtained from \( T \) by
  removing the entry \( n \). Let \( \mm'\in \HH_{n-1} \) be
  the Hessenberg function given by \( \mm'(i) = \mm(i+1)-1 \) for
  \( i\in [n-1] \) and let \( r= n-\mm(1) \). We define
  \( p_\mm(T;q) \) recursively by \( p_\emptyset(\emptyset;q)=1 \) for
  \( n=0 \) and
  \[
    p_\mm(T;q)=
    \begin{cases}
       \psi_{k}^{(r)}(T';q) p_{\mm'}(T';q)
      & \mbox{if \( T=f_{k}^{(r)}(T') \) for some \( 0\le k\le l \),}\\
      0 & \mbox{otherwise.}
    \end{cases}
  \]
\end{defn}

For example, see Figure~\ref{fig:exam_modified}.

\begin{thm}[{\cite[Theorem 3]{Hikita2024}}]
\label{thm:Hikita_m}
  Let \( \mm \) be a Hessenberg function of length \( n \). Then we have
  \begin{equation}\label{eq:7}
    X_{\mm}(x;q)=\sum_{\lambda \vdash n}\prod_{i=1}^{\ell(\lambda)} [\lambda_i]_q! \sum_{T\in\SYT(\lambda)}p_\mm(T;q) e_{\lambda}(x).  
  \end{equation}
\end{thm}

\section{Algorithm for restricted modular law}
\label{sec:restr-modular-laws}
In this section, we introduce the restricted modular law, which is a slight modification  
of the modular law of Abreu and Nigro~\cite{Abreu2021}. Using this, we provide an explicit  
algorithm for computing the value of a function \( f \) on \( \HH \) that satisfies the  
restricted modular law (\Cref{Algorithm}). Specifically, we show how to reduce the
computation of \( f(\mm) \) for an arbitrary \( \mm \in \HH \) to evaluating \( f \) at
disjoint unions of paths.  
Consequently, this algorithm implies that a function satisfying the restricted modular law  
is uniquely determined by its values on disjoint unions of paths (\Cref{thm: master}).

\begin{defn}\label{def:5}
  Let
  \( (\mm, \mm', \mm'') \in \HH_n^3 \). We say that \( (\mm, \mm', \mm'') \) is a
  \emph{modular triple of type~I} if there exists \(i\in [n-1]\)
  satisfying the following conditions:
  \begin{enumerate}
  \item \(\mm(i)+1=\mm'(i)=\mm''(i)-1\) and \(\mm'(i-1)<\mm'(i)<\mm'(i+1)\), where \(\mm'(0)=0\).
  \item \(\mm(j)=\mm'(j)=\mm''(j)\) for all \(j\in [n]\setminus \{i\}\).
  \item \(\mm'(\mm'(i))=\mm'(\mm'(i)+1)\).
  \end{enumerate}

  We say that \( (\mm, \mm', \mm'') \) is a
  \emph{modular triple of type~II} if there exists \(i\in [n-1]\)
  satisfying the following conditions:
  \begin{enumerate}
  \item \(\mm'(i)+1=\mm'(i+1)\), \( \mm(i) = \mm'(i) = \mm''(i)-1 \), and
    \( \mm(i+1)+1 = \mm'(i+1) = \mm''(i+1) \).
  \item \(\mm(j)=\mm'(j)=\mm''(j)\) for all \(j\in [n]\setminus \{i,i+1\}\).
  \item \( (\mm')^{-1}(\{i\})=\emptyset \).
  \end{enumerate}
  If, in addition, \( i \neq 1 \), then
  \( (\mm, \mm', \mm'') \) is called a \emph{restricted modular triple of
    type~II}. See Figures~\ref{fig:image6} and \ref{fig:image7}.

Let $A$ be a $\QQ(q)$-algebra. We say that a function
$f:\HH\rightarrow A$ satisfies the \emph{modular law} if
\[
    (1+q)f(\mm') = qf(\mm) + f(\mm'')
\]
whenever the triple $(\mm,\mm',\mm'')$ is either a modular triple of
type I or a modular triple of type II. We say a function
$f:\HH\rightarrow A$ satisfies the \emph{restricted modular
  law} if
\begin{equation}\label{eq:6}
    (1+q)f(\mm') = qf(\mm) + f(\mm'')
\end{equation}
whenever the triple $(\mm,\mm',\mm'')$ is either a modular triple of type I or a restricted modular triple of type II. 
\end{defn}

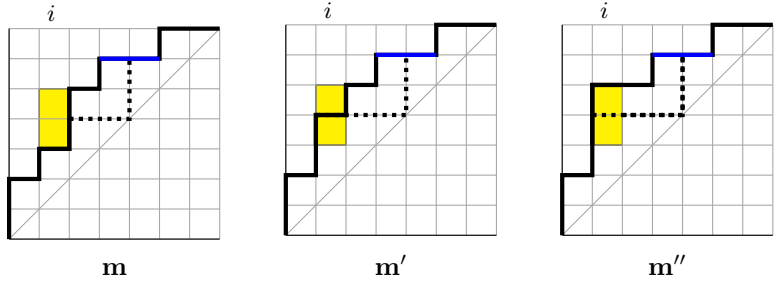
\begin{figure}
  \centering

 \begin{tikzpicture}[scale=0.4]
  \draw [fill=yellow] (1,3) rectangle (2,5);

  \foreach \i in {0,1,...,7}
  \draw[color=gray!70] (\i,0) -- (\i,7);

  \foreach \j in {0,1,...,7}
  \draw[color=gray!70]  (0,\j) -- (7,\j);

  \draw (0,0) -- (7,0);
  \draw (0,0) -- (0,7);
  \draw[color=gray!70] (0,0) -- (7,7);

  \draw[ultra thick]
  (0,0) -- (0,2) -- (1,2) -- (1,3) -- (2,3) -- (2,5) -- (3,5) -- (3,6) -- (5,6) -- (5,7) -- (7,7);
  \draw[ultra thick, dotted]
  (2,4) -- (4,4) -- (4,6);
   \draw[ultra thick]
  (1,3) -- (2,3) -- (2,5);
  \draw[ultra thick, color=blue]
  (3,6) -- (5,6);
  
  \node at (1.4, 7.5) {\small\( i \)};
  \node at (3.5,-1) {\( \mm \)};

\end{tikzpicture}
\qquad
\begin{tikzpicture}[scale=0.4]
  \draw [fill=yellow] (1,3) rectangle (2,5);

  \foreach \i in {0,1,...,7}
  \draw[color=gray!70] (\i,0) -- (\i,7);

  \foreach \j in {0,1,...,7}
  \draw[color=gray!70]  (0,\j) -- (7,\j);

  \draw (0,0) -- (7,0);
  \draw (0,0) -- (0,7);
  \draw[color=gray!70] (0,0) -- (7,7);

  \draw[ultra thick]
  (0,0) -- (0,2) -- (1,2) -- (1,4) -- (2,4) -- (2,5) -- (3,5) -- (3,6) -- (5,6) -- (5,7) -- (7,7);
  \draw[ultra thick, dotted]
  (1,4) -- (4,4) -- (4,6);
  \draw[ultra thick]
  (1,3) -- (1,4) -- (2,4) -- (2,5);
  \draw[ultra thick, color=blue]
  (3,6) -- (5,6);
  
  \node at (1.4, 7.5) {\small\( i \)};
  \node at (3.5,-1) {\( \mm' \)};
\end{tikzpicture}
\qquad
\begin{tikzpicture}[scale=0.4]
  \draw [fill=yellow] (1,3) rectangle (2,5);

  \foreach \i in {0,1,...,7}
  \draw[color=gray!70] (\i,0) -- (\i,7);

  \foreach \j in {0,1,...,7}
  \draw[color=gray!70]  (0,\j) -- (7,\j);

  \draw (0,0) -- (7,0);
  \draw (0,0) -- (0,7);
  \draw[color=gray!70] (0,0) -- (7,7);

  \draw[ultra thick]
  (0,0) -- (0,2) -- (1,2) -- (1,5) -- (3,5) -- (3,6) -- (5,6) -- (5,7) -- (7,7);
  \draw[ultra thick, dotted]
  (1,4) -- (4,4) -- (4,6);
  \draw[ultra thick, dotted]
  (2,4) -- (4,4) -- (4,6);
   \draw[ultra thick]
  (1,3) -- (1,5) -- (2,5);
  \draw[ultra thick, color=blue]
  (3,6) -- (5,6);
  
  \node at (1.4, 7.5) {\small\( i \)};
  \node at (3.5,-1) {\( \mm'' \)};
\end{tikzpicture}
 
\caption{A modular triple \( (\mm,\mm',\mm'') \) of type I. Note that
  \( \mm \), \( \mm' \), and \( \mm'' \) differ only in the
  highlighted region. }\label{fig:image6}
\end{figure} 

\begin{figure}
  \centering

 \begin{tikzpicture}[scale=0.4]
  \draw [fill=yellow] (2,5) rectangle (4,6);

  \foreach \i in {0,1,...,7}
  \draw[color=gray!70] (\i,0) -- (\i,7);

  \foreach \j in {0,1,...,7}
  \draw[color=gray!70]  (0,\j) -- (7,\j);

  \draw (0,0) -- (7,0);
  \draw (0,0) -- (0,7);
  \draw[color=gray!70] (0,0) -- (7,7);

  \draw[ultra thick]
  (0,0) -- (0,2) -- (1,2) -- (1,4) -- (2,4) -- (2,5) -- (4,5) -- (4,6) -- (5,6) -- (5,7) -- (7,7);
  \draw[ultra thick, dotted]
  (1,3) -- (3,3) -- (3,5);
  \draw[ultra thick]
  (2,5) -- (4,5) -- (4,6);
  \draw[ultra thick, color=blue]
  (1,2) -- (1,4);
  
  \node at (2.4, 7.5) {\small\( i \)};
  \node at (3.6, 7.5) {\small\( i+1 \)};
  \node at (3.5,-1) {\( \mm \)};
\end{tikzpicture}
\qquad
\begin{tikzpicture}[scale=0.4]
  \draw [fill=yellow] (2,5) rectangle (4,6);

  \foreach \i in {0,1,...,7}
  \draw[color=gray!70] (\i,0) -- (\i,7);

  \foreach \j in {0,1,...,7}
  \draw[color=gray!70]  (0,\j) -- (7,\j);

  \draw (0,0) -- (7,0);
  \draw (0,0) -- (0,7);
  \draw[color=gray!70] (0,0) -- (7,7);

  \draw[ultra thick]
  (0,0) -- (0,2) -- (1,2) -- (1,4) -- (2,4) -- (2,5) -- (3,5) -- (3,6) -- (5,6) -- (5,7) -- (7,7);
  \draw[ultra thick, dotted]
  (1,3) -- (3,3) -- (3,6);
  \draw[ultra thick]
  (2,5) -- (3,5) -- (3,6) -- (4,6);
  \draw[ultra thick, color=blue]
  (1,2) -- (1,4);
  
  \node at (2.4, 7.5) {\small\( i \)};
  \node at (3.6, 7.5) {\small\( i+1 \)};
  \node at (3.5,-1) {\( \mm' \)};
\end{tikzpicture}
\qquad
\begin{tikzpicture}[scale=0.4]
  \draw [fill=yellow] (2,5) rectangle (4,6);

  \foreach \i in {0,1,...,7}
  \draw[color=gray!70] (\i,0) -- (\i,7);

  \foreach \j in {0,1,...,7}
  \draw[color=gray!70]  (0,\j) -- (7,\j);

  \draw (0,0) -- (7,0);
  \draw (0,0) -- (0,7);
  \draw[color=gray!70] (0,0) -- (7,7);

  \draw[ultra thick]
  (0,0) -- (0,2) -- (1,2) -- (1,4) -- (2,4) -- (2,6) -- (5,6) -- (5,7) -- (7,7);
  \draw[ultra thick, dotted]
  (1,3) -- (3,3) -- (3,6);
  \draw[ultra thick]
  (2,5) -- (2,6) -- (4,6);
  \draw[ultra thick, color=blue]
  (1,2) -- (1,4);
  
  \node at (2.4, 7.5) {\small\( i \)};
  \node at (3.6, 7.5) {\small\( i+1 \)};
  \node at (3.5,-1) {\( \mm'' \)};
\end{tikzpicture}
\caption{A restricted modular triple \( (\mm,\mm',\mm'') \) of type
  II. Note that \( \mm \), \( \mm' \), and \( \mm'' \) differ only in
  the highlighted region.}\label{fig:image7}
\end{figure}
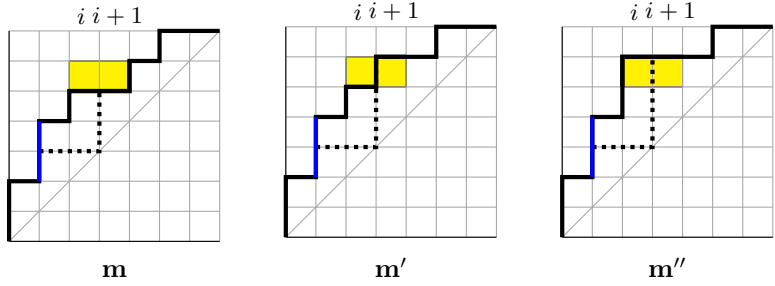 

Abreu and Nigro~\cite{Abreu2021} showed that a function satisfying the
modular law is uniquely determined by its values on disjoint unions of
complete graphs. Note that the restricted modular law imposes fewer
constraints than the modular law. Nevertheless, we show that the
restricted modular law still suffices to characterize a given function
using only its values on a small family of graphs (see
\Cref{Algorithm} and \Cref{thm: master}).

\begin{defn}
  Let \( \mm \in \HH_n \) be a Hessenberg function that is not a union
  of paths. Let \( \alpha \) be the largest integer in \( [n-1] \) such
  that \( \mm^{-1}(\{\alpha\}) = \emptyset \) and
   \( \mm^{-1}(\{\alpha +1\}) \neq \emptyset \) with
  \( \min \left( \mm^{-1}(\{\alpha+1\}) \right) < \alpha \). We say that
  \( \mm\in \HH_n \) is \emph{flat} (respectively \emph{non-flat}) if
  \( \mm(\alpha)=\mm(\alpha+1) \) (respectively \( \mm(\alpha)<\mm(\alpha+1) \)). We
  denote \( \alpha(\mm)=\alpha \) and
  \( \beta(\mm) = \min \left( \mm^{-1}(\{\alpha+1\}) \right) \).
\end{defn}

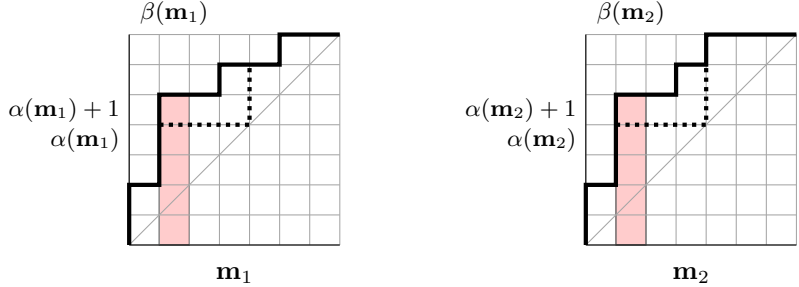
\begin{figure}
  \centering
 \begin{tikzpicture}[scale=0.4]
  \draw [fill=red!20] (1,0) rectangle (2,5);

  \foreach \i in {0,1,...,7}
  \draw[color=gray!70] (\i,0) -- (\i,7);

  \foreach \j in {0,1,...,7}
  \draw[color=gray!70]  (0,\j) -- (7,\j);

  \draw (0,0) -- (7,0);
  \draw (0,0) -- (0,7);
  \draw[color=gray!70] (0,0) -- (7,7);

  \draw[ultra thick]
  (0,0) -- (0,2) -- (1,2) -- (1,5) -- (3,5) -- (3,6) -- (5,6) -- (5,7) -- (7,7);
  \draw[ultra thick, dotted]
  (1,4) -- (4,4) -- (4,6);
  
  \node[left] at (0, 4.5) {\small\( \alpha(\mm_1)+1 \)};
  \node[left] at (0, 3.5) {\small\( \alpha(\mm_1) \)};
  \node[above] at (1.5, 7) {\small\( \beta(\mm_1) \)};
  \node at (3.5,-1) {\( \mm_1 \)};
\end{tikzpicture}
\qquad\qquad
 \begin{tikzpicture}[scale=0.4]
  \draw [fill=red!20] (1,0) rectangle (2,5);

  \foreach \i in {0,1,...,7}
  \draw[color=gray!70] (\i,0) -- (\i,7);

  \foreach \j in {0,1,...,7}
  \draw[color=gray!70]  (0,\j) -- (7,\j);

  \draw (0,0) -- (7,0);
  \draw (0,0) -- (0,7);
  \draw[color=gray!70] (0,0) -- (7,7);

  \draw[ultra thick]
  (0,0) -- (0,2) -- (1,2) -- (1,5) -- (3,5) -- (3,6) -- (4,6) -- (4,7) -- (7,7);
  \draw[ultra thick, dotted]
  (1,4) -- (4,4) -- (4,6);
  
  \node[left] at (0, 4.5) {\small\( \alpha(\mm_2)+1 \)};
  \node[left] at (0, 3.5) {\small\( \alpha(\mm_2) \)};
  \node[above] at (1.5, 7) {\small\( \beta(\mm_2) \)};
  \node at (3.5,-1) {\( \mm_2 \)};
\end{tikzpicture}
\caption{The Hessenberg function \( \mm_1 \) is flat,
  while \( \mm_2 \) is non-flat. The highlighted regions
  indicate the \( \beta(\mm_1) \)-th and \( \beta(\mm_2) \)-th
  columns of the grids.}
\label{fig:flat}
\end{figure} 

See \Cref{fig:flat}. Note that in the above definition, the assumption
that \( \mm \) is not a union of paths guarantees the existence of
\( \alpha(\mm) \). Moreover, the condition
\( \min \left( \mm^{-1}(\{\alpha+1\}) \right) < \alpha \) implies that
\( \alpha(\mm)>1 \). When we say \( \mm \) is flat or non-flat, we always
make this assumption.

\begin{lem}\label{lem: modular law for level}
  Suppose that \( \mm \in \HH_n \) is flat. Let \( \mm_0\in \HH_n \)
  and \( \mm_1\in \HH_n \) be given by
\[
    \mm_0(i) = \begin{cases}
        \mm(i) & \text{if } i \neq \beta(\mm), \\
        \mm(i) - 2 & \text{if } i = \beta(\mm),
    \end{cases}
    \qquad \text{and} \qquad
    \mm_1(i) = \begin{cases}
        \mm(i) & \text{if } i \neq \beta(\mm), \\
        \mm(i) - 1 & \text{if } i = \beta(\mm).
    \end{cases}
\]
See Figure~\ref{fig:image9}. Then, for any function
\( f : \HH \to A \) that satisfies the restricted modular law, we have
\begin{equation} \label{eq: modular law for level}
    f(\mm) = (1 + q) f(\mm_1) - q f(\mm_0).
\end{equation}
\end{lem}
\begin{proof}
  By construction, \((\mm_0,\mm_1,\mm)\) is a modular triple of type
  I. Thus, \eqref{eq: modular law for level} holds.
\end{proof}

\begin{figure}
  \centering
 \begin{tikzpicture}[scale=0.35]
  \draw [fill=yellow] (1,3) rectangle (2,5);

  \foreach \i in {0,1,...,7}
  \draw[color=gray!70] (\i,0) -- (\i,7);

  \foreach \j in {0,1,...,7}
  \draw[color=gray!70]  (0,\j) -- (7,\j);

  \draw (0,0) -- (7,0);
  \draw (0,0) -- (0,7);
  \draw[color=gray!70] (0,0) -- (7,7);

  \draw[ultra thick]
  (0,0) -- (0,2) -- (1,2) -- (1,3) -- (2,3) -- (2,5) -- (3,5) -- (3,6) -- (5,6) -- (5,7) -- (7,7);
  \draw[ultra thick, dotted]
  (2,4) -- (4,4) -- (4,6);
  
  \node[left] at (0, 4.5) {\small\( \alpha+1 \)};
  \node[left] at (0, 3.5) {\small\( \alpha \)};
  \node at (3.5,-1) {\( \mm_0 \)};
\end{tikzpicture}
\qquad
 \begin{tikzpicture}[scale=0.35]
  \draw [fill=yellow] (1,3) rectangle (2,5);

  \foreach \i in {0,1,...,7}
  \draw[color=gray!70] (\i,0) -- (\i,7);

  \foreach \j in {0,1,...,7}
  \draw[color=gray!70]  (0,\j) -- (7,\j);

  \draw (0,0) -- (7,0);
  \draw (0,0) -- (0,7);
  \draw[color=gray!70] (0,0) -- (7,7);

  \draw[ultra thick]
  (0,0) -- (0,2) -- (1,2) -- (1,4) -- (2,4) -- (2,5) -- (3,5) -- (3,6) -- (5,6) -- (5,7) -- (7,7);
  \draw[ultra thick, dotted]
  (1,4) -- (4,4) -- (4,6);
  
  \node[left] at (0, 4.5) {\small\( \alpha+1 \)};
  \node[left] at (0, 3.5) {\small\( \alpha \)};
  \node at (3.5,-1) {\( \mm_1 \)};
\end{tikzpicture}
\qquad
\begin{tikzpicture}[scale=0.35]
  \draw [fill=yellow] (1,3) rectangle (2,5);

  \foreach \i in {0,1,...,7}
  \draw[color=gray!70] (\i,0) -- (\i,7);

  \foreach \j in {0,1,...,7}
  \draw[color=gray!70]  (0,\j) -- (7,\j);

  \draw (0,0) -- (7,0);
  \draw (0,0) -- (0,7);
  \draw[color=gray!70] (0,0) -- (7,7);

  \draw[ultra thick]
  (0,0) -- (0,2) -- (1,2) -- (1,5) -- (3,5) -- (3,6) -- (5,6) -- (5,7) -- (7,7);
  \draw[ultra thick, dotted]
  (1,4) -- (4,4) -- (4,6);
  
  \node[left] at (0, 4.5) {\small\( \alpha+1 \)};
  \node[left] at (0, 3.5) {\small\( \alpha \)};
  \node at (3.5,-1) {\( \mm \)};
\end{tikzpicture}
\caption{The triple \( (\mm_0,\mm_1,\mm) \) in \Cref{lem: modular law for level},
where \( \alpha=\alpha(\mm) \).}
\label{fig:image9}
\end{figure}
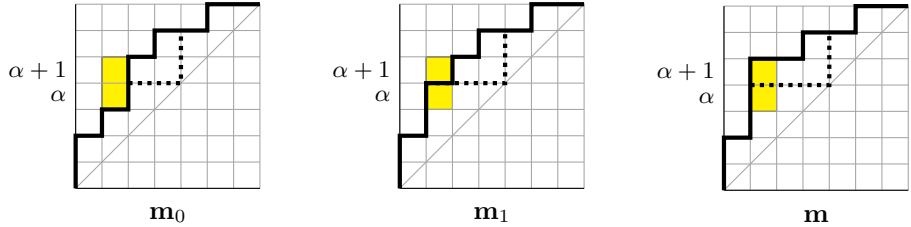 

\begin{figure}
  \centering
  \begin{tikzpicture}[scale=0.35]
  \draw [fill=yellow] (3,5) rectangle (5,6);

  \foreach \i in {0,1,...,7}
  \draw[color=gray!70] (\i,0) -- (\i,7);

  \foreach \j in {0,1,...,7}
  \draw[color=gray!70]  (0,\j) -- (7,\j);

  \draw (0,0) -- (7,0);
  \draw (0,0) -- (0,7);
  \draw[color=gray!70] (0,0) -- (7,7);

  \draw[ultra thick]
  (0,0) -- (0,2) -- (1,2) -- (1,5) -- (3,5) -- (5,5) -- (5,6) -- (5,7) -- (7,7);
  \draw[ultra thick, dotted]
  (1,4) -- (4,4) -- (4,5);
  
  \node[left] at (0, 4.5) {\small\( \alpha+1 \)};
  \node[left] at (0, 3.5) {\small\( \alpha \)};
  \node at (3.5,-1) {\( \mm_0 \)};
\end{tikzpicture}
 \qquad
\begin{tikzpicture}[scale=0.35]
  \draw [fill=yellow] (3,5) rectangle (5,6);

  \foreach \i in {0,1,...,7}
  \draw[color=gray!70] (\i,0) -- (\i,7);

  \foreach \j in {0,1,...,7}
  \draw[color=gray!70]  (0,\j) -- (7,\j);

  \draw (0,0) -- (7,0);
  \draw (0,0) -- (0,7);
  \draw[color=gray!70] (0,0) -- (7,7);

  \draw[ultra thick]
  (0,0) -- (0,2) -- (1,2) -- (1,5) -- (3,5) -- (4,5) -- (4,6) -- (5,6) -- (5,7) -- (7,7);
  \draw[ultra thick, dotted]
  (1,4) -- (4,4) -- (4,5);
  
  \node[left] at (0, 4.5) {\small\( \alpha+1 \)};
  \node[left] at (0, 3.5) {\small\( \alpha \)};
  \node at (3.5,-1) {\( \mm \)};
\end{tikzpicture}
 \qquad
\begin{tikzpicture}[scale=0.35]
  \draw [fill=yellow] (3,5) rectangle (5,6);

  \foreach \i in {0,1,...,7}
  \draw[color=gray!70] (\i,0) -- (\i,7);

  \foreach \j in {0,1,...,7}
  \draw[color=gray!70]  (0,\j) -- (7,\j);

  \draw (0,0) -- (7,0);
  \draw (0,0) -- (0,7);
  \draw[color=gray!70] (0,0) -- (7,7);

  \draw[ultra thick]
  (0,0) -- (0,2) -- (1,2) -- (1,5) -- (3,5) -- (3,6) -- (5,6) -- (5,7) -- (7,7);
  \draw[ultra thick, dotted]
  (1,4) -- (4,4) -- (4,6);
  
  \node[left] at (0, 4.5) {\small\( \alpha+1 \)};
  \node[left] at (0, 3.5) {\small\( \alpha \)};
  \node at (3.5,-1) {\( \mm_2 \)};
\end{tikzpicture}
\caption{The triple \( (\mm_0,\mm_1,\mm) \) in \Cref{lem: modular law for level},
where \( \alpha=\alpha(\mm) \).}
\label{fig:image10}
\end{figure}
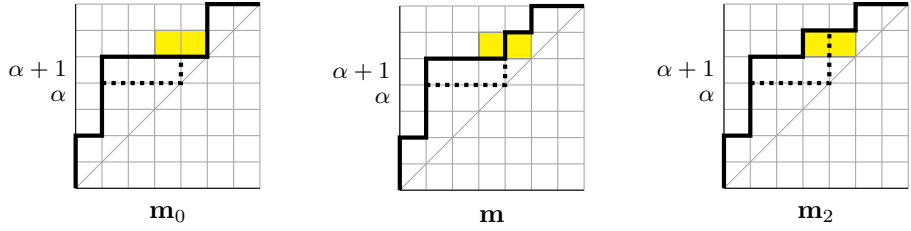

\begin{lem}\label{lem: application of modular law}
  Suppose that \( \mm \in \HH_n \) is non-flat. Let \( \mm_0\in \HH_n \)
  and \( \mm_2\in \HH_n \) be given by
    \[
        \mm_0(i) = \begin{cases}
            \mm(i) & \mbox{if \(i \neq \alpha(\mm)+1 \)},\\
            \mm(i)-1 & \mbox{if \(i = \alpha(\mm)+1\)},
        \end{cases} \qquad \text{and} \qquad \mm_2(i) = \begin{cases}
            \mm(i) & \mbox{if \(i \neq \alpha(\mm)\)},\\
            \mm(i)+1 & \mbox{if \(i = \alpha(\mm)\)}.
        \end{cases}
    \]
    See \Cref{fig:image10}. Furthermore, let \(\mm^{(1)}_0\) and \(\mm^{(1)}\) be given by
    \[
    \mm^{(1)}_0(i)=\begin{cases}
    \mm_2(i) & \mbox{if \(i \neq \beta(\mm)\)},\\
    \mm_2(i) - 2 & \mbox{if \(i = \beta(\mm)\)},
    \end{cases} \qquad \text{and} \qquad
    \mm^{(1)}(i)=\begin{cases}
    \mm_2(i) & \mbox{if \(i \neq \beta(\mm)\)},\\
    \mm_2(i) - 1 & \mbox{if \(i = \beta(\mm)\)}.
    \end{cases}
    \]
    See \Cref{fig:image12}. Then, for a function \(f:\HH\to A\) that
    satisfies the restricted modular law, we have
    \begin{equation}\label{eq: application of modular law}
    (1+q)f(\mm) = (1+q)f(\mm^{(1)}) + q f(\mm_0) - q f(\mm^{(1)}_0).    
    \end{equation}
\end{lem}

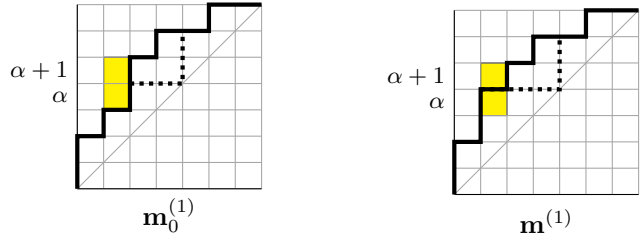
\begin{figure}
  \centering

  \begin{tikzpicture}[scale=0.35]
  \draw [fill=yellow] (1,3) rectangle (2,5);

  \foreach \i in {0,1,...,7}
  \draw[color=gray!70] (\i,0) -- (\i,7);

  \foreach \j in {0,1,...,7}
  \draw[color=gray!70]  (0,\j) -- (7,\j);

  \draw (0,0) -- (7,0);
  \draw (0,0) -- (0,7);
  \draw[color=gray!70] (0,0) -- (7,7);

  \draw[ultra thick]
  (0,0) -- (0,2) -- (1,2) -- (1,3) --(2,3) -- (2,5) -- (3,5) -- (3,6) -- (5,6) -- (5,7) -- (7,7);
  \draw[ultra thick, dotted]
  (2,4) -- (4,4) -- (4,6);
  
  \node[left] at (0, 4.5) {\small\( \alpha+1 \)};
  \node[left] at (0, 3.5) {\small\( \alpha \)};
  \node at (3.5,-1) {\( \mm_0^{(1)} \)};
\end{tikzpicture}
\qquad \qquad 
\begin{tikzpicture}[scale=0.35]
  \draw [fill=yellow] (1,3) rectangle (2,5);

  \foreach \i in {0,1,...,7}
  \draw[color=gray!70] (\i,0) -- (\i,7);

  \foreach \j in {0,1,...,7}
  \draw[color=gray!70]  (0,\j) -- (7,\j);

  \draw (0,0) -- (7,0);
  \draw (0,0) -- (0,7);
  \draw[color=gray!70] (0,0) -- (7,7);

  \draw[ultra thick]
  (0,0) -- (0,2) -- (1,2) -- (1,4) --(2,4) -- (2,5) -- (3,5) -- (3,6) -- (5,6) -- (5,7) -- (7,7);
  \draw[ultra thick, dotted]
  (1,4) -- (4,4) -- (4,6);
  
  \node[left] at (0, 4.5) {\small\( \alpha+1 \)};
  \node[left] at (0, 3.5) {\small\( \alpha \)};
  \node at (3.5,-1) {\( \mm^{(1)} \)};
\end{tikzpicture}
\caption{The Hessenberg functions $\mm^{(1)}_0$ and $\mm^{(1)}$ in
  \Cref{lem: application of modular law}, where \( \alpha=\alpha(\mm) \).}
\label{fig:image12}
\end{figure} 

\begin{proof}
  We must have \( \mm(\alpha(\mm)+1)=\mm(\alpha(\mm))+1 \), because otherwise it
  would violate the maximality of \( \alpha(\mm) \). Hence,
  \( (\mm_0,\mm,\mm_2) \) is a modular triple of type~II. Moreover,
  since \( \alpha(\mm)>1 \), it is a restricted modular triple of type~II.
  Therefore, by \eqref{eq:6}, we have
\begin{equation}\label{eq: 61}
    (1+q)f(\mm)=qf(\mm_0) + f(\mm_2).
\end{equation}
Again, by construction, \((\mm^{(1)}_0,\mm^{(1)},\mm_2)\) is a modular
triple of type~I. Then, by \eqref{eq:6}, we have
\begin{equation}\label{eq: 62}
    (1+q)f(\mm^{(1)})=qf(\mm^{(1)}_0) + f(\mm_2).
\end{equation}
Combining \eqref{eq: 61} and \eqref{eq: 62} concludes the proof.
\end{proof}

\begin{algo}\label{Algorithm}
  Suppose that \( f:\HH\rightarrow A \) is a function that satisfies
  the restricted modular law. For \( \mm\in\HH_n \), we define the
  following algorithm for representing \( f(\mm) \) as a linear
  combination of \(f(\pp_{n_1}+\cdots+\pp_{n_d})\)'s.
\begin{description} 
\item[Case 1]\label{algorithm step 1} If
  \(\mm=\pp_{n_1}+\cdots+\pp_{n_d}\) for some positive integers
  \(n_1,\dots,n_d\), then there is nothing to do, and the algorithm
  terminates.
\item[Case 2]\label{algorithm step 2} If \(\mm\) is flat, find
  \(\mm_0\) and \(\mm_1\) as in \Cref{lem: modular law for level}. Use
  \eqref{eq: modular law for level} to express \( f(\mm) \) as a
  linear combination of \( f(\mm_0) \) and \( f(\mm_1) \). Repeat the
  algorithm for each of \(\mm_0\) and \(\mm_1\).
\item[Case 3]\label{algorithm step 3} If \(\mm\) is non-flat, find
  \(\mm_0\), \( \mm^{(1)}_0 \), and \(\mm^{(1)}\) as in \Cref{lem:
    application of modular law}. Use \eqref{eq: application of modular
    law} to express \( f(\mm) \) as a linear combination of
  \( f(\mm_0) \), \( f(\mm^{(1)}_0) \), and \( f(\mm^{(1)}) \). Repeat
  the algorithm for each of \(\mm_0\), \( \mm^{(1)}_0 \), and
  \(\mm^{(1)}\).
\end{description}
\end{algo}

\begin{figure}
  \centering

\begin{tikzpicture}[scale=0.3] 
  \board{1}{12}
  \draw[line width=1pt, ->] (2,11.7) -- (-1,11.3);
  \draw[line width=1pt, ->](5,11.7) -- (15,5.5);
  \draw[ultra thick]
  (1,12) -- (1,15) -- (2,15) -- (2,16) -- (4,16) -- (4,17) -- (6,17);
  \board{-13}6
  \draw[line width=1pt, ->] (-12,5.7) -- (-13,5.3);
  \draw[line width=1pt, ->](-11,5.7) -- (-11,-0.7);
  \draw[ultra thick]
  (-13,6) -- (-13,8) -- (-12,8) -- (-12,9) -- (-11,9) -- (-11,11) -- (-8,11);
  \board{-6}6
  \draw[line width=1pt,->, color=red] (-6.5,8.5) -- (-7.5,8.5);
  \draw[line width=1pt,->, color=red] (-5,5.7) -- (-6,5.3);
  \draw[line width=1pt,->, color=red] (-2,5.7) -- (-1,5.3);
  \draw[ultra thick]
  (-6,6) -- (-6,8) -- (-5,8) -- (-5,10) -- (-3,10) -- (-3,11) -- (-1,11);
  \board{-17}0
  \draw[ultra thick]
  (-17,0) -- (-17,2) -- (-16,2) -- (-16,3) -- (-15,3) -- (-15,4) -- (-14,4)
  -- (-14,5) -- (-12,5);
  \board{-10}0
  \draw[line width=1pt, ->] (-6,-0.3) -- (-5,-0.7);
  \draw[line width=1pt, ->](-10,-0.3) -- (-10,-6.7);
  \draw[ultra thick]
  (-10,0) -- (-10,2) -- (-8,2) -- (-8,5) -- (-5,5);
  \board{-3}0
  \draw[line width=1pt, ->] (1,-0.3) -- (2,-0.7);
  \draw[line width=1pt, ->](-3,-0.3) -- (-3,-6.7);
  \draw[ultra thick]
  (-3,0) -- (-3,2) -- (-2,2) -- (-2,4) -- (1,4) -- (1,5) -- (2,5);
  \board{8}0
  \draw[line width=1pt, ->] (9,-0.3) -- (8,-0.7);
  \draw[line width=1pt, ->](11,-0.3) -- (11,-6.7);
  \draw[ultra thick]
  (8,0) -- (8,1) -- (9,1) -- (9,3) -- (10,3) -- (10,5) -- (13,5);
  \board{15}0
  \draw[line width=1pt,->, color=red] (14.5,2.5) -- (13.5,2.5);
  \draw[line width=1pt,->, color=red] (16,-0.3) -- (15,-0.7);
  \draw[line width=1pt,->, color=red] (19,-0.3) -- (20,-0.7);
  \draw[ultra thick]
  (15,0) -- (15,1) -- (16,1) -- (16,4) -- (18,4) -- (18,5) -- (20,5);
  \board{-16}{-6}
  \draw[ultra thick]
  (-16,-6) -- (-16,-4) -- (-15,-4) -- (-15,-3) -- (-13,-3) -- (-13,-1) -- (-11,-1);
  \board{-9}{-6}
  \draw[ultra thick]
  (-9,-6) -- (-9,-4) -- (-7,-4) -- (-7,-2) -- (-6,-2) -- (-6,-1) -- (-4,-1);
  \board{-2}{-6}
  \draw[ultra thick]
  (-2,-6) -- (-2,-4) -- (-1,-4) -- (-1,-3) -- (0,-3) -- (0,-2) -- (2,-2) --
  (2,-1) -- (3,-1);
  \board{5}{-6}
  \draw[ultra thick]
  (5,-6) -- (5,-5) -- (6,-5) -- (6,-3) -- (7,-3) -- (7,-2) -- (8,-2) --
  (8,-1) -- (10,-1);
  \board{12}{-6}
  \draw[line width=1pt, ->](12,-6.3) -- (12,-12.7);
  \draw[line width=1pt, ->] (16,-6.3) -- (17,-6.7);
  \draw[ultra thick]
  (12,-6) -- (12,-5) -- (13,-5) -- (13,-4) -- (14,-4) -- (14,-1) -- (17,-1);
  \board{19}{-6}
  \draw[line width=1pt, ->](19,-6.3) -- (19,-12.7);
  \draw[line width=1pt, ->] (23,-6.3) -- (24,-6.7);
  \draw[ultra thick]
  (19,-6) -- (19,-5) -- (20,-5) -- (20,-2) -- (23,-2) -- (23,-1) -- (24,-1);
  \board{-13}{-12}
  \draw[ultra thick]
  (-13,-12) -- (-13,-10) -- (-11,-10) -- (-11,-9) -- (-10,-9) -- (-10,-7) -- (-8,-7);
  \board{-6}{-12}
  \draw[ultra thick]
  (-6,-12) -- (-6,-10) -- (-4,-10) -- (-4,-8) -- (-2,-8) -- (-2,-7) -- (-1,-7);
  \board{6}{-12}
  \draw[ultra thick]
  (6,-12) -- (6,-11) -- (7,-11) -- (7,-9) -- (9,-9) -- (9,-7) -- (11,-7);
  \board{13}{-12}
  \draw[ultra thick]
  (13,-12) -- (13,-11) -- (14,-11) -- (14,-10) -- (15,-10) -- (15,-8) -- (16,-8)
  -- (16,-7) -- (18,-7);
  \board{20}{-12}
  \draw[ultra thick]
  (20,-12) -- (20,-11) -- (21,-11) -- (21,-9) -- (22,-9) -- (22,-8) -- (24,-8)
  -- (24,-7) -- (25,-7);
  \board{8}{-18}
  \draw[ultra thick]
  (8,-18) -- (8,-17) -- (9,-17) -- (9,-16) -- (10,-16) -- (10,-15) -- (11,-15)
  -- (11,-13) -- (13,-13);
  \board{15}{-18}
  \draw[ultra thick]
  (15,-18) -- (15,-17) -- (16,-17) -- (16,-16) -- (17,-16) -- (17,-14) -- (19,-14)
  -- (19,-13) -- (20,-13);
\end{tikzpicture}
\qquad

\caption{An illustration of the process in \Cref{Algorithm} for
  \( \mm=(3,4,4,5,5)\in \HH_5 \). Here, each \( \mm_i \) can be
  written as a linear combination of those to which it points. The red
  arrows indicate that Case 3 of \Cref{Algorithm} is being
  used.}\label{fig:tree}
\end{figure}
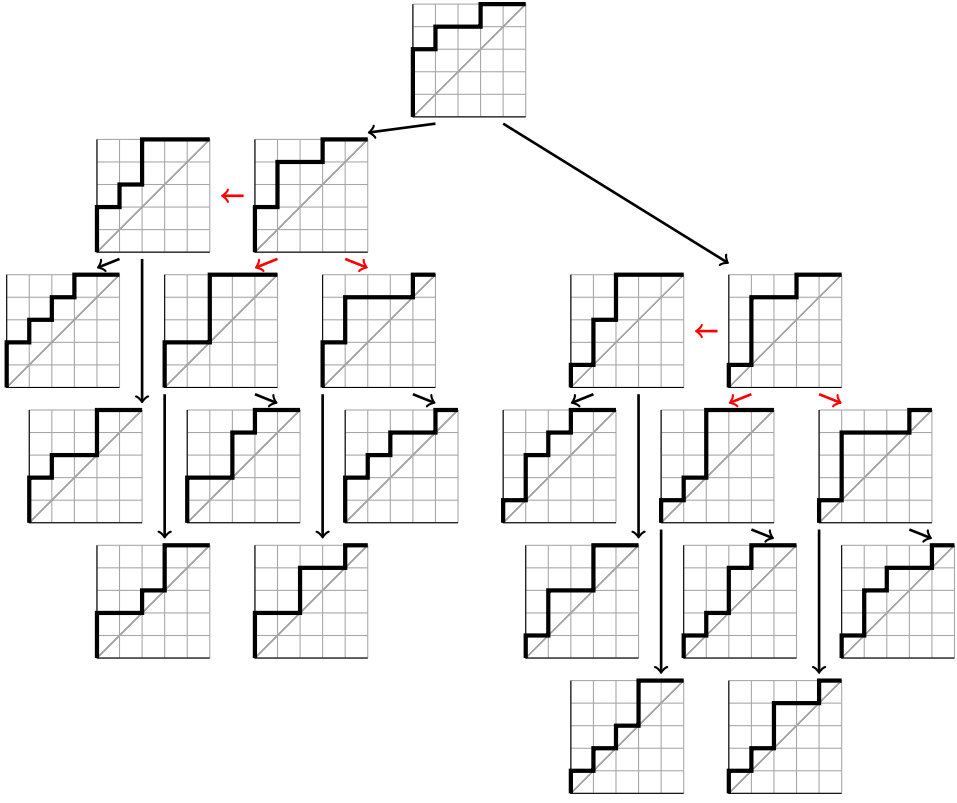 

See \Cref{fig:tree} for an example of \Cref{Algorithm}. 

\begin{lem}\label{lem: algorithm terminates}
  \Cref{Algorithm} always terminates.
\end{lem}

\begin{proof}
  We use strong induction on the area of \( \mm \) to prove the claim
  that \Cref{Algorithm} terminates. Let \( a\ge0 \) and suppose that
  the claim holds for any \( \mm \in \HH_n\) with \( \area(\mm) <a \).
  Let \( \mm \in \HH_n\) with \( \area(\mm) = a \).

  If \(\mm=\pp_{n_1}+\cdots+\pp_{n_d}\) for some positive integers
  \(n_1,\dots,n_d\), then it falls into Case 1 of the algorithm, which
  then terminates. Otherwise, \( \mm \) is flat or non-flat.

  If \( \mm \) is flat, then, by Case 2 of the algorithm, we can
  express \( f(\mm) \) as a linear combination of \( f(\mm_0) \) and
  \( f(\mm_1) \), where \( \area(\mm_0), \area(\mm_1)<a \). Thus, the
  claim holds by the inductive hypothesis.

  Now assume that \( \mm \) is non-flat. By Case 3 of the algorithm,
  we can express \( f(\mm) \) as a linear combination of
  \( f(\mm_0) \), \( f(\mm^{(1)}_0) \), and \( f(\mm^{(1)}) \), where
  \( \area(\mm_0), \area(\mm^{(1)}_0)<a \) and
  \( \area(\mm^{(1)})= a\). Hence, by the inductive hypothesis, it
  suffices to prove the claim for \( \mm^{(1)} \).

  We continue the process of \Cref{Algorithm} for \( \mm^{(1)} \). If
  it falls into Case 1 or Case 2, then we are done by the above
  arguments. Therefore, we assume that \( \mm^{(1)} \) goes to Case 3.
  This produces another \( \mm^{(2)} \), for which it suffices to
  prove the claim. Continuing in this way, we can define a sequence
  \( \mm^{(1)},\mm^{(2)},\dots \), assuming that each
  \( \mm^{(\ell)} \) falls into Case 3.

  In this process, each \( \mm^{(\ell+1)} \) is obtained from
  \( \mm^{(\ell)} \) by moving a square from column \( i \) to column
  \( j \), where \( i < j \), as illustrated in \Cref{fig:image13}.
  Therefore, the sequence \( \mm^{(1)},\mm^{(2)},\dots \) eventually
  terminates, as at each step a square is moved strictly to the right.
  We conclude that there exists \( r \) such that \( \mm^{(r)} \)
  belongs to Case 1 or Case 2. Then the algorithm terminates,
  completing the proof.
\end{proof}

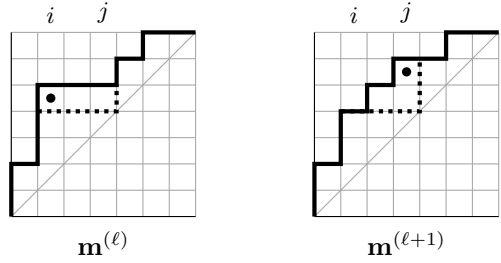
\begin{figure}
  \centering

\begin{tikzpicture}[scale=0.35]

  \foreach \i in {0,1,...,7}
  \draw[color=gray!70] (\i,0) -- (\i,7);

  \foreach \j in {0,1,...,7}
  \draw[color=gray!70]  (0,\j) -- (7,\j);

  \draw (0,0) -- (7,0);
  \draw (0,0) -- (0,7);
  \draw[color=gray!70] (0,0) -- (7,7);

  \draw[ultra thick]
  (0,0) -- (0,2) -- (1,2) -- (1,5) -- (3,5) -- (4,5) -- (4,6) -- (5,6) -- (5,7) -- (7,7);
  \draw[ultra thick, dotted]
  (1,4) -- (4,4) -- (4,5);
  
  \node[above] at (1.5, 7) {\small\( i \)};
  \node[above] at (3.5, 7) {\small\( j \)};
  \node at (3.5,-1) {\( \mm^{(\ell)} \)};
  \fill (1.5,4.5) circle (5pt);
\end{tikzpicture}
\qquad \qquad 
\begin{tikzpicture}[scale=0.35]

  \foreach \i in {0,1,...,7}
  \draw[color=gray!70] (\i,0) -- (\i,7);

  \foreach \j in {0,1,...,7}
  \draw[color=gray!70]  (0,\j) -- (7,\j);

  \draw (0,0) -- (7,0);
  \draw (0,0) -- (0,7);
  \draw[color=gray!70] (0,0) -- (7,7);

  \draw[ultra thick]
  (0,0) -- (0,2) -- (1,2) -- (1,4) --(2,4) -- (2,5) -- (3,5) -- (3,6) -- (5,6) -- (5,7) -- (7,7);
  \draw[ultra thick, dotted]
  (1,4) -- (4,4) -- (4,6);
  
  \node[above] at (1.5, 7) {\small\( i \)};
  \node[above] at (3.5, 7) {\small\( j \)};
  \node at (3.5,-1) {\( \mm^{(\ell+1)} \)};
  \fill (3.5,5.5) circle (5pt);
\end{tikzpicture}
\caption{The Hessenberg function \( \mm^{(\ell+1)} \) is obtained from
  \( \mm^{(\ell)} \) by moving the square (\( \bullet \)) 
  from column \( i \) to column \( j \) with \( i<j \).}
\label{fig:image13}
\end{figure} 

\begin{thm}\label{thm: master}
  Let \(f:\HH\to A\) be a function that satisfies the
  restricted modular law, as in \Cref{def:5}. Then \(f\) is determined
  by its values \(f(\pp_{n_1}+\cdots+\pp_{n_d})\) at the disjoint
  unions of the paths.
\end{thm}
\begin{proof}
  By \Cref{lem: algorithm terminates}, for any \( \mm\in\HH \), we can
  use \Cref{Algorithm} to express \( f(\mm) \) as a linear combination
  of \(f(\pp_{n_1}+\cdots+\pp_{n_d})\)'s. Thus \( f(\mm) \) is
  uniquely determined.
\end{proof}

\begin{proof}[Proof of \Cref{thm:E=G=S_intro}]
  In Proposition \ref{prop:E=G=S_at_path} we show that \Cref{thm:E=G=S_intro} is true when \( \mm=\pp_n \). Together with a multiplicative property proved in Propositions~\ref{prop:E_multiplicative},
  \ref{prop:G_multiplicative} and \ref{prop:S_multiplicative}, we deduce that if \( \mm = \pp_{n_1} + \cdots + \pp_{n_d} \) for some
  \( n_1,\dots,n_d\), we have 
  \begin{equation}\label{eq:25} 
    E_\mm(x;q) = G_\mm(x;q) = S_\mm(x;q),
  \end{equation}
  and
  \begin{equation}\label{eq:26}
    E_{\mm,k}(x;q) = G_{\mm,k}(x;q)
  \end{equation}
  for \( 1\le k \le n \). We also have that any function
  \( f:\HH\to A \) given by \( f(\mm)=E_\mm(x;q) \), \( G_\mm(x;q) \),
  \( S_\mm(x;q) \), \( E_{\mm,k}(x;q) \), or \( G_{\mm,k}(x;q) \)
  satisfies the restricted modular law from
  Propositions~\ref{prop:restricted_modular_law_for_E},
  \ref{prop:restricted_modular_law_for_G} and \ref{prop:restricted_modular_law_for_S}.
  Consequently, \Cref{thm: master} gives the desired identities.
\end{proof}

\section{Restricted modular law for \( E \), \( G\), and \( S \)}\label{sec: rm for EGS}
As a function satisfying the modular law is uniquely determined by its values on
disjoint unions of complete graphs, we showed in \Cref{thm: master} that a function satisfying
the restricted modular law is uniquely determined by its values on disjoint unions of paths.
In this section, we prove that the symmetric functions \( E_{\mm,k}(x;q) \), \( G_{\mm,k}(x;q) \),
and \( S_\mm(x;q) \) satisfy the restricted modular law, which is a main ingredient in
the proof of \Cref{thm:E=G=S_intro}.

\subsection{Restricted modular law for \( E \)}

Recall the definition of \( E_{\mm,k}(x;q) \): for \( \mm\in\HH_n \) and \( 1\le k\le n \),
\[
  E_{\mm, k}(x;q) = \sum_{\lambda \vdash n} \frac{c_{\lambda,k}(\mm;q)}{[k]_q} e_\lambda(x).
\]

In this subsection, we demonstrate that both \( E_{\mm,k}(x;q) \) and \( E_{\mm}(x;q) \)
satisfy the restricted modular law by establishing that \( c_{\lambda,k}(\mm;q) \)
satisfies it. Our approach is inspired by the proof of \Cref{thm:Hikita_intro}
given in \cite{Hikita2024}. 
We begin by introducing some terminology and key results from \cite{Hikita2024}.

Let \( V_n \) be a \( \QQ(q) \)-vector space with a formal
basis indexed by the standard Young tableaux of size \( n \). In other
words, an element \( f\in V_n \) can be represented as
\( f=\sum_{T\in \SYT(n)}c_TT \), where \( c_T\in \QQ(q) \).
Note that \( V_0\) is the one-dimensional vector space with a basis
\( \{\emptyset\} \), where \( \emptyset \) denotes an empty tableau.
Let \( V= \bigoplus_{n\ge0}V_n \).

For a nonnegative integer \( r \), we define the linear map
\( \Omega_r: V\rightarrow V \) by
\begin{equation*}
\Omega_r(T)=\sum_{k\geq 0} \psi_k^{(r)}(T;q)f_{k}^{(r)}(T).
\end{equation*}
 For a sequence
\( \alpha=(\alpha_1,\alpha_2,\dots,\alpha_{\ell}) \) of nonnegative
integers, we define the linear map  \( \Omega_\alpha: V\rightarrow V \) by
\[
  \Omega_{\alpha}=\Omega_{\alpha_{\ell}}\cdots\Omega_{\alpha_2}\Omega_{\alpha_1}.
\]
For \( \mm\in \HH_n \), we associate the sequence
\[
  \mm^{c}:=(n-\mm(n),n-\mm(n-1),\dots,n-\mm(1)).
\]
Then, by \Cref{defn:psi_pT}, we have
\begin{equation}\label{eq:18}
\Omega_{\mm^{c}}(\emptyset)=\sum_{T\in \SYT(n)}p_\mm(T;q)T.
\end{equation}

For a standard Young tableau \( T \) of size \( n \) and a positive
integer \( m<n \), we define \( \tau_m(T) \) to be the tableau
obtained by switching \( m \) and \( m+1 \). Note that \( \tau_m(T) \)
may not be a standard Young tableau. Lastly, we define \( K_{m,n} \)
to be the subspace of \( V_m\) generated by \( T-\tau_m(T) \) for
all \( T\in \SYT(n) \) such that \( \tau_m(T)\in \SYT(n) \).

We need the following results due to Hikita.

\begin{lem}[{\cite[Page 16]{Hikita2024}}]
\label{lem:3}
For a modular triple
\( (\mm,\mm',\mm'') \in \HH_n^3 \) of type I such that \(\mm(i)+1=\mm'(i)=\mm''(i)-1\)
and \(\mm'(i-1)<\mm'(i)<\mm'(i+1)\), where \( i\in [n-1] \), we have
\begin{equation}\label{eq:19}
  \left( (1+q)\Omega_{(\mm')^{c}}-q\Omega_{\mm^{c}}-\Omega_{(\mm'')^{c}}\right)(\emptyset)\in K_{n-1-\mm(i),n}.
\end{equation}
\end{lem}

\begin{lem}[{\cite[Page 19]{Hikita2024}}]
\label{lem:7}
For a modular triple
\( (\mm,\mm',\mm'') \in \HH_n^3 \) of type II such that \(\mm'(i)+1=\mm'(i+1)\),
\( \mm(i) = \mm'(i) = \mm''(i)-1 \), and
\( \mm(i+1)+1 = \mm'(i+1) = \mm''(i+1) \), where \(i\in [n-1] \), we
have
\begin{equation}\label{eq:20}
\left( (1+q)\Omega_{(\mm')^{c}}-q\Omega_{\mm^{c}}-\Omega_{(\mm'')^{c}} \right)(\emptyset)\in K_{n-i,n}.
\end{equation}
\end{lem}

We are now ready to prove that \( c_{\lambda,k}(\mm;q) \) satisfies
the restricted modular law. Recall that
  \[
    c_{\lambda,k}(\mm;q) = \prod_{i=1}^{\ell(\lambda)} [\lambda_i]_q!
    \sum_{T \in \SYT_k(\lambda)} p_\mm(T;q).
  \]

\begin{prop}\label{prop:restricted_modular_law_for_E}
  For a fixed partition \( \lambda \) and an integer \( k\ge1 \), the
  function \( f:\HH\rightarrow A \) defined by
  \( f(\mm) = c_{\lambda,k}(\mm;q) \) satisfies the restricted modular
  law.
  Consequently, for a fixed integer \( k\ge1 \), the function
  \( f:\HH\rightarrow A \) defined by
  \( f(\mm) = E_{\mm,k}(x;q) \) satisfies the restricted modular law.
\end{prop}

\begin{proof}
  Let \( \pr_{\lambda,k}:V\to \QQ(q) \) be the linear map given by
\begin{equation}\label{eq:21}
  \pr_{\lambda,k}(T)=
  \begin{cases}1 \qquad \text{if \( T\in \SYT_k(\lambda) \)}, \\
    0 \qquad \text{otherwise}.
  \end{cases}
\end{equation}
It is easy to check that if \( m<n-1 \), then \( \pr_{\lambda,k} \) maps
any element of \( K_{m,n} \) to zero.

Suppose that
\( (\mm,\mm',\mm'') \in \HH_n^3 \) is a modular triple of type I such that
\(\mm(i)+1=\mm'(i)=\mm''(i)-1\) and \(\mm'(i-1)<\mm'(i)<\mm'(i+1)\),
where \(\mm'(0)=0\). Note that \( n-1-\mm(i)<n-1 \). Therefore, by
applying the map \( \pr_{\lambda,k} \) to \eqref{eq:19} and using
\eqref{eq:18}, we obtain
\begin{equation}\label{eq:30}
  (1+q)c_{\lambda,k} (\mm';q) - qc_{\lambda,k} (\mm;q)- c_{\lambda,k}
  (\mm'';q) = 0.
\end{equation}

Suppose now that
\( (\mm,\mm',\mm'') \in \HH_n^3 \) is a restricted modular triple
of type II such that \(\mm'(i)+1=\mm'(i+1)\),
\( \mm(i) = \mm'(i) = \mm''(i)-1 \), and
\( \mm(i+1)+1 = \mm'(i+1) = \mm''(i+1) \), where \( 2\le i\le n-1 \).
Note that, since \( i\geq 2 \), we have \( n-i<n-1 \). Therefore, by
applying the map \( \pr_{\lambda,k} \) to \eqref{eq:20} and using
\eqref{eq:18}, we also obtain \eqref{eq:30}, which completes the
proof.
\end{proof}

\subsection{Restricted modular law for \( G \)}

Recall the definition of \( G_{\mm,k}(x;q) \): for \( \mm\in\HH_n \) and \( 1\le k\le n \),
\[
  G_{\mm,k}(x;q) = e_k(x) g_{\mm,n-k}(x;q).
\]
In this subsection, we show that Abreu and Nigro's $g$-functions satisfy
the restricted modular law, which implies that \( G_{\mm,k}(x;q) \)
also satisfies the restricted modular law. 

\begin{prop}\label{prop:restricted_modular_law_for_G}
  For a fixed integer \( k\ge0 \), the function
  \( f:\HH\rightarrow A \) defined by
  \( f(\mm) = g_{\mm,k}(x;q) \) satisfies the restricted modular law.
  Consequently, for a fixed integer \( k\ge1 \), the function
  \( f:\HH\rightarrow A \) defined by
  \( f(\mm) = G_{\mm,k}(x;q) \) satisfies the restricted modular law.
\end{prop}

Our proof of \Cref{prop:restricted_modular_law_for_G} closely
mirrors the work of Abreu and Nigro \cite[Proposition 3.4]{Abreu2021b},
who established a $q$-analog of the power sum expansion for the
chromatic symmetric function via the modular law. For completeness, we
present a full proof here. Recall the definition of the
\( g \)-functions in \eqref{eq:g_def}.

Let \( (\mm, \mm', \mm'') \in \HH_n^3 \) be a modular triple of type~I or a restricted modular triple of type~II. Since
\[
    \mm(i) \le \mm'(i) \le \mm''(i) \quad \text{for all } i \in [n],
\]
we have the containment
\[
    \mathfrak{S}_{n,\mm} \subseteq \mathfrak{S}_{n,\mm'} \subseteq \mathfrak{S}_{n,\mm''}.
\]
Our strategy to prove \Cref{prop:restricted_modular_law_for_G} is
to find a suitable bijection, which is described as follows.

\begin{prop}\label{pro:g-bij}
  Let \( (\mm, \mm', \mm'') \in \HH_n^3 \) be a modular triple of
  type~I or a restricted modular triple of type~II. Then there exists
  a bijection
\begin{equation}\label{eq:4 g-version}
    \phi: \mathfrak{S}_{n,\mm'} \to \mathfrak{S}_{n,\mm''} \setminus (\mathfrak{S}_{n,\mm'} \setminus \mathfrak{S}_{n,\mm})
\end{equation}
such that for all \( \sigma \in \mathfrak{S}_{n,\mm'} \), the cycle
type and the size of the cycle containing \( 1 \) in \( \sigma \) are
the same as those in \( \phi(\sigma) \), and
\begin{equation}\label{eq: (1+q)q^inv_m' =, g-version}
    (1 + q) q^{\wt_{\mm'}(\sigma)} = \wt(\sigma) + q^{\wt_{\mm''}(\phi(\sigma))},
\end{equation}
where
\[
    \wt(\sigma) =
    \begin{cases}
        q^{1 + \wt_{\mm}(\sigma)} & \text{if } \sigma \in \mathfrak{S}_{n,\mm}, \\
        q^{\wt_{\mm''}(\sigma)} & \text{if } \sigma \in \mathfrak{S}_{n,\mm'} \setminus \mathfrak{S}_{n,\mm}.
    \end{cases}
\]
\end{prop}

Note that by \eqref{eq: (1+q)q^inv_m' =, g-version} and the fact that
\( \phi \) preserves the cycle type and the size of the cycle
containing \( 1 \), if \( \sigma = \tau_1 \cdots \tau_j \) in cycle
notation with \( \tau_1 \) being the cycle containing \( 1 \), we have
\begin{equation}\label{eq:1}
    (1 + q) q^{\wt_{\mm'}(\sigma)} \eta(\sigma) = \wt(\sigma) \eta(\sigma) + q^{\wt_{\mm''}(\phi(\sigma))} \eta(\phi(\sigma)),
\end{equation}
where
\[
  \eta(\sigma) = (-1)^{|\tau_1| - n + k} \, h_{|\tau_1| - n + k}(x)
  \omega\left( \rho_{(|\tau_2|, \dots, |\tau_j|)}(x; q) \right),
\]
Summing \eqref{eq:1} over all \( \sigma \in \mathfrak{S}_{n,\mm'} \),
we obtain
\begin{equation*}
    (1 + q) g_{\mm',k}(x; q) = q\cdot g_{\mm,k}(x; q) + g_{\mm'',k}(x; q).
\end{equation*}
Hence, \Cref{pro:g-bij} implies \Cref{prop:restricted_modular_law_for_G}.

The rest of this subsection is devoted to proving \Cref{pro:g-bij},
which is established in \Cref{lem:case1,lem:case2}. We begin by
introducing the necessary definitions and results.

\begin{defn}\label{def:1}
  Let \( w=w_1 \cdots w_n \) be a word of integers. The \emph{pattern}
  \( \pat(w) \) of \( w \) is the word of length $n$
  such that for all \( i\in [n] \), if \( w_i \) is the \( d \)-th
  smallest element in \( \{w_1,\dots, w_n\} \), then \( \pat(w)_i = d \).
\end{defn}

For example, if \( w = 3523 \), then \( \pat(w) = 2312 \).

\begin{defn}\label{def:2}
  Let \( \sigma \in \mathfrak{S}_{n} \). For integers
  \( 1\le i<j<k\le n \), we denote by \( \sigma|_{i,j,k} \) the
  subword of \( \sigma \) consisting of \( i \), \( j \), and \( k \).
\end{defn}

Recall that \( \sigma^c \) is the word obtained by removing
parentheses from the cycle decomposition of \( \sigma \). For example,
if \( \sigma=462153 = (1,4)(2,6,3)(5) \), then \( \sigma^c = 142635 \)
and \( \pat(\sigma^c|_{3,4,6}) = \pat(463) = 231 \).

\begin{lem}\label{lem:9}
  Suppose that \( (\mm, \mm', \mm'')\in\HH_n^3 \) is a modular triple
  of type~I, with \( i \in [n-1] \) as in \Cref{def:5}. Let
  \( j=\mm'(i) \). For \( \sigma \in \mathfrak{S}_{n,\mm'} \), define
\[
    \phi(\sigma) =
    \begin{cases}
        \sigma & \text{if } \sigma(i) \neq j \text{ and } \pat(\sigma^c|_{i,j,j+1}) \in \{123, 132, 231, 321\}, \\
        \sigma' & \text{if } \sigma(i) = j \text{ or } \pat(\sigma^c|_{i,j,j+1}) \in \{213, 312\},
    \end{cases}
\]
where \( \sigma' = (j, j+1) \sigma (j, j+1) \), the conjugate of
\( \sigma \) by the transposition \( (j, j+1) \). Then \( \phi \) is a
bijection with the following domain and codomain:
\[
  \phi: \mathfrak{S}_{n,\mm'} \to \mathfrak{S}_{n,\mm''} \setminus
  (\mathfrak{S}_{n,\mm'} \setminus \mathfrak{S}_{n,\mm}).
\]
Moreover, the permutations \( \phi(\sigma) \) and \( \sigma \) have
the same cycle type and their cycles containing \( 1 \) have the same
size.
\end{lem}

\begin{proof}
  Since \( (\mm, \mm', \mm'') \) is a modular triple of type~I, we
  have \( j = \mm'(i)=\mm(i)+1>i \) and
  \( j = \mm'(i)=\mm''(i)-1<n \), and therefore
  \( 1\le i<j<j+1\le n \). Moreover, we have \(\mm'(j)=\mm'(j+1) \)
  and \( \mm(t)=\mm'(t)=\mm''(t) \) for all
  \(t\in[n]\setminus\{i\} \). Thus,
\begin{align}
    \label{eq:a}
    \mathfrak{S}_{n,\mm'} &= \{ \sigma \in \mathfrak{S}_{n,\mm''} : \sigma(i) \le j \}, \\
    \label{eq:b}
    \mathfrak{S}_{n,\mm} &= \{ \sigma \in \mathfrak{S}_{n,\mm''} : \sigma(i) \le j - 1 \}, \\
    \label{eq:c}
    \mathfrak{S}_{n,\mm'} \setminus \mathfrak{S}_{n,\mm} &= \{ \sigma \in \mathfrak{S}_{n,\mm''} : \sigma(i) = j \}, \\
    \label{eq:d}
    \mathfrak{S}_{n,\mm''} \setminus (\mathfrak{S}_{n,\mm'} \setminus \mathfrak{S}_{n,\mm}) &= \{ \sigma \in \mathfrak{S}_{n,\mm''} : \sigma(i) \ne j \}.
\end{align}

Let \(\sigma\in\mathfrak{S}_{n,\mm'}\) and \( \nu=\phi(\sigma) \). We
claim that the map \( \phi \) has the desired codomain, that is,
\( \nu \) lies in the set given in \eqref{eq:d}.
Note that for any \( \pi\in \mathfrak{S}_n \), we have
\( \pi\in \mathfrak{S}_{n,\mm''} \) if and only if
\( \pi(j,j+1)\in \mathfrak{S}_{n,\mm''} \) since
\( \mm''(j)=\mm''(j+1) \). We also have that
\( \pi\in \mathfrak{S}_{n,\mm''} \) if and only if
\( (j,j+1)\pi\in \mathfrak{S}_{n,\mm''} \) since
\( (\mm'')^{-1}(\{ j \})=\emptyset \). Therefore,
\( \sigma\in\mathfrak{S}_{n,\mm'}\subseteq\mathfrak{S}_{n,\mm''} \)
implies \( \sigma'\in \mathfrak{S}_{n,\mm''} \). Since \( \nu \) is
either \( \sigma \) or \( \sigma' \), we obtain
\( \nu\in \mathfrak{S}_{n,\mm''} \). Thus, by \eqref{eq:d}, it
suffices to show that \( \nu(i)\ne j \). If \( \sigma(i) = j \), then
\( \nu(i) = \sigma'(i)=j+1\ne j \). If \( \sigma(i) \ne j \), then
\( \sigma(i)=t \) for some \( t<j \) by \eqref{eq:a}. Thus
\( \sigma'(i)=\sigma(i)=t \), and therefore, \( \nu(i)=t\ne j \).
This shows the claim.

By \eqref{eq:a} and \eqref{eq:d}, it is straightforward to check that the map \( \phi \) is a bijection
whose inverse map is given by
\[
    \phi^{-1}(\sigma) =
    \begin{cases}
        \sigma & \text{if } \pat(\sigma^c|_{i,j,j+1}) \in \{123, 132, 231, 321\}, \\
      \sigma' & \text{if } \sigma(i)=j+1 \text{ or }
                \pat(\sigma^c|_{i,j,j+1}) \in \{213, 312\}.
    \end{cases}
\]
The last statement follows directly from the construction of
\( \phi \) and the fact that \( j>1 \).
\end{proof}

\begin{lem}\label{lem:case1}
  \Cref{pro:g-bij} holds for any modular triple of type~I.
\end{lem}

\begin{proof}
  Following the notation in \Cref{lem:9}, we will show that the map
  \( \phi \) satisfies the desired conditions in \Cref{pro:g-bij}.
  Note that it only remains to show that \( \phi \) satisfies
  \eqref{eq: (1+q)q^inv_m' =, g-version}.

  Let \( \sigma \in \mathfrak{S}_{n,\mm'} \). We claim that if
  \( \sigma(i) = j \) or
  \( \pat(\sigma^c|_{i,j,j+1}) \in \{213, 312\} \), then
  \( (\sigma')^c = (j, j+1)\sigma^c \). To prove the claim, observe
  that for any \(\pi\in \mathfrak{S}_n\) and
  \(\pi'=(j,j+1)\pi(j,j+1)\), we have \((\pi')^{c}=(j,j+1)\pi^c\) unless
  \(\pi\) satisfies one of the following two conditions:
\begin{enumerate}
\item \(j\) and \(j+1\) belong to the same cycle and \(j\) is the
  smallest integer in that cycle.
\item \(j\) is the smallest integer in the cycle it belongs to and the
  same is true for \(j+1\).
\end{enumerate}
Note that if \( \sigma \in \mathfrak{S}_{n,\mm'} \) satisfies
\( \sigma(i) = j \) or
\( \pat(\sigma^c|_{i,j,j+1}) \in \{213, 312\} \), then it does not
satisfy either of the two conditions above. Thus, the claim holds.

Now, we prove \eqref{eq: (1+q)q^inv_m' =, g-version}. By \eqref{eq:b}
and \eqref{eq:c}, it suffices to show:
\begin{align}
    \label{eq:e}
    \sigma(i) \neq j \quad &\Rightarrow \quad (1 + q) q^{\wt_{\mm'}(\sigma)} = q^{1 + \wt_{\mm}(\sigma)} + q^{\wt_{\mm''}(\phi(\sigma))}, \\
    \label{eq:f}
    \sigma(i) = j \quad &\Rightarrow \quad (1 + q) q^{\wt_{\mm'}(\sigma)} = q^{\wt_{\mm''}(\sigma)} + q^{\wt_{\mm''}(\phi(\sigma))}.
\end{align}
To do this, recall that
\[
  \wt_{\mm'}(\sigma) = \inv_{\mm'}((\sigma^c)^{-1})
  = |\{ (u,v)\in [n]\times[n] : \mbox{\( u<v\le\mm'(u) \), and $v$ precedes $u$ in \(\sigma^c\)} \}|.
\]
Let
\[
  d= |\{ (u,v)\in [n]\times[n] : \mbox{\(\{u,v\} \not\subseteq \{i,j,j+1\}\), \( u<v\le\mm'(u) \), and $v$ precedes $u$ in \(\sigma^c\)} \}|.
\]
Then, by the claim and the fact that \(\mm'(j)=\mm'(j+1) \) and
\( \mm(t)=\mm'(t)=\mm''(t) \) for all \(t\in[n]\setminus\{i\} \), the
following table confirms \eqref{eq:e} and \eqref{eq:f}. In the table,
the value \( \wt_{\mm''}(\phi(\sigma)) \) is highlighted in red when
\( \phi(\sigma) = \sigma' \).
\begin{center}
\begin{tabular}{c|c|c|c|c|c}
    \( \sigma(i)=j \) & \( \pat(\sigma^c|_{i,j,j+1}) \) & \( \wt_{\mm'}(\sigma) \) & \( 1 + \wt_{\mm}(\sigma) \) & \( \wt_{\mm''}(\sigma) \) & \( \wt_{\mm''}(\phi(\sigma)) \) \\
    \hline
    False & 123 & \( d+0 \) & \( d+1 \) &  & \( d+0 \) \\
    False & 132 & \( d+1 \) & \( d+2 \) &  & \( d+1 \) \\
    False & 213 & \( d+1 \) & \( d+1 \) &  & \( \textcolor{red}{d+2} \) \\
    False & 231 & \( d+1 \) & \( d+1 \) &  & \( d+2 \) \\
    False & 312 & \( d+1 \) & \( d+2 \) &  & \( \textcolor{red}{d+1} \) \\
    False & 321 & \( d+2 \) & \( d+2 \) &  & \( d+3 \) \\
    True & 123 & \( d+0 \) &  & \( d+0 \) & \( \textcolor{red}{d+1} \) \\
    True & 312 & \( d+1 \) &  & \( d+2 \) & \( \textcolor{red}{d+1} \) \\
\end{tabular}
\end{center}
This completes the proof.
\end{proof}

We now consider restricted modular triples of type II.

\begin{lem}\label{lem:10}
  Suppose that \( (\mm,\mm',\mm'') \) is a restricted modular triple
  of type II, with \( 2\le i\le n-1 \) as in \Cref{def:5}. Let
  \( j = \mm'(i+1) \). For \( \sigma \in \mathfrak{S}_{n,\mm'} \),
  define
\[
    \phi(\sigma) =
    \begin{cases}
        \sigma & \text{if } \sigma(i+1) \neq j \text{ and } \pat(\sigma^c|_{i,i+1,j}) \in \{123, 213, 312, 321\}, \\
        \sigma' & \text{if } \sigma(i+1) = j \text{ or } \pat(\sigma^c|_{i,i+1,j}) \in \{132, 231\},
    \end{cases}
\]
where \( \sigma' = (i, i+1) \sigma (i, i+1) \), the conjugate of
\( \sigma \) by the transposition \( (i, i+1) \). Then \( \phi \) is a
bijection with the following domain and codomain:
\[
  \phi: \mathfrak{S}_{n,\mm'} \to \mathfrak{S}_{n,\mm''} \setminus
  (\mathfrak{S}_{n,\mm'} \setminus \mathfrak{S}_{n,\mm}).
\]
Moreover, the permutations \( \phi(\sigma) \) and \( \sigma \) have
the same cycle type and their cycles containing \( 1 \) have the same
size.
\end{lem}

\begin{proof}
  Since \( (\mm,\mm',\mm'') \) is a restricted modular triple of type
  II, we have \( j = \mm'(i+1) = \mm(i+1)+1 >i+1 \), hence
  \( 1\le i<i+1<j\le n \). Moreover, we have
  \( \mm(t)=\mm'(t)=\mm''(t) \) for all
  \( t\in[n]\setminus\{i,i+1\} \). Thus,
\begin{align}
    \label{eq:a'}
  \mathfrak{S}_{n,\mm'}
  &= \{ \sigma \in \mathfrak{S}_{n,\mm''} : \sigma(i) \ne j \}, \\
    \label{eq:b'}
  \mathfrak{S}_{n,\mm}
  &= \{ \sigma \in \mathfrak{S}_{n,\mm''} : \sigma(i) \ne j, \sigma(i+1) \ne j \}, \\
    \label{eq:c'}
  \mathfrak{S}_{n,\mm'} \setminus \mathfrak{S}_{n,\mm}
  &= \{ \sigma \in \mathfrak{S}_{n,\mm''} : \sigma(i+1) = j \}, \\
    \label{eq:d'}
  \mathfrak{S}_{n,\mm''} \setminus (\mathfrak{S}_{n,\mm'} \setminus \mathfrak{S}_{n,\mm})
  &= \{ \sigma \in \mathfrak{S}_{n,\mm''} : \sigma(i+1) \neq j \}.
\end{align}

Let \( \sigma \in \mathfrak{S}_{n,\mm'} \) and \( \nu = \phi(\sigma) \).
We first show that the codomain of \( \phi \) is given by \eqref{eq:d'}.
For any \( \pi\in \mathfrak{S}_n \), we have
\( \pi\in \mathfrak{S}_{n,\mm''} \)  if and only if
\( \pi(i,i+1)\in \mathfrak{S}_{n,\mm''} \) since \( \mm''(i)=\mm''(i+1) \).
In addition, \( \pi\in \mathfrak{S}_{n,\mm''} \) if and only if
\( (i,i+1)\pi\in \mathfrak{S}_{n,\mm''} \) since \( (\mm'')^{-1}(\{ i \})=\emptyset \).
Therefore, \( \sigma \in \mathfrak{S}_{n,\mm'}\subseteq \mathfrak{S}_{n,\mm''}  \)
implies that \( \sigma' \in \mathfrak{S}_{n,\mm''} \).
Since \( \nu \) is either \( \sigma \) or \( \sigma' \), we obtain
\( \nu\in\mathfrak{S}_{n,\mm''} \)
Thus, by \eqref{eq:d'}, we need to check that \( \nu(i+1) \ne j \).
If \( \sigma(i+1) = j \), then
\( \nu(i) = \sigma'(i)=j \), hence \( \nu(i+1) \ne j \).
If \( \sigma(i+1) \ne j \), then
\( j\not\in\{\sigma(i),\sigma(i+1)\} \) by \eqref{eq:a'}, which
implies that
\( j\not\in\{\sigma(i),\sigma(i+1),\sigma'(i),\sigma'(i+1)\} \).
Therefore, \( \nu(i+1)\ne j \). This shows the map \( \phi \) has the desired codomain.

By \eqref{eq:a'} and \eqref{eq:d'}, it is straightforward to check that
the map \( \phi \) is a bijection whose inverse map is given by
\[
    \phi^{-1}(\sigma) =
    \begin{cases}
        \sigma & \text{if } \pat(\sigma^c|_{i,i+1,j}) \in \{123, 213, 312, 321\}, \\
      \sigma' & \text{if }\sigma(i)=j \text{ or }
                \pat(\sigma^c|_{i,i+1,j}) \in \{132,231\}.
    \end{cases}
\]
The last statement follows directly from the construction of
\( \phi \) and the fact that \( i>1 \).
\end{proof}

\begin{lem}\label{lem:case2}
  \Cref{pro:g-bij} holds for any restricted modular triple of type~II.
\end{lem}

\begin{proof}
  Following the notation in \Cref{lem:10}, we will show that the map
  \( \phi \) satisfies the desired conditions in \Cref{pro:g-bij}.
  Note that it only remains to show that \( \phi \) satisfies
  \eqref{eq: (1+q)q^inv_m' =, g-version}.

  Let \( \sigma \in \mathfrak{S}_{n,\mm'} \). We claim that if
  \( \sigma(i+1) = j \) or
  \( \pat(\sigma^c|_{i,i+1,j}) \in \{132, 231\} \), then
  \( (\sigma')^c = (i, i+1)\sigma^c \). This can be proved by the same
  reasoning as in the proof of Lemma \ref{lem:case1}.

  Now we prove \eqref{eq: (1+q)q^inv_m' =, g-version}. By
  \eqref{eq:b'} and \eqref{eq:c'}, it suffices to show:
\begin{align}
    \label{eq:e'}
    \sigma(i+1) \neq j \quad &\Rightarrow \quad (1 + q) q^{\wt_{\mm'}(\sigma)} = q^{1 + \wt_{\mm}(\sigma)} + q^{\wt_{\mm''}(\phi(\sigma))}, \\
    \label{eq:f'}
    \sigma(i+1) = j \quad &\Rightarrow \quad (1 + q) q^{\wt_{\mm'}(\sigma)} = q^{\wt_{\mm''}(\sigma)} + q^{\wt_{\mm''}(\phi(\sigma))}.
\end{align}
To do this, let
\[
  d= |\{ (u,v)\in [n]\times[n] : \mbox{\(\{u,v\} \not\subseteq \{i,i+1,j\}\), \( u<v\le\mm'(u) \), and $v$ precedes $u$ in \(\sigma^c\)} \}|.
\]
Then, by the claim and the fact that \((\mm')^{-1}(\{i\})=\emptyset\)
and \( \mm(t)=\mm'(t)=\mm''(t) \) for all
\( t\in[n]\setminus\{i,i+1\} \), the following table confirms
\eqref{eq:e'} and \eqref{eq:f'}. In the table, the value
\( \wt_{\mm''}(\phi(\sigma)) \) is highlighted in red when
\( \phi(\sigma) = \sigma' \).
\begin{center}
\begin{tabular}{c|c|c|c|c|c}
    \( \sigma(i+1)=j \) & \( \pat(\sigma^c|_{i,i+1,j}) \) & \( \wt_{\mm'}(\sigma) \) & \( 1 + \wt_{\mm}(\sigma) \) & \( \wt_{\mm''}(\sigma) \) & \( \wt_{\mm''}(\phi(\sigma)) \) \\
    \hline
    False & 123 & \( d+0 \) & \( d+1 \) &  & \( d+0 \) \\
    False & 132 & \( d+1 \) & \( d+1 \) &  & \( \textcolor{red}{d+2} \) \\
    False & 213 & \( d+1 \) & \( d+2 \) &  & \( d+1 \) \\
    False & 231 & \( d+1 \) & \( d+2 \) &  & \( \textcolor{red}{d+1} \) \\
    False & 312 & \( d+1 \) & \( d+1 \) &  & \( d+2 \) \\
    False & 321 & \( d+2 \) & \( d+2 \) &  & \( d+3 \) \\
    True & 132 & \( d+1 \) &  & \( d+1 \) & \( \textcolor{red}{d+2} \) \\
    True & 231 & \( d+1 \) &  & \( d+2 \) & \( \textcolor{red}{d+1} \) \\
\end{tabular}
\end{center}
This completes the proof.
\end{proof}

\subsection{Restricted modular law for \(S\)}

Recall that
\[
  S_{\mm}(x;q)=\sum_{\lambda \vdash n}\sum_{T\in \PT'_{\mm}(\lambda)}q^{\inv_{\mm}(T)}s_{\lambda}(x).
\]
The goal of this subsection is to prove the following proposition.

\begin{prop}\label{prop:restricted_modular_law_for_S}
  The function \( f:\HH\rightarrow A \) defined by
  \( f(\mm) = S_{\mm}(x;q) \) satisfies the restricted modular law.
\end{prop}

We first need some definitions and known results.

\begin{defn}\label{def:7}
  Let \( P \) be a poset on \( [n] \) and consider a weak composition
  \( \alpha \) with \( |\alpha|\leq n \). A \emph{\(P\)-array} of shape
  \(\alpha \) is a filling \( T \) of \( D(\alpha) \) with distinct
  integers in \( [n] \) such that \( T(i,j)<_P T(i,j+1) \) whenever
  \( (i,j), (i,j+1)\in D(\alpha) \). The set of \(P\)-arrays of shape
  \( \alpha\) is denoted by \( \PA_P(\alpha) \). Let
  \( \PA'_P(\alpha) \) denote the set of \( T\in \PA_P(\alpha) \) such
  that \( T(1,1)=1 \). For \( T\in \PA_\mm(\alpha) \), we define
  \( \inv_\mm(T) \) in the same way as in \Cref{def:inv}.
\end{defn}

In the above definition, if \( (1,1)\notin D(\alpha) \), i.e.,
\( \alpha_1=0 \), then \( \PA'_P(\alpha) \) is simply the empty set.
Recall that we identify \( \mm\in \HH_n \) with its
corresponding poset \( P \) and also with the incomparability graph
of \( P \). Therefore we abuse notation to denote
\( \PA_\mm(\alpha) \), \( \PA'_\mm(\alpha) \) and \( \inv_\mm(T) \).

For a partition
\( \lambda=(\lambda_1,\lambda_2,\dots,\lambda_{\ell}) \) and a
permutation \( w\in \mathfrak{S_{\ell}} \), let \( w(\lambda) \) be
the integer sequence
\( w(\lambda) = (w(\lambda)_1,\dots,w(\lambda)_\ell) \), where
\( w(\lambda)_i = \lambda_{w(i)}+i-w(i) \). For \( \lambda \vdash n \)
and \( \mm\in \HH_n \), Shareshian and Wachs \cite[Proof of
Theorem 6.3]{Shareshian2016} showed that
\begin{equation}\label{eq:22}
  \sum_{T\in \PT_\mm(\lambda)}q^{\inv_\mm(T)}=
  \sum_{w \in \mathfrak{S}_{\ell(\lambda)}}(-1)^{\ell(w)}\sum_{T\in \PA_\mm(w(\lambda))}q^{\inv_\mm(T)}
\end{equation}
by constructing an explicit sign-reversing involution on
\( \bigcup_{w\in\mathfrak{S}_{\ell(\lambda)}}\PA_\mm(w(\lambda)) \) with the fixed
point set \( \PT_{\mm}(\lambda) \). The following lemma can be derived
from \cite[Proof of Theorem 4.5]{Shareshian2016}.

\begin{lem}
\label{lem:6}
For \( \mm\in\HH_n \) and nonnegative integers \( a \) and
\( b \) with \( a+b\leq n \), there exists a bijection
  \begin{equation}\label{eq:24}
    \SW_{a,b}: \PA_\mm((a,b))\rightarrow \PA_\mm((b,a))
  \end{equation}
  such that \( \inv_\mm(T)=\inv_\mm(\SW_{a,b}(T)) \), and the set of
  entries in \( T \) is equal to that of \( \SW_{a,b}(T) \).
\end{lem}

Using the map \( \SW_{a,b} \), we obtain the following result. The
proof is mostly identical to that of \eqref{eq:22}. However, we
present it for the sake of completeness.

\begin{lem}\label{lem:8}
For a partition
\( \lambda=(\lambda_1,\lambda_2,\dots,\lambda_{\ell}) \) of \( n \) and
\( \mm\in \HH_n \),
we have
\begin{equation}\label{eq:23}
  \sum_{T\in \PT'_\mm(\lambda)}q^{\inv_\mm(T)}
  =\sum_{w \in \mathfrak{S}_{\ell}}(-1)^{\ell(w)}\sum_{T\in \PA'_\mm(w(\lambda))}q^{\inv_\mm(T)}.
\end{equation}
\end{lem}

\begin{proof}
  It suffices to construct a sign-reversing involution \( \phi \) on
  \( \bigcup_{w\in\mathfrak{S}_{\ell}}\PA'_\mm(w(\lambda)) \) with the
  fixed point set \( \PT'_{\mm}(\lambda) \), that is, if
  \( T\in \PA'_\mm(w(\lambda)) \) and \( \phi(T)\ne T \), then
  \( \inv_\mm(\phi(T)) = \inv_\mm(T) \),
  \( \phi(T)\in \PA'_\mm(\sigma(\lambda)) \), and
  \( (-1)^{\ell(\sigma)} = -(-1)^{\ell(w)} \) for some
  \( \sigma\in \mathfrak{S}_{\ell} \). Let
  \( T\in \PA'_{\mm}(w(\lambda)) \). Consider the pairs \( (i,j) \) of
  integers with \( i\geq 2 \) satisfying either:
  \begin{itemize}
  \item \( T(i-1,j)>_\mm T(i,j)\), or
  \item \( (i-1,j)\notin D(w(\lambda)) \) and
    \( (i,j)\in D(w(\lambda)) \).
  \end{itemize}
  If there are no such pairs \( (i,j) \), then we define
  \( \phi(T) = T \). This can happen if and only if \( w \) is the
  identity permutation and \( T\in \PT'_\mm(\lambda) \).

  Suppose that the pairs \( (i,j) \) exist. Choose the pair
  \( (i,j) \) such that \( j \) is minimal, and among those, \( i \)
  is also minimal. Let \( a=w(\lambda)_{i-1}-(j-1) \) and
  \( b=w(\lambda)_i-j \), and consider \( L\in \PA_\mm((a,b)) \) given
  by \( L(1,j')=T(i-1,j'+j-1) \) and \( L(2,j')=T(i,j'+j) \) for
  \( j'\ge1 \). Let \( L'=\SW_{a,b}(L) \), where \( \SW_{a,b} \) is
  the map in~\Cref{lem:6}, and let \( w' = w(i-1,i)\), which is the
  permutation obtained from \( w \) by switching \( w(i-1) \) and
  \( w(i) \). We define \( \phi(T) = T'\in \PA_\mm(w'(\lambda)) \) as
  follows:
  \[
    T'(i',j')=
    \begin{cases}
      L'(1,j'-j+1) & \mbox{if \( i'=i-1 \) and \( j'\geq j \)},\\
      L'(2,j'-j) & \mbox{if \( i'=i \) and \( j'\geq j+1 \)},\\
      T(i',j') & \mbox{otherwise.}
    \end{cases}
  \]
  It is easy to check that \( T'(1,1)=1 \), since otherwise we would
  have \( (i,j) = (2,1) \), but then
  \( 1=T(1,1)=T(i-1,j)>_{\mm}T(i,j) \), which is impossible. Thus,
  we conclude \( T'\in \PA'_\mm(w'(\lambda)) \).

  It is easy to see that \( \phi \) is a desired sign-reversing involution,
  which completes the proof.
\end{proof}

By \Cref{lem:8}, to prove \Cref{prop:restricted_modular_law_for_S}, it suffices to show the
following proposition. 

\begin{prop}\label{prop: RM for special coloring}
  For a fixed composition \( \alpha \), the function
  \( f_\alpha:\HH\rightarrow A \) defined by
  \( f_\alpha(\mm) = \sum_{T\in\PA'_{\mm}(\alpha)}q^{\inv_\mm(T)} \) satisfies
  the restricted modular law. 
\end{prop}

\begin{proof}
  Let \((\mm,\mm',\mm'')\in \HH_n^3 \) be a modular triple of type~I
  or a restricted modular triple of type~II. We denote by \( E(\mm) \)
  the set of edges of the graph \( \mm \), i.e,
  \[
    E(\mm) = \{ (i,j)\in [n]\times[n]: i<j\le \mm(i)\}.
  \]
  Since \( E(\mm)\subseteq E(\mm')\subseteq E(\mm'') \), we have
  \( \PA'_{\mm''}(\alpha) \subseteq \PA'_{\mm'}(\alpha) \subseteq
  \PA'_{\mm}(\alpha) \). We claim that there exists a bijection
\begin{equation}\label{eq:4}
  \phi: \PA'_{\mm'}(\alpha) \to \PA'_{\mm}(\alpha)\setminus (\PA'_{\mm'}(\alpha) \setminus \PA'_{\mm''}(\alpha))
  \end{equation}
  such that for all \( T\in\PA'_{\mm'}(\alpha) \),
\begin{equation}\label{eq: (1+q)q^inv_m' =}
    (1+q)q^{\inv_{\mm'}(T)} = q^{1+\inv_\mm(\phi(T))} + \wt(T),
\end{equation}
  where
  \[
    \wt(T) =
    \begin{cases}
        q^{\inv_{\mm''}(T)} & \mbox{if \( T \in \PA'_{\mm''}(\alpha) \)},\\
        q^{1+\inv_{\mm}(T)}& \mbox{if \( T \in \PA'_{\mm'}(\alpha) \setminus \PA'_{\mm''}(\alpha) \)}.
    \end{cases}
  \]
  Note that summing \eqref{eq:
    (1+q)q^inv_m' =} over all \( T\in\PA'_{\mm'}(\alpha) \) gives the
  desired equation
  \begin{equation}\label{eq:34}
(1+q)\sum_{T\in\PA'_{\mm'}(\alpha)}q^{\inv_{\mm'}(T)}
= q\sum_{T\in\PA'_{\mm}(\alpha)}q^{\inv_\mm(T)}
+ \sum_{T\in\PA'_{\mm''}(\alpha)}q^{\inv_{\mm''}(T)}.
  \end{equation}
  Hence, it suffices to prove the claim. 

  \textbf{Case 1:} Suppose that \( (\mm,\mm',\mm'') \) is a modular
  triple of type I, with \( i\in[n-1] \) as in \Cref{def:5}. Let
  \( j = \mm'(i) \). Note that \( j= \mm'(i) = \mm(i)+1 > 1 \) and
  \( j=\mm'(i)=\mm''(j)-1 < n \). Since \((\mm, \mm', \mm'')\) is a
  modular triple of type I, the following conditions hold:
  \begin{description}
  \item[(C1)] \( E(\mm'') = E(\mm') \sqcup \{(i,j+1)\} = E(\mm) \sqcup \{(i,j),
(i,j+1)\} \).
\item[(C2)] \( (j,j+1) \) is an edge in each of the graphs \( \mm \),
  \( \mm' \), and \( \mm'' \).
\item[(C3)] For every \( t\in [n]\setminus\{i,j,j+1\} \) and
  \( H\in \{\mm,\mm',\mm''\} \), we have \( (j,t)\in E(H) \) if and only
  if \( (j+1,t)\in E(H) \).
  \end{description}
  Note that Condition (C3) comes from $\mm'(j)=\mm'(j+1)$. Here, we use the notation \( A\sqcup B \) to mean the union
  \( A\cup B \), with the additional information that
  \( A\cap B = \emptyset \).

  Let \( T \in \PA'_{\mm}(\alpha) \). For each \( k \in [n] \), we
  denote by \( r_k(T) \) the row index of the entry \( k \) in
  \( T \). Let \( \sigma(T) = \pat(r_i(T) r_j(T) r_{j+1}(T))\) then by Condition~(C2), we
  have \( r_j(T)\ne r_{j+1}(T) \). Therefore,
  \[ 
    \sigma(T)\in \mathfrak{S}_3 \cup \{121,212,112,221\}.
  \]
  By Conditions~(C1), (C2), and (C3), we have
  \begin{align}
    \label{eq:35}    \PA'_{\mm'}(\alpha)
    &= \{T\in \PA'_{\mm}(\alpha): \sigma(T) \in \mathfrak{S}_3 \cup \{121,\,212\}\}, \\
   \label{eq:36} \PA'_{\mm''}(\alpha)
    &= \{T\in \PA'_{\mm}(\alpha): \sigma(T) \in \mathfrak{S}_3\} , \\
   \label{eq:48} \PA'_{\mm'}(\alpha) \setminus \PA'_{\mm''}(\alpha)
    &= \{T\in \PA'_{\mm}(\alpha): \sigma(T) \in \{121,\,212\}\} , \\
   \label{eq:47} \PA'_{\mm}(\alpha)\setminus (\PA'_{\mm'}(\alpha) \setminus \PA'_{\mm''}(\alpha))
    &= \{T\in \PA'_{\mm}(\alpha): \sigma(T) \in \mathfrak{S}_3 \cup \{112,221\}\}.
  \end{align}

  We are now ready to construct the map \( \phi \) in \eqref{eq:4}.
  For \( T \in \PA'_{\mm'}(\alpha) \), we define
  \[
    \phi(T) =
    \begin{cases}
     T & \mbox{if \( \sigma(T)\in \{123,132,312,321\} \)},\\
     T'& \mbox{if \( \sigma(T)\in \{213,231,121,212\} \)},
    \end{cases}
  \]
  where \( T' \) is the filling obtained from \( T \) by switching the
  entries \( j \) and \( j+1 \). Note that since \( j > 1 \), we have
  \( T'(1,1) = 1 \). By \eqref{eq:35} and \eqref{eq:47}, the map
  \( \phi \) is a bijection as described in \eqref{eq:4}. It remains to
  prove \eqref{eq: (1+q)q^inv_m' =}. By \eqref{eq:36} and
  \eqref{eq:48}, we can restate it as follows:
  \begin{align}
    \label{eq:49}
    \sigma(T) \in \mathfrak{S}_3 \quad &\Rightarrow \quad 
   (1+q)q^{\inv_{\mm'}(T)} = q^{1+\inv_\mm(\phi(T))} +
      q^{\inv_{\mm''}(T)},\\
    \label{eq:50}
   \sigma(T) \in \{121,212\} \quad &\Rightarrow \quad 
     (1+q)q^{\inv_{\mm'}(T)} = q^{1+\inv_\mm(\phi(T))} +
    q^{1+\inv_{\mm}(T)} .
  \end{align}
    To do this, recall that
  \[
    \inv_{\mm'}(T) =  | \{(u,v)\in E(\mm') : \mbox{\( u<v \) and \( u \) is lower than \( v \) in \( T \)}\} |.
  \]
  Let 
  \[
    c = | \{(u,v)\in E(\mm') :
    \mbox{\( \{u,v\}\not\subseteq\{i,j,j+1\} \), \( u<v \), and
      \( u \) is lower than \( v \) in \( T \)}\} |.
  \]
  Then, by the following table, we obtain \eqref{eq:49} and \eqref{eq:50}, as desired. In the following table, the value \( 1 + \inv_{\mm}(\phi(T)) \) is highlighted in red whenever \( \phi(T) = T' \).

  \begin{center}
   \begin{tabular}{c|c|c|c|c}
    \( \sigma(T) \)  & \( \inv_{\mm'}(T) \) & 1+\( \inv_{\mm}(\phi(T)) \) & \( \inv_{\mm''}(T)\) & \( 1+ \inv_{\mm}(T)\)  \\
    \hline
    123 & \( c+0 \) & \( c+1 \) & \( c+0 \) & \\
    132 & \( c+1 \) & \( c+2 \) & \( c+1 \) & \\
    213 & \( c+1 \) & \textcolor{red}{\( c+2 \)} & \( c+1 \) & \\
    231 & \( c+1 \) & \textcolor{red}{\( c+1 \)} & \( c+2 \) & \\
    312 & \( c+1 \) & \( c+1 \) & \( c+2 \) & \\
    321 & \( c+2 \) & \( c+2 \) & \( c+3 \) & \\
    121 & \( c+1 \) & \textcolor{red}{\( c+1 \)} &  & \( c+2 \) \\
    212 & \( c+1 \) & \textcolor{red}{\( c+2 \)} &  & \( c+1 \)
\end{tabular}            
  \end{center}

  \textbf{Case 2:} Suppose that \( (\mm,\mm',\mm'') \) is a restricted modular
  triple of type II, with \( 2\le i\le n-1 \) as in \Cref{def:5}. Let \( j = \mm'(i+1) \). Since \((\mm, \mm', \mm'')\) is a restricted modular triple of type II, the following conditions hold:
  \begin{description}
  \item[(C1)] \( E(\mm'') = E(\mm') \sqcup \{(i,j)\} = E(\mm) \sqcup \{(i,j),
(i+1,j)\} \).
\item[(C2)] \( (i,i+1) \) is an edge in each of the graphs \( \mm \),
  \( \mm' \), and \( \mm'' \).
\item[(C3)] For every \( t\in [n]\setminus\{i,i+1,j\} \) and
  \( H\in \{\mm,\mm',\mm''\} \), we have \( (i,t)\in E(H) \) if and only
  if \( (i+1,t)\in E(H) \).
  \end{description}

  In this case, for \( T \in \PA'_{\mm}(\alpha) \), let \( \sigma(T) = \pat(r_i(T) r_j(T) r_{j+1}(T))\) then by Condition~(C2), we
  have \( r_i(T)\ne r_{i+1}(T) \). Therefore,
  \[ 
    \sigma(T)\in \mathfrak{S}_3 \cup \{121,212,211,122\}.
  \]
  By Conditions~(C1), (C2), and (C3), we have
  \begin{align}
    \label{eq:51}    \PA'_{\mm'}(\alpha)
    &= \{T\in \PA'_{\mm}(\alpha): \sigma(T) \in \mathfrak{S}_3 \cup \{121,\,212\}\}, \\
   \label{eq:52} \PA'_{\mm''}(\alpha)
    &= \{T\in \PA'_{\mm}(\alpha): \sigma(T) \in \mathfrak{S}_3\} , \\
   \label{eq:53} \PA'_{\mm'}(\alpha) \setminus \PA'_{\mm''}(\alpha)
    &= \{T\in \PA'_{\mm}(\alpha): \sigma(T) \in \{121,\,212\}\} , \\
   \label{eq:54} \PA'_{\mm}(\alpha)\setminus (\PA'_{\mm'}(\alpha) \setminus \PA'_{\mm''}(\alpha))
    &= \{T\in \PA'_{\mm}(\alpha): \sigma(T) \in \mathfrak{S}_3 \cup \{221,122\}\}.
  \end{align}

  We are now ready to construct the map \( \phi \) in \eqref{eq:4}.
  For \( T \in \PA'_{\mm'}(\alpha) \), we define
  \[
    \phi(T) =
    \begin{cases}
     T & \mbox{if \( \sigma(T)\in \{123,213,231,321\} \)},\\
     T'& \mbox{if \( \sigma(T)\in \{132,312,121,212\} \)},
    \end{cases}
  \]
  where \( T' \) is the filling obtained from \( T \) by switching the
  entries \( i \) and \( i+1 \).
  Note that since \( i > 1 \), we have
  \( T'(1,1) = 1 \). By \eqref{eq:51} and \eqref{eq:54}, the map
  \( \phi \) is a bijection as described in \eqref{eq:4}. It remains to
  prove \eqref{eq: (1+q)q^inv_m' =}. By \eqref{eq:52} and
  \eqref{eq:53}, we can restate it as follows:
  \begin{align}
    \label{eq:55}
    \sigma(T) \in \mathfrak{S}_3 \quad &\Rightarrow \quad 
   (1+q)q^{\inv_{\mm'}(T)} = q^{1+\inv_\mm(\phi(T))} +
      q^{\inv_{\mm''}(T)},\\
    \label{eq:56}
   \sigma(T) \in \{121,212\} \quad &\Rightarrow \quad 
     (1+q)q^{\inv_{\mm'}(T)} = q^{1+\inv_\mm(\phi(T))} +
    q^{1+\inv_{\mm}(T)} .
  \end{align}
  Let 
  \[
    c = | \{(u,v)\in E(\mm) :
    \mbox{\( \{u,v\}\not\subseteq\{i,i+1,j\} \), \( u<v \), and
      \( u \) is lower than \( v \) in \( T \)}\} |.
  \]
  Then, by the following table, we obtain \eqref{eq:55} and \eqref{eq:56}, as desired. In the following table, the value \( 1 + \inv_{\mm}(\phi(T)) \) is highlighted in red whenever \( \phi(T) = T' \).

  \begin{center}
   \begin{tabular}{c|c|c|c|c}
    \( \sigma(T) \)  & \( \inv_{\mm'}(T) \) & 1+\( \inv_{\mm}(\phi(T)) \) & \( \inv_{\mm''}(T)\) & \( 1+ \inv_{\mm}(T)\)  \\
    \hline
    123 & \( c+0 \) & \( c+1 \) & \( c+0 \) & \\
        132 & \( c+1 \) & \textcolor{red}{\( c+2 \)} & \( c+1 \) & \\
    213 & \( c+1 \) & \( c+2 \) & \( c+1 \) & \\
    231 & \( c+1 \) & \( c+1 \) & \( c+2 \) & \\
    312 & \( c+1 \) & \textcolor{red}{\( c+1 \)} & \( c+2 \) & \\
    321 & \( c+2 \) & \( c+2 \) & \( c+3 \) & \\
    121 & \( c+1 \) & \textcolor{red}{\( c+2 \)} &  & \( c+1 \) \\
    212 & \( c+1 \) & \textcolor{red}{\( c+1 \)} &  & \( c+2 \)
\end{tabular}            
  \end{center}
This completes the proof.
\end{proof}

\section{Multiplicativity for \( E \), \( G \) and \( S \)}
\label{sec: multiplicativity}

Let \( \mm \) be a Hessenberg function given by \( \mm=\mm_1+\mm_2 \) for some Hessenberg
functions \( \mm_1 \) and \( \mm_2 \). Then it is straightforward to check that the chromatic
quasisymmetric functions satisfy the following multiplicativity property:
\[
  X_{\mm}(x;q) = X_{\mm_1}(x;q) X_{\mm_2}(x;q).
\]
In this section, we show that \( E_{\mm,k}(x;q) \), \( G_{\mm,k}(x;q) \), and \( S_\mm(x;q) \)
also have a similar multiplicative property in Propositions~\ref{prop:E_multiplicative},
\ref{prop:G_multiplicative}, and \ref{prop:S_multiplicative}, respectively.

\subsection{Multiplicativity for \( E \)}
Recall the definition of \( E_{\mm,k}(x;q) \): for
\( \mm\in \HH_n \) and \( 1\le k \le n \),
\[
  E_{\mm,k}(x;q) = \sum_{\lambda\vdash n} \frac{c_{\lambda,k}(\mm;q)}{[k]_q} e_\lambda(x).
\]
\begin{prop}\label{prop:E_multiplicative}
  Let \( \mm \in \HH_n\)
  such that \( \mm = \mm_1 + \mm_2 \) for some Hessenberg functions
  \( \mm_1 \) and \( \mm_2 \).
  For \( 1\leq k \leq n \), we have
  \[
    E_{\mm,k}(x;q) = E_{\mm_1,k}(x;q) X_{\mm_2}(x;q).
  \]
  Consequently,
  \[
    E_{\mm}(x;q) = E_{\mm_1}(x;q) X_{\mm_2}(x;q).
  \]
\end{prop}

\begin{proof}
  Let \( \mm_1\in \HH_{n_1} \) and \( \mm_2 \in \HH_{n_2} \).
  The identity we
  need to show is
   \[
     \sum_{\lambda \vdash n} \frac{c_{\lambda,k}(\mm;q)}{[k]_q}e_{\lambda}(x)
     = \sum_{\mu\vdash n_1}\frac{c_{\mu,k}(\mm_1;q)}{[k]_q} e_{\mu}(x) 
        \sum_{\nu \vdash n_2} c_{\nu}(\mm_2;q) e_{\nu}(x),
   \]
   where \( c_\nu(\mm_2;q)=\sum_{i\ge 1}c_{\nu,i}(\mm_2;q) \).
   Hence, it suffices to show that for \( \lambda\vdash n \),
   \begin{equation}\label{eq:5}
     c_{\lambda,k}(\mm;q)
     =\sum_{\substack{\mu\vdash n_1, \nu \vdash n_2\\
         \Sort(\mu \cup \nu)=\lambda}} c_{\mu,k}(\mm_1;q)c_{\nu}(\mm_2;q),
   \end{equation} 
   where \( \Sort(\mu\cup\nu) \) denotes the partition obtained by sorting
   \( (\mu_1,\dots,\mu_{\ell(\mu)},\nu_1,\nu_2,\dots, \nu_{\ell(\nu)})
   \) in weakly decreasing order. Note that if \( \lambda=\Sort(\mu\cup\nu) \),
   then
   \( \prod_{i\ge1}[\lambda_i]_q! = \prod_{i\ge1}[\mu_i]_q!
   \prod_{i\ge1}[\nu_i]_q! \). Therefore, \eqref{eq:5} is equivalent
   to
  \begin{equation}\label{eq:43}
    \sum_{T\in \SYT_k(\lambda)}p_\mm(T;q)
    =
\sum_{\substack{\mu\vdash n_1, \nu \vdash n_2 \\ 
         \Sort(\mu \cup \nu)=\lambda}}
    \sum_{S\in \SYT_k(\mu)}
    p_{\mm_1}(S;q) 
     \sum_{R \in \SYT(\nu)}p_{\mm_2}(R;q). 
  \end{equation}
 
  Now let \( \mathcal{A} \) be the set of pairs
  \( (S,R)\in \SYT_k(\mu)\times \SYT(\nu) \) for some
  \( \mu\vdash n_1 \) and \( \nu\vdash n_2 \) such that
  \( \Sort(\mu\cup\nu) = \lambda \), \( p_{\mm_1}(S;q)\neq 0 \) and
  \( p_{\mm_2}(R;q)\neq 0 \). Let \( \mathcal{B} \) be the set of
  \( T\in\SYT_k(\lambda) \) such that \( p_\mm(T;q)\neq 0 \).
  In order to prove \eqref{eq:43}, it suffices to find a bijection
  \( \phi : \mathcal{A} \rightarrow \mathcal{B} \) such that if \( \phi(S,R)=T \),
  then \( p_\mm(T;q)=p_{\mm_1}(S;q)p_{\mm_2}(R;q) \).

  For \( (S,R)\in \mathcal{A} \), we define \( \phi(S,R)=T \), where the
  \( j \)-th column of \( T \) is obtained by placing the \( j \)-th
  column of \( S \) under the \( j \)-th column of \( R \) for each
  \( j \). More precisely, \( T\in \SYT(\Sort(\mu\cup \nu)) \) is given by
  \[
    T(i,j):=
    \begin{cases}
      R(i,j) & \mbox{if \( 1\le i\le \nu'_j \)},\\
      S(i-\nu'_j,j)+n_2 & \mbox{if \( \nu'_j+1\le i\le \nu'_j+\mu'_j \)}.
    \end{cases}
  \]
  For example, see Figure~\ref{fig:exam_multi}. Let \( T_{\le i} \) denote
  the tableau obtained from \( T \) by removing all entries greater than
  \( i \). By definition, we have \( T_{\le n_2}=R \), and \( T \) must be in \( \SYT_k(\lambda) \) as \( S\in \SYT_k(\mu) \).

  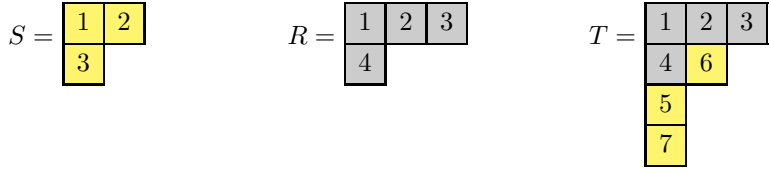
\begin{figure}
  \centering
  \begin{tikzpicture}[scale=1]
   \node[below] at (0,0) { \( S= \)
      \begin{ytableau}
        *(yellow!70) 1 & *(yellow!70) 2\\
        *(yellow!70) 3
      \end{ytableau}
    };
  \node[below] at (4,0) { \( R= \)
      \begin{ytableau}
        *(gray!40) 1 & *(gray!40) 2 & *(gray!40) 3\\
        *(gray!40) 4
      \end{ytableau}
    };
   \node[below] at (8,0) { \( T= \)
      \begin{ytableau}
        *(gray!40) 1 & *(gray!40) 2 & *(gray!40) 3\\
        *(gray!40) 4 & *(yellow!70) 6\\
        *(yellow!70) 5\\
        *(yellow!70) 7
      \end{ytableau}
    };  
  \end{tikzpicture}
  
  \caption{The map \( \phi(S,R)=T \) for \( \mm_1=(2,3,3) \) and \( \mm_2=(3,4,4,4) \).}
  \label{fig:exam_multi}
  \end{figure} 
  By the construction, we have
  \begin{equation}\label{eq:32}
    p_{\mm_2}(R;q)
    =\prod_{i=1}^{n_2}\psi_{k_i}^{(n_2-\mm_2(n_2+1-i))}(R_{\le i-1};q)
    =\prod_{i=1}^{n_2}\psi_{k_i}^{(n-\mm(n+1-i))}(T_{\le i-1};q),
  \end{equation}
  where \( k_i \) is the integer such that \( T_{\le i}=f_{k_i}^{(n-\mm(n+1-i))}(T_{\le i-1}) \).
  Furthermore, for each \( n_2+1\le i\le n_1+n_2 \), since \( n-\mm(n+1-i)\geq n_2 \),
  by the definition of \( \vec\delta^{(r)}(T) \),
  we have
  \begin{equation}\label{eq:12}
    \boldsymbol{\delta}^{(n-\mm(n+1-i))}(T_{\le i-1};q) 
    =\left(1^{b_0}, 0^{a_1}, 1^{b_1}, \ldots, 0^{a_l}, 1^{b_l}, 0^{a_{l+1}+n_2}\right)
  \end{equation}
  where
  \begin{equation}\label{eq:13}
    \boldsymbol{\delta}^{(n_1-\mm_1(n_1+1-(i-n_2)))}(S_{\le i-n_2-1};q)
    =\left(1^{b_0}, 0^{a_1}, 1^{b_1}, \ldots, 0^{a_l}, 1^{b_l}, 0^{a_{l+1}}\right).
  \end{equation}
  Thus,
  \begin{equation}\label{eq:33}
    p_{\mm_1}(S;q)
    =\prod_{i=n_2+1}^{n}\psi_{k_{i}}^{(n-\mm(n+1-i))}(T_{\le i-1};q),
  \end{equation}
  and \( n_2+i \) is in column \( j \) in \( T \) if and only if \( i \)
  in column \( j \) is in \( S \).

  By \eqref{eq:32} and \eqref{eq:33}, we have
  \( p_\mm(T;q)=p_{\mm_1}(S;q)p_{\mm_2}(R;q) \). By \eqref{eq:12} and
  \eqref{eq:13}, \( \phi:\mathcal{A} \to \mathcal{B} \) is a desired
  bijection, which completes the proof.
\end{proof}

\subsection{Multiplicativity for \( G \)}
Recall the definition of \( G_{\mm,k}(x;q) \): for
\( \mm\in \HH_n \) and \( 1\le k \le n \),
\[
  G_{\mm,k}(x;q) = e_k(x) g_{\mm,n-k}(x;q),
\]
where
\[
  g_{\mm,n-k}(x;q) := \sum_{\substack{\sigma=\tau_1\cdots\tau_j\in \mathfrak{S}_{n,\mm} \\ |\tau_1|\ge k}}
    (-1)^{|\tau_1|-k} q^{\wt_\mm(\sigma)} h_{|\tau_1|-k}(x) \omega(\rho_{|\tau_2|}(x;q)\cdots\rho_{|\tau_j|}(x;q)),
\]
and \( \mathfrak{S}_{n,\mm} \) is the set of permutations \( \sigma\in \mathfrak{S}_n \)
such that \( \sigma(i) \le \mm(i) \) for each \( i\in[n] \).

\begin{prop}\label{prop:G_multiplicative}
  Let \( \mm \in \HH_n\) such that \( \mm = \mm_1 + \mm_2 \) for some Hessenberg functions
  \( \mm_1 \) and \( \mm_2 \).
  For \( 1\leq k \leq n \), we have
  \[
    G_{\mm,k}(x;q) = G_{\mm_1,k}(x;q) X_{\mm_2}(x;q).
  \]
  Consequently, we have
  \[
    G_{\mm}(x;q) = G_{\mm_1}(x;q) X_{\mm_2}(x;q).
  \]
\end{prop}
\begin{proof}
  Let \( \mm_1\in\HH_{n_1} \) and \( \mm_2\in\HH_{n_2} \), and let
  \( \sigma \) be a permutation in \( \mathfrak{S}_{n,\mm} \). Since
  \( \mm=\mm_1+\mm_2 \), we have \( \sigma(i)\le n_1 \) and
  \( \sigma(j) > n_1 \) for \( i\in[n_1] \) and
  \( j\in \{ n_1+1,\dots,n_1+n_2\} \). Using this fact, we define
  \( \sigma_1\in\mathfrak{S}_{n_1,\mm_1} \) and
  \( \sigma_2\in\mathfrak{S} _{n_2,\mm_2} \) by
  \[
    \sigma_1(i) := \sigma(i) \qand \sigma_2(j) := \sigma(j+n_1)-n_1
  \]
  for \( i\in [n_1] \) and \( j\in [n_2] \). It is obvious that \( \sigma\mapsto (\sigma_1,\sigma_2)
  \) is a bijection between \( \mathfrak{S}_{n,\mm} \) and \( \mathfrak{S}_{n_1,\mm_1}\times\mathfrak{S}
  _{n_2,\mm_2} \), and each cycle of \( \sigma \) naturally corresponds to either a cycle of
  \( \sigma_1 \) or a cycle of \( \sigma_2 \). Furthermore, we have
  \begin{align*}
    \wt_\mm(\sigma)&=|\{(i,j)\in [n]\times [n] : i<j\leq \mm(i),
                    ~j \text{ precedes \( i \) in } \sigma^c \}|\\
                  &=\wt_{\mm_1}(\sigma_1)+
                    |\{(i,j) : n_1+1\leq i<j\leq n_1+\mm_2(i-n_1),
                    ~j \text{ precedes \( i \) in } \sigma^c \}|\\
                  &=\wt_{\mm_1}(\sigma_1)+\wt_{\mm_2}(\sigma_2).
\end{align*}
  Therefore \( G_{\mm,k}(x;q) \) can be written as
  \begin{align*}
    G_{\mm,k}(x;q)
      &=e_k(x)g_{\mm,n-k}(x;q)\\   
      &= e_k(x)
        \sum_{\substack{\sigma_1=\tau_1\cdots\tau_r\in \mathfrak{S}_{n_1,\mm_1}\\ |\tau_1|\ge k}}
    (-1)^{|\tau_1|-k} q^{\wt_{\mm_1}(\sigma_1)} h_{|\tau_1|-k}(x)
    \omega(\rho_{|\tau_2|}(x;q)\cdots\rho_{|\tau_r|}(x;q)) \\
    &\qquad\quad \times \sum_{\sigma_2=\tau'_1\cdots\tau'_s\in \mathfrak{S}_{n_2,\mm_2}} q^{\wt_{\mm_2}(\sigma_2)} \omega(\rho_{|\tau'_1|}(x;q)\cdots\rho_{|\tau'_s|}(x;q)) \\
      &= G_{\mm_1,k}(x;q) \sum_{\sigma_2=\tau'_1\cdots\tau'_s\in \mathfrak{S}_{n_2,\mm_2}} q^{\wt_{\mm_2}(\sigma_2)} \omega(\rho_{|\tau'_1|}(x;q)\cdots\rho_{|\tau'_s|}(x;q)).
  \end{align*}
  Then, it suffices to show that
  \[
    X_{\mm_2}(x;q) = \sum_{\sigma=\tau'_1\cdots\tau'_s\in \mathfrak{S}_{n_2,\mm_2}} q^{\wt_{\mm_2}(\sigma)} \omega(\rho_{|\tau'_1|}(x;q)\cdots\rho_{|\tau'_s|}(x;q)),
  \]
  which is proved in \cite[Theorem~1.2]{Abreu2021b}.
\end{proof}

\subsection{Multiplicativity for \( S \)}
We first extend the definition of \( P \)-tableaux to a skew shape. Let \( \mm\in\HH_n \) and
consider partitions \( \lambda \) and \( \mu \) with \( \mu\subseteq\lambda \) and
\( |\lambda/\mu|=n \). Then a filling \( T \) of \( \lambda/\mu \) with (distinct) integers
in \( [n] \) is a \emph{\( P \)-tableau} with respect to \( \mm \) if \( T(i,j) <_{\mm} T(i,j+1) \)
and \(T(i-1,j) \not >_{\mm} T(i,j)\) whenever the cells are in \( \lambda/\mu \). We denote the
set of \( P \)-tableaux on \( \lambda/\mu \) by \( \PT_{\mm}(\lambda/\mu) \), and
\( \inv_\mm(T) \) for \( T\in \PT_{\mm}(\lambda/\mu) \) is defined in the same way
as in \Cref{def:inv}.

\begin{lem}\label{lem:skew_P_tab}
  Let \( \mm \in \HH_n\) and let \( \mu \) be a partition.
  Then we have
  \[
    s_\mu(x) X_\mm(x;q)
      =\sum_{\substack{\lambda\supseteq\mu \\ |\lambda/\mu|=n}}
        \sum_{T\in\PT_{\mm}(\lambda/\mu)} q^{\inv_{\mm}(T)} s_{\lambda}(x).
  \]
\end{lem}
\begin{proof}
  Using the skewing
  operator \( s_\mu^\perp \) and the Jacobi--Trudi formula for skew
  Schur functions, we have that for a partition \( \lambda \),
  \begin{align*}
    \langle s_\mu(x) X_\mm(x;q), s_\lambda(x) \rangle
      &=  \langle X_\mm(x;q), s_{\lambda/\mu}(x) \rangle \\
      &= \left\langle X_\mm(x;q), \sum_{w\in S_n} (-1)^{\ell(w)} h_{w(\lambda,\mu)}(x) \right\rangle,
  \end{align*}
  where \( w(\lambda,\mu) \in \ZZ^n \) is given by
  \[
    w(\lambda,\mu)_i := \lambda_{w(i)}-\mu_i-w(i)+i,
  \]
  and \( h_\alpha(x) := h_{\alpha_1}(x)\cdots h_{\alpha_n}(x) \) for \( \alpha\in\ZZ^n \). Note that if \( \alpha \) contains a negative entry then we regard \( h_{\alpha}(x)=0 \). 
  Then it is enough to show that
  \[
    \left\langle X_\mm(x;q), \sum_{w\in S_n} (-1)^{\ell(w)} h_{w(\lambda,\mu)}(x) \right\rangle
    = \sum_{T\in\PT_{\mm}(\lambda/\mu)} q^{\inv_{\mm}(T)},
  \]
  which follows from a similar argument used in the proof of \cite[Theorem~6.3]{Shareshian2016}.
\end{proof}

\begin{prop}\label{prop:S_multiplicative}
   Let \( \mm \in \HH_n\)
  such that \( \mm = \mm_1 + \mm_2 \) for some Hessenberg functions
  \( \mm_1 \) and \( \mm_2 \).
  Then we have
  \begin{equation}\label{eq:Z'=Z'Z}
    S_\mm(x;q) = S_{\mm_1}(x;q) X_{\mm_2}(x;q).
  \end{equation}
\end{prop}

\begin{proof}
  Let \(\mm_1\in \HH_{ n_1} \) and \(\mm_2 \in \HH_{ n_2} \). For \( \lambda \vdash n\) and \( T\in\PT'_\mm(\lambda) \), we associate \( T_1 \) given as the collection of cells of \( T \) whose entries belong to
  \( [n_1] \). Since \( i <_\mm j \) for all \( i\in [n_1] \) and
  \( n_1+1\le j\le n_1+n_2 \), we have \( T_1\in\PT'_\mm(\mu) \) for
  some partition \( \mu\subseteq\lambda \). Let \( T_2 \) be the
  tableau obtained from \( T \) by removing \( T_1 \) and subtracting
  each remaining entry by \( n_1 \) so that
  \( T_2\in\PT_{\mm_2}(\lambda/\mu) \). Furthermore, it is elementary to check that
  \( \inv_\mm(T) = \inv_{\mm_1}(T_1)+\inv_{\mm_2}(T_2) \). Therefore,
  we can write \( S_\mm(x;q) \) as follows:
  \begin{align*}
    S_\mm(x;q)
      &= \sum_{\lambda\vdash n_1+n_2}
        \sum_{\substack{\mu\vdash n_1 \\ \mu\subseteq\lambda}} \sum_{T_1\in\PT'_{\mm_1}(\mu)} q^{\inv_{\mm_1}(T_1)}
        \sum_{T_2\in\PT_{\mm_2}(\lambda/\mu)} q^{\inv_{\mm_2}(T_2)} s_\lambda(x) \\
      &= \sum_{\mu\vdash n_1} \sum_{T_1\in\PT'_{\mm_1}(\mu)} q^{\inv_{\mm_1}(T_1)}
        \sum_{\substack{\lambda\vdash n_1+n_2 \\ \lambda\supseteq\mu}} \sum_{T_2\in\PT_{\mm_2}(\lambda/\mu)} q^{\inv_{\mm_2}(T_2)} s_\lambda(x).
  \end{align*}
  Applying Lemma~\ref{lem:skew_P_tab}, we conclude
  \[
    S_\mm(x;q)
      = \sum_{\mu\vdash n_1} \sum_{T_1\in\PT'_{\mm_1}(\mu)} q^{\inv_{\mm_1}(T_1)} s_\mu(x;q) X_{\mm_2}(x;q)
      = S_{\mm_1}(x;q) X_{\mm_2}(x;q),
  \]
  as desired.
\end{proof}

\section{The path case}\label{sec: path case}

The goal of this section is to show the following:

\begin{prop}\label{prop:E=G=S_at_path}
  For \( n\ge 1 \), we have
    \begin{equation}\label{eq:E=G=S_at_path}
      E_{\pp_n}(x;q)=G_{\pp_n}(x;q)=S_{\pp_n}(x;q).
    \end{equation}
    Furthermore, for \( 1\le k \le n \), we have
    \begin{equation}\label{eq:E_k=G_k_at_path}
      E_{\pp_n,k}(x;q) = G_{\pp_n,k}(x;q).
    \end{equation}
\end{prop}

We first prove \eqref{eq:E_k=G_k_at_path} by conducting an explicit
computation for \( E_{\pp_n,k}(x;q) \) in \Cref{sec:path-case-1}. Note
that \eqref{eq:E_k=G_k_at_path} implies
\( E_{\pp_n}(x;q)=G_{\pp_n}(x;q) \). In \Cref{sec:path-case-2}, we
complete the proof of \eqref{eq:E=G=S_at_path} by providing a
combinatorial argument to show \( G_{\pp_n}(x;q)=S_{\pp_n}(x;q) \).

\begin{remark}\label{rem:2}
We note that Abreu and Nigro \cite[Proposition 3.15]{Abreu2023} showed
that the coefficient of \( m_\lambda \) in \( g_{\pp_n,k}(x;q) \) is
the generating function for multi-derangments.
\end{remark}

\subsection{The path case for \( E \) and \( G \)}
\label{sec:path-case-1}

In this subsection, we prove \eqref{eq:E_k=G_k_at_path}. For each
\( 1 \leq k \leq n \), it was shown in
\cite[Corollary~3.11]{Abreu2023} that
\begin{equation}\label{eq:g_k = F|_Lambda^k}
    G_{\mm,k}(x;q) = e_k(x) g_{\mm,n-k}(x;q)
= e_k(x) \sum_{\alpha\vDash n-k} e_\alpha(x) \prod_{i=1}^{\ell(\alpha)}\left([\alpha_i]_q - 1\right).
\end{equation}
We show that \( E_{\pp_n,k}(x;q) \) has the same formula. We start by
defining some terminology.

For \( \mm_1 \in \HH_{n_1} \) and \( \mm_2\in \HH_{n_2}\), we define
\( \mm_1 \vee \mm_2 \) to be the Hessenberg function given by
  \[ 
    \mm_1\vee \mm_2
    :=(\mm_1(1), \dots, \mm_1(n_1-1), \mm_2(1)+n_1-1, \dots, \mm_2(n_2)+n_1-1)
    \in \HH_{n_1+n_2-1}.
  \]
  Note that the graph \( \mm_1 \vee \mm_2 \) is obtained by appending the
  graph \( \mm_1 \) to the beginning of the graph \( \mm_2 \).
  We specifically consider the case when \( \mm_1=(2,2)\in \HH_2\), which corresponds to the operation
  of adding a vertex and joining it to the vertex \( 1 \) of the graph \( \mm_2 \) by an edge.
  In particular, we have \( \pp_{n+1}=(2,2)\vee\pp_{n} \).

  For a partition \( \lambda \), we write \( k\in \lambda \) if
  \( k \) is a part of \( \lambda \), and \( k\not\in \lambda \)
  otherwise.

\begin{lem}\label{lem:path_attach}
  Let \( \mm\in \HH_n \).
  Suppose that \( \lambda \vdash n+1\) is a partition with
  \( k\in \lambda \). Let \( \mu \) be the partition obtained from the
  partition \( \lambda \) by removing a part \( k \) and adding a part
  \( k-1 \). Then
  \begin{equation}\label{eq:path_j}
    \frac{1}{[k]_q} c_{\lambda,k}((2,2)\vee \mm;q)=
    \begin{cases}
      \sum_{j=2}^{n}\frac{[j]_q-1}{[j]_q}c_{\mu,j}(\mm;q)
      & \mbox{if \( k=1 \)},\\
      \frac{1}{[k-1]_q} c_{\mu,k-1}(\mm;q)
      & \mbox{if \( k>1 \)}.
    \end{cases}
  \end{equation}
\end{lem}

\begin{proof} 
  From Theorem~\ref{thm:Hikita_m}, we have
  \[
    c_{\lambda,k}((2,2)\vee \mm;q)
    = \prod_{i\ge1}[\lambda_i]_q!
    \sum_{T \in \SYT_k(\lambda)} p_{(2,2)\vee \mm}(T;q).
  \]
  Consider \( T\in \SYT_k(\lambda) \) such that
  \( p_{(2,2)\vee \mm}(T;q)\ne 0 \). Let
  \( T' \) be the tableau obtained from \( T \) by removing \( n+1 \).
  Suppose that \( n \) is in column \( j \) of \( T' \) so that
  \( \boldsymbol{\delta}^{(n-1)}(T')=(0^{j-1},1,0^{n-j}) \). Let
  \( r= (n+1)-((2,2)\vee \mm)(1) = n-1 \). Since
  \( p_{(2,2)\vee \mm}(T;q)\ne 0 \), we have \( T = f^{(r)}_t(T') \) for
  some \( t\in \{0,1\} \) and
  \begin{equation}\label{eq:44}
    p_{(2,2)\vee \mm}(T;q) = \psi^{(r)}_t(T';q) p_{\mm}(T';q).
  \end{equation}
  Moreover, since \( n+1 \) is in column \( k \), recalling
  \eqref{eq:15}, if \( k=1 \), then \( t=0 \) and
  \( \psi^{(r)}_0(T';q) = \frac{q[j-1]_q}{[j]_q} = \frac{[j]_q-1}{[j]_q} \); whereas if
  \( k>1 \), then \( t=1 \), \( j=k-1 \), and
  \( \psi^{(r)}_1(T';q) = \frac{1}{[k-1]_q} \). Therefore, summing
  \eqref{eq:44} over all such \( T \)'s yields \eqref{eq:path_j}.
\end{proof}

Using \Cref{lem:path_attach}, we prove the following formula for
\( E_{\pp_n,k}(x;q) \). Comparing it with~\eqref{eq:g_k =
  F|_Lambda^k}, we complete the proof of \eqref{eq:E_k=G_k_at_path}.

\begin{prop}\label{prop:path_formula}
  For \( 1\leq k \leq n \), we have
  \begin{equation}\label{eq:3}
    E_{\pp_n,k}(x;q)
    =e_k(x)  \sum_{\alpha \vDash n-k} e_\alpha(x) 
    \prod_{i=1}^{\ell(\alpha)}\left(\left[\alpha_i\right]_q-1\right).
  \end{equation}
\end{prop}

\begin{proof}
  It suffices to show that for
  \( \lambda \vdash n \),
  \begin{equation}\label{eq:path_coeff}
    \frac{c_{\lambda,k}(\pp_n;q)}{[k]_q}=
    \sum_{\substack{\alpha \vDash n-k\\ \Sort(\alpha\cup (k))=\lambda}}
    \prod_{i=1}^{\ell(\alpha)}\left( [\alpha_i]_q-1 \right),
  \end{equation}
  where \( \Sort(\alpha\cup (k)) \) is the partition obtained by
  rearranging \( (\alpha_1, \dots, \alpha_{\ell(\alpha)}, k) \) in
  weakly decreasing order. To show \eqref{eq:path_coeff}, we use
  induction on \( n \). If \( n=1 \), then both sides of
  \eqref{eq:path_coeff} are equal to \( 1 \).

  Let \( n\ge2 \) and assume that \eqref{eq:path_coeff} holds for
  \( n-1 \). Recall that  \( \pp_n=(2,2)\vee \pp_{n-1} \).
  Let \( \lambda \vdash n \) and \( 1\le k\le n \).
  Note that if \( k \not\in \lambda \), then both
  sides of \eqref{eq:path_coeff} are equal to \( 0 \). Thus, we assume
  that \( k\in \lambda \), and let \( \mu \) be the partition obtained
  from \( \lambda \) by removing a part \( k \) and adding a part
  \( k-1 \).

  If \( k=1 \), then by \Cref{lem:path_attach} and the
  induction hypothesis, we have
  \begin{align*}
    c_{\lambda,1}(\pp_n;q)
    &=\sum_{j=2}^{n-1}\frac{[j]_q-1}{[j]_q}c_{\mu,j}(\pp_{n-1};q)\\
    &=\sum_{j=2}^{n-1}([j]_q-1)\sum_{\substack{\alpha\vDash (n-1)-j\\
        \Sort(\alpha \cup (j) )=\mu}}\prod_{i=1}^{\ell(\alpha)}([\alpha_i]_q-1)
    =\sum_{\substack{\alpha\vDash n-1\\ \Sort(\alpha \cup (1))=\lambda}}
    \prod_{i=1}^{\ell(\alpha)}([\alpha_i]_q-1),
  \end{align*}
  as desired. If \( 1<k\leq n \), then, again, by
  \Cref{lem:path_attach} and the induction hypothesis, we have
  \[
        \frac{c_{\lambda,k}(\pp_n;q)}{[k]_q}=\frac{c_{\mu,k-1}(\pp_{n-1};q)}{[k-1]_q}
                                      =\sum_{\substack{\alpha\vDash (n-1)-(k-1)\\ \Sort(\alpha\cup (k-1))=\mu}}
    \prod_{i=1}^{\ell(\alpha)}([\alpha_i]_q-1)
                                      =\sum_{\substack{\alpha\vDash n-k\\ \Sort(\alpha\cup (k))=\lambda}}
    \prod_{i=1}^{\ell(\alpha)}([\alpha_i]_q-1).
  \]
  Therefore, \eqref{eq:path_coeff} holds for all \( n\geq 1 \).
\end{proof}

\subsection{The path case for \( S \)}
\label{sec:path-case-2}

In this subsection, we show the following proposition.

\begin{prop}\label{pro:2}
  We have
\[
  S_{\pp_n}(x;q)=E_{\pp_n}(x;q).
\]
\end{prop}

Summing \eqref{eq:3} over all \( k\in[n] \) gives
  \[
    E_{\pp_n}(x;q)
    =\sum_{\alpha \vDash n} e_\alpha(x)
    \prod_{i=2}^{\ell(\alpha)}\left(\left[\alpha_i\right]_q-1\right).
  \]
  By \eqref{eq:Kostka},
  we have
  \[
    E_{\pp_n}(x;q)
    = \sum_{\lambda\vdash n}  \left( \sum_{\alpha \vDash n} K_{\lambda',\alpha} 
    \prod_{i=2}^{\ell(\alpha)}\left(\left[\alpha_i\right]_q-1\right) \right)
    s_\lambda(x).
  \]
  Hence, by the definition of \( S_{\pp_n}(x;q) \) in
  \eqref{eq:S_def_intro}, to prove \Cref{pro:2}, it suffices to show
  that for all partitions \( \lambda \), we have
  \begin{equation}\label{eq:2}
    \sum_{T\in \PT'_{\pp_{|\lambda|}}(\lambda)} q^{\inv_{\pp_{|\lambda|}}(T)}
    =  \sum_{\alpha \vDash |\lambda|} K_{\lambda',\alpha} 
    \prod_{i=2}^{\ell(\alpha)}\left(\left[\alpha_i\right]_q-1\right).
  \end{equation}
  Let \( L(\lambda) \) and \( R(\lambda) \) denote the left-hand side and
  the right-hand side of \eqref{eq:2}, respectively.

  By the definition of \( R(\lambda) \), we have
  \begin{equation}\label{eq:9}
    R(\lambda) =
    \begin{cases}
      1+\sum_{\lambda/\mu:  \text{ nonempty vertical strip}} ([|\lambda/\mu|]_q-1) R(\mu)
      & \mbox{if \( \lambda=(1^n) \) for some \( n\ge0 \)},\\
      \sum_{\lambda/\mu: \text{ nonempty vertical strip}} ([|\lambda/\mu|]_q-1) R(\mu)
      & \mbox{otherwise.}
    \end{cases}
  \end{equation}
  Note that \eqref{eq:9} completely determines \( R(\lambda) \) for
  all partitions \( \lambda \).
  Hence, it suffices to show that
  \( L(\lambda) \) also satisfies \eqref{eq:9} with \( R \) replaced
  by \( L \). Hence, it remains to show the following proposition.

  \begin{prop}\label{pro:1}
  For any partition \( \lambda \),
  \[
    L(\lambda) =
    \begin{cases}
      1+\sum_{\lambda/\mu: \text{ nonempty vertical strip}} ([|\lambda/\mu|]_q-1) L(\mu)
      & \mbox{if \( \lambda=(1^n) \) for some \( n\ge0 \)},\\
      \sum_{\lambda/\mu: \text{ nonempty vertical strip}} ([|\lambda/\mu|]_q-1) L(\mu)
      & \mbox{otherwise.}
    \end{cases}
  \]
  \end{prop}

  To prove \Cref{pro:1}, we need some lemmas. Recall \Cref{def:inv}.
  Note that for \( i,j\in [n] \) with \( i\ne j \), we have
  \( i<_{\pp_n}j \) if and only if \( j\ge i+2 \), or equivalently,
  \( j\not<_{\pp_n}i \) if and only if \( j\ge i-1 \). Thus, for
  \( T\in \PT'_{\pp_n}(\lambda) \), any consecutive entries in each
  row increase by at least \( 2 \), and any consecutive entries in
  each column either decrease by \( 1 \) or increase. For example, see
  \Cref{fig:2,fig:1} (ignore the colors at this moment).

\begin{lem}\label{lem:PT-prop1}
  Let \( \lambda\vdash n \) and \( T\in \PT'_{\pp_n}(\lambda) \). If
  \( T(i,j)>T(i+d,j) \) for some \( d>0 \), then for all
  \( t\in [d] \), we have \( T(i+t,j) = T(i+t-1,j)-1 \).
\end{lem}

\begin{proof}
This can be proved easily by induction on \( d \).
\end{proof}

\begin{lem}\label{lem:PT-prop2}
  Let \( \lambda\vdash n \) and \( T\in \PT'_{\pp_n}(\lambda) \).
  There are no integers \( r,s\in T \) such that \( r<s \) and \( s \)
  is weakly above \( r \) and strictly to the left of \( r \).
\end{lem}

\begin{proof}
 Suppose that there exist such \( r \) and \( s \). Then \( s \) must
 be strictly above \( r \) because all entries in the row containing
 \( r \) and to the left of \( r \) are smaller than \( r \). Thus,
 \( T(a,b)=s \) and \( T(a',b')=r \) for some \( a<a' \) and
 \( b<b' \). Let \( t=T(a,b') \). Since \( r<s=T(a,b)<T(a,b')=t \), we have
 \( T(a,b')=t>r=T(a',b') \). Then, by \Cref{lem:PT-prop1},
 we have
 \[
   (T(a,b'),T(a+1,b'),\dots,T(a',b')) = (t,t-1,\dots,r).
 \]
This implies that \( s=T(j,b') \) for some \( a\le j\le a' \), which
 is impossible because \( T(a,b)=s \). Therefore, there are no such
 \( r \) and \( s \).
\end{proof}

Now, we are ready to show \Cref{pro:1}.

\begin{proof}[Proof of \Cref{pro:1}]
  Suppose \( \lambda\vdash n \). Let \( S\in \PT'_{\pp_n}((1^n)) \) be
  the tableau given by \( S(i,1)=i \) for \( i\in [n] \). Let
  \( \mathcal{A} \) be the set of pairs \( (T',j) \) of
  \( T'\in \PT'_{\pp_{|\mu|}}(\mu) \) and
  \( 1\le j\le |\lambda/\mu|-1 \) for some partition \( \mu \) such
  that \( \lambda/\mu \) is a vertical strip. We will construct a
  bijection
  \( \phi:\PT'_{\pp_n}(\lambda)\setminus\{ S \} \to \mathcal{A} \)
  such that if \( \phi(T) = (T',j) \), then
  \begin{equation}\label{eq:11}
    \inv_{\pp_n}(T) = \inv_{\pp_{|\mu|}}(T') +j,
  \end{equation}
 where \( \mu \)
  is the shape of \( T' \). Since \( S\in \PT'_{\pp_n}(\lambda) \) if
  and only if \( \lambda=(1^n) \), and \( q^{\inv_{\pp_n}(S)}=1 \), it
  suffices to construct such a map.

  Let \( T\in \PT'_{\pp_n}(\lambda)\setminus\{ S \} \). Let \( m \) be
  the largest entry of \( T \) such that \( m-1 \) is below \( m \).
  Such \( m \) always exists, by the definition of
  \( \PT'_{\pp_n}(\lambda) \) and the fact that \( T\neq S \). By
  \Cref{lem:PT-prop2}, the entries \( n,n-1,\dots,m+1 \) form a
  vertical strip in this order from bottom to top, and \( m \) is
  strictly above and weakly to the right of \( m+1 \).

  We claim that there exists an integer \( 2\le \ell\le m \) such that
  \begin{equation}\label{eq:10}
    n,n-1,\dots,m+2,m+1, \ell,\ell+1,\dots,m-1,m
  \end{equation}
  form a vertical strip, say \( \alpha \), in this order from bottom
  to top. In other words, \( \alpha \) satisfies the two conditions:
  no two cells of \( \alpha \) lie in the same row and
  \( \lambda/\alpha \) is a partition.

  If \( m \) and \( m+1 \) are in different columns, then
  \( n,n-1,\dots,m+1,m \) form a vertical strip, hence \( \ell=m \)
  satisfies the conditions. Suppose that \( m \) and \( m+1 \) are in
  the same column, say \( T(a,b)=m \) and \( T(a+t,b)=m+1 \). Since
  \( T(a+i,b)<T(a,b) \) for \( i\in [t-1] \), by applying
  \Cref{lem:PT-prop1} to the cells
  \( (a,b),(a+1,b),\dots,(a+t-1,b) \), we obtain \( T(a+i,b)=m-i \)
  for \( i\in [t-1] \). Since the entries in each row of \( T \) are
  strictly increasing, there are no cells to the right of any integer
  in \( \alpha \) in the same row. Moreover, by \Cref{lem:PT-prop1},
  there are no cells in \( \lambda\setminus\alpha \) that is below any
  cell of \( \alpha \) in the same column. Hence, \( \alpha \) is a
  vertical strip. Thus, \( \ell=m-t+1 \) satisfies the conditions, and
  the claim holds. See \Cref{fig:2}. Note that since \( T(1,1)=1 \),
  we must have \( \ell\ge2 \).

 \begin{figure}
    \centering
\begin{ytableau}
1 & 3 & 8 & *(gray!30)11\\
2 & 4 & 10\\
5 & 7 & *(gray!30)12\\
6 & 9\\
*(gray!30)13
\end{ytableau}
\qquad \qquad 
\begin{ytableau}
1 & 3 & *(gray!30)11\\
2 & 4 & 10\\
5 & 7 & 9\\
6 & 8 & *(gray!30)12\\
*(gray!30)13
\end{ytableau}
\caption{A \( P \)-tableau \( T_1\in \PT'_{\pp_n}((4,3,3,2,1)) \) on
  the left and \( T_2\in \PT'_{\pp_n}((3,3,3,3,1)) \) on the right,
  where \( n=13 \). For both \( T_1 \) and \( T_2 \), we have \( m=11 \),
  and the cells with entries \( n,n-1,\dots,m \) are colored gray. In
  \( T_1 \), the entries \( m \) and \( m+1 \) are in different
  columns, hence \( n,n-1,\dots,m \) form a vertical strip. In
  \( T_2 \), the entries \( m \) and \( m+1 \) are in the column, and
  \( n,n-1,\dots,m \) do not form a vertical strip. However,
  \( 13=n,n-1,\dots,m-t=9 \) form a vertical strip.}
  \label{fig:2}
\end{figure}

  Now, let \( \ell \) be the smallest integer such that \eqref{eq:10}
  is a vertical strip, say \( \alpha \). We define
  \( \phi(T) = (T',j) \), where \( T' \) is the tableau obtained from
  \( T \) by removing the cells containing the integers in
  \eqref{eq:10} and
  \[
    j = |\{i\in \{ \ell, \dots, m \}: i-1 \text{ is below } i \text{ in } T\}|
    = \begin{cases}
    m-\ell+1  & \mbox{if \( \ell-1 \) is below \( \ell \) in \( T \)},\\
    m-\ell  & \mbox{otherwise.}
    \end{cases}
  \]
  Note that by construction, \eqref{eq:11} holds. See \Cref{fig:1}.

  \begin{figure}
    \centering
\begin{ytableau}
1 & 4 & 7 & 12 & *(red!20)18\\
3 & 5 & 10 & 14 & *(red!20)17\\
2 & 9 & 13\\
6 & 8 & *(red!20)16\\
11 & \textbf{15} & *(yellow!70)19\\
*(yellow!70)20\\
*(yellow!70)21
\end{ytableau} \qquad \qquad 
\begin{ytableau}
1 & 4 & 7 & 12 & *(red!20)18\\
3 & 5 & 10 & 14 & *(red!20)17\\
2 & 9 & \textbf{15}\\
6 & 8 & *(yellow!70)16\\
11 & 13 & *(yellow!70)19\\
*(yellow!70)20\\
*(yellow!70)21
\end{ytableau}
\caption{A \( P \)-tableau \( T_1\in \PT'_{\pp_n}(\lambda) \) on the
  left and \( T_2\in \PT'_{\pp_n}(\lambda) \) on the
  right, where \( n=21 \) and \( \lambda=(5,5,3,3,3,1,1) \). For both
  \( T_1 \) and \( T_2 \), we have \( m=18 \) and \( \ell=16 \). The
  cells with an entry \( \ell \le i\le n \) are colored red or
  yellow, depending on whether \( i-1 \) is below \( i \) or not. The
  integer \( j \) is the number of red cells, which is \( 3 \) for
  \( T_1 \) and \( 2 \) for \( T_2 \).}
    \label{fig:1}
  \end{figure}

  We first show that \( (T',j)\in \mathcal{A} \), that is,
  \( \lambda\setminus \alpha \) is a partition and
  \( 1\le j \le |\alpha|-1 \).

  Since \( m-1 \) is below \( m \), we have \( j\ge1 \). By
  the definition of \( j \), we have
  \( j\le m-\ell+1 \le n-\ell+1 = |\alpha| \). Thus, it suffices to
  show that \( j\ne |\alpha| \). If \( j=|\alpha| \), then
  \( \ell-1,\ell, \ell+1,\dots,n \) form a vertical strip in this
  order from bottom to top, which is a contradiction to the minimality
  of \( \ell \). Thus, \( 1\le j\le |\alpha|-1 \), and we obtain
  \( (T',j)\in \mathcal{A} \).

  It remains to show that \( \phi \) is a bijection. To do this,
  consider an arbitrary element \( (T',j)\in \mathcal{A} \). Let
  \( \mu \) be the shape of \( T' \) and \( \alpha= \lambda/\mu \).
  Let \( c \) be the lowest cell of \( \alpha \) that is in or above
  the row containing the integer \( |\mu| \) in \( T' \). If there is
  no such cell, we set \( c \) to be the topmost cell of
  \( \lambda \). We label the cells in \( \alpha\setminus\{c\} \) with
  \( 1,2,\dots,|\alpha|-1 \) from top to bottom. Then we define
  \( T \) as the tableau obtained from \( T' \) by filling the cells
  in \( \alpha \) with \( |\mu|+1,|\mu|+2, \ldots, n \) in this order
  as follows: fill the cell labeled \( j \) and all cells above it
  from bottom to top, and then fill the remaining cells from top to
  bottom. It is easy to check that the map \( (T',j)\mapsto T \) is
  the inverse of \( \phi \), which completes the proof.
  \end{proof}

\section{Sink theorem}\label{sec: sink}
In this section, we give an analog of Stanley's sink
theorem~\cite[Theorem~3.3]{Stanley1995} for \( S_\mm(x;q) \).

For a graph \( G \) on \( [n] \), an \emph{acyclic orientation} of
\( G \) is an assignment of a direction to each edge such that the
orientation induces no directed cycles. We denote by \( \AO_G \) the
set of all acyclic orientations of \( G \). For an acyclic orientation
\( \theta \) of \( G \), a vertex \( v \) of \( G \) is called a
\emph{sink} of \( \theta \) if the direction of each edge incident
with \( v \) is toward \( v \). We let \( \sink(\theta) \) be the
number of sinks of \( \theta \). Since the graph \( G \) is finite, we
have \( \sink(\theta)\ge 1 \) (we assume \( n\ge1 \)). In addition, we
define
\[
  \asc_G(\theta) = |\{ (i,j)\in E(G) : \mbox{\( i<j \) and the direction in \( \theta \) is toward
  \( j \)} \} |.
\]

In \cite{Stanley1995}, Stanley provides a relationship between the chromatic symmetric function
of a graph \( G \) and acyclic orientations of \( G \). Moreover, Shareshian and Wachs generalized
the connection to the chromatic quasisymmetric functions in \cite{Shareshian2016}.
\begin{thm}[\cite{Stanley1995,Shareshian2016}]\label{thm:sink_theorem}
  For a Hessenberg function \( \mm\in \HH_n \), let
  \[
    X_\mm(x;q) = \sum_{\lambda\vdash n} c_\lambda(q) e_\lambda(x).
  \]
  Then we have that for \( \ell\ge 1 \),
  \[
    \sum_{\ell(\lambda)=\ell} c_\lambda(q)
      = \sum_{\substack{\theta\in\AO_\mm \\ \sink(\theta)=\ell}} q^{\asc_\mm(\theta)}.
  \]
\end{thm}
In fact, Stanley proved the theorem for an arbitrary graph \( G \) when \( q=1 \).
However, the chromatic quasisymmetric function of \( G \) is generally not symmetric, and thus
the \( e \)-expansion of \( X_G(x; q) \) cannot be considered in general.
For this reason, Shareshian and Wachs focused only on natural unit interval graphs.

We will give an analog of \Cref{thm:sink_theorem} for \( S_\mm(x;q) \).
For \( \mm\in \HH_n \), let \( \AO'_\mm \) be the set of acyclic orientations
\( \theta \) of the graph \( \mm \) such that the vertex \( 1 \) is a sink of \( \theta \).
\begin{thm}\label{thm:sink'_theorem}
  For a Hessenberg function \( \mm\in \HH_n \), let
  \[
    S_\mm(x;q) = \sum_{\lambda\vdash n} c'_\lambda(q) e_\lambda(x).
  \]
  Then for \( \ell\ge 1 \), we have 
  \[
    \sum_{\ell(\lambda)=\ell} c'_\lambda(q)
      = \sum_{\substack{\theta\in\AO'_\mm \\ \sink(\theta)=\ell}} q^{\asc_\mm(\theta)}.
  \]
\end{thm}
Our proof of the theorem follows the idea of Chow's proof of \Cref{thm:sink_theorem} in
\cite{Chow0}.
First, we present a lemma dealing with the \( e \)-coefficients and \( s \)-coefficients of a given
symmetric function.
\begin{lem} \label{lem:sum_c_eq_sum_a}
  Let \( f(x) \) be a symmetric function in \( \Lambda_n \).
  Suppose that
  \[
    f(x) = \sum_{\lambda\vdash n} c_\lambda e_\lambda(x)
      = \sum_{\lambda\vdash n} a_\lambda s_\lambda(x).
  \]
  Then for \( \ell\ge 1 \), we have
  \[
    \sum_{\ell(\lambda)=\ell} c_\lambda
      = \sum_{k\ge \ell} (-1)^{k-\ell} \binom{k-1}{\ell-1} a_{(k, 1^{n-k})}.
  \]
\end{lem}
\begin{proof}
  We define a \( \QQ(q) \)-algebra homomorphism \( \zeta \) from \( \Lambda \) to \( \QQ(q)[t]
  \) given by
  \[
    \zeta: e_i(x) \mapsto t,
  \]
  so that
  \begin{equation} \label{eq:zeta=e}
    \zeta(f(x)) = \sum_{\ell\ge 1} \sum_{\ell(\lambda)=\ell} c_\lambda t^\ell.
  \end{equation}
  Recall the dual Jacobi--Trudi identity:
  \[
    s_\lambda(x) = \det
    \begin{pmatrix}
      e_{\lambda'_1}(x) & e_{\lambda'_1+1}(x) & \dots & e_{\lambda'_1+m-1}(x) \\
      e_{\lambda'_2-1}(x) & e_{\lambda'_2}(x) & \dots & e_{\lambda'_2+m-2}(x) \\
      \vdots & \vdots & \ddots & \vdots \\
      e_{\lambda'_m-m+1}(x) & e_{\lambda'_m-m+2}(x) & \dots & e_{\lambda'_m}(x)
    \end{pmatrix},
  \]
  where \( m = \lambda_1 \), \( e_0(x) = 1 \) and \( e_k(x)=0 \) for \( k < 0 \).
  Now consider \( \zeta(s_\lambda(x)) \). If
\( \lambda'_2-1 > 0 \),
then the first and second rows of the matrix are identical under the
map \( \zeta \), and thus \( \zeta(s_\lambda(x)) = 0 \). In addition,
we claim that if \( \lambda=(k, 1^{n-k}) \) for some \( k\ge 1 \),
then
  \[
    \zeta(s_\lambda(x)) = \sum_{i=1}^k (-1)^{k-i} \binom{k-1}{i-1} t^i.
  \]
  The claim can be easily verified by induction on \( k \), so we omit the detail.
  Once the claim holds, we have
  \begin{equation} \label{eq:zeta=s}
    \zeta(f(x)) = \sum_{k\ge 1} \sum_{i=1}^k (-1)^{k-i} \binom{k-1}{i-1} a_{(k,1^{n-k})} t^i,
  \end{equation}
  and comparing the coefficients in \eqref{eq:zeta=e} and \eqref{eq:zeta=s} completes the proof.
\end{proof}

Next, we observe a relationship between acyclic orientations and
\( P \)-tableaux with respect to \( \mm\in \HH_n \).
Let \( T \) be a \( P \)-tableau. Then we can obtain
an acyclic orientation from \( T \) as follows. Let \( (i,j) \) be
an edge in the graph \( \mm \). We assign a direction to the edge from
\( i \) toward \( j \) if \( j \) lies above \( i \) in \( T \); otherwise,
assign the opposite direction. Note that \( i \) and \( j \) cannot
lie on the same row because they are incomparable. For example, let
\( \mm=(2,4,4,5,5) \), and
\[
  T = \vcenter{\hbox{\ytableaushort{15,3,4,2}}}.
\]
Then the corresponding orientation is
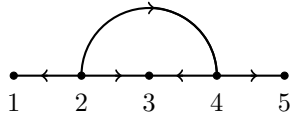
\[
  \begin{tikzpicture}[scale=0.9]
    \foreach \i in {1,...,5}
    \filldraw (\i,2) circle (1.5pt);
    \foreach \i in {1,...,5}
    \node at (\i,1.6) {\i}; 
    \draw[thick] [->] (2,2) -- (1.4,2);
    \draw[thick] [->] (2,2) -- (2.6,2);
    \draw[thick] [->] (4,2) -- (3.4,2);
    \draw[thick] [->] (4,2) -- (4.6,2);
    \draw[thick] (1.5,2) edge (1,2) (2.5,2) edge (3,2) (3.5,2) edge (3,2) (4.5,2) edge (5,2);
    \draw[-<,thick] (4,2) arc[radius=1, start angle=0, end angle=91];
    \draw[thick] (4,2) arc[radius=1, start angle=0, end angle=180];
  \end{tikzpicture}
\]
By construction, the orientation is acyclic. We denote the orientation by \( \theta(T) \).
Note that if \( i \) lies on the first row of \( T \), then the vertex \( i \) is a sink in
\( \theta(T) \), and hence we deduce that if \( T\in\PT'_\mm(\lambda) \) for some
\( \lambda \vdash n \), then \( \theta(T)\in\AO'_\mm \).
Furthermore, one can easily check that \( \inv_\mm(T) = \asc_\mm(\theta(T)) \).

Note that if \( \theta\in\AO_\mm \), then \( \theta \) has at least
one sink. Moreover, there is no edge between any two sinks, so all
sinks are comparable under the partial order \( <_\mm \). Therefore,
there exists a smallest sink of \( \theta \) under \( <_\mm \), which
we will simply call the \emph{smallest sink} of \( \theta \).

\begin{lem} \label{lem:acyclic_orientation_PTab}
  For a Hessenberg function \( \mm \in \HH_n \),
  let \( \theta \) be
  an acyclic orientation of the graph \( \mm \) with \( \ell \) sinks such that
  the vertex \( 1 \) is a sink. Then for \( i\in [\ell] \), we have
  \[
    |\{T\in\PT'_\mm((i, 1^{n-i})) : \theta(T) = \theta \}| = \binom{\ell-1}{i-1}.
  \]
\end{lem}
\begin{proof}
  Let \( A \) be a set of \( i \) sinks in \( \theta \) with
  \( 1\in A \). Then we can write
  \( A=\{a_1 <_\mm \dots <_\mm a_i\} \), where \( a_1 = 1 \). Since
  there are \( \binom{\ell-1}{i-1} \) ways to choose \( A \), it
  suffices to find a correspondence between such sets \( A \) and the
  \( P \)-tableaux \( T\in\PT'_\mm((i, 1^{n-i})) \) satisfying
  \( \theta(T)=\theta \).

  To construct \( T \) from \( A \), we enumerate \( [n]\setminus A \)
  as follows. First, let \( B=[n]\setminus A \), and consider the
  acyclic orientation \( \theta_B \) obtained by restricting
  \( \theta \) to \( B \). Then we set \( b_1 \) to be the smallest
  sink of \( \theta_B \). Remove \( b_1 \) from \( B \), and set
  \( b_2 \) similarly. By iterating this procedure until
  \( B=\emptyset \), we obtain a list \( b_1,\dots, b_{n-i} \) of the
  elements in \( [n]\setminus A \). We claim that for
  \( 1\le j < n-i \), we have \( b_j \not>_\mm b_{j+1} \). Indeed, if
  \( b_j >_\mm b_{j+1} \), then the vertices \( b_j \) and
  \( b_{j+1} \) are not adjacent in the graph \( \mm \). Then the
  vertex \( b_{j+1} \) is a sink of
  \( \theta_{[n]\setminus(A\cup\{b_1,\dots,b_{j-1} \} )} \), which
  implies \( b_j <_\mm b_{j+1} \) by the definition of \( b_j \). This
  is a contradiction, so we must have \( b_j \not>_\mm b_{j+1} \).
  Now, we set
  \[
    T = \vcenter{\hbox{
      \begin{ytableau}
        \scalebox{0.9}{\( 1 \)} & \scalebox{0.9}{\( a_2 \)}&
        \scalebox{0.9}{\( \cdots \)} & \scalebox{0.9}{\( a_i \)} \\
        \scalebox{0.9}{\( b_1 \)} \\
        \scalebox{0.9}{\( \vdots \)} \\
        \scalebox{0.8}{\( b_{n-i} \)}
      \end{ytableau}
    }}.
  \]
  By the construction of \( a_j \)'s and \( b_j \)'s, \( T \) is a \( P \)-tableau of shape
  \( (i,1^{n-i}) \) such that \( \theta(T) = \theta \).

  Now, let \( T\in\PT'_{\mm}((i,1^{n-i})) \) such that
  \( \theta(T) = \theta \). We show that \( T \) can be obtained by
  the above construction. Note that the elements in the first row of
  \( T \) are sinks of \( \theta \). For \( 1\le j \le n-i \), let
  \( b_j=T(j+1,1) \) be the element in the \( (j+1) \)-st row of
  \( T \) so that \( b_1\not>_\mm \dots \not>_\mm b_{n-i} \), and let
  \( B_j = \{b_j,\dots,b_{n-i}\} \). Since \( \theta(T) = \theta \),
  the vertex \( b_j \) is a sink of \( \theta_{B_j} \). Moreover, we
  claim that \( b_j \) is the smallest sink of \( \theta_{B_j} \). For
  a contradiction, suppose that there exists the largest integer
  \( j_0 \) for which the claim is false. Let \( b_k \) be the
  smallest sink of \( \theta_{B_{j_0}} \). Then \( j_0<k\le n-i \) and
  \( b_{k} <_\mm b_{j_0} \). If \( k = j_0+1 \), then
  \( b_{j_0} >_\mm b_{j_0+1} \), which is a contradiction. Therefore
  \( k > j_0+1 \). By the maximality of \( j_0 \), the vertex
  \( b_{j_0+1} \) is the smallest sink of \( \theta_{B_{j_0+1}} \), so
  we have \( b_{j_0+1}<_\mm b_k \). Then
  \( b_{j_0+1}<_\mm b_k <_\mm b_{j_0} \), which is a contradiction to
  \( b_{j_0} \not>_\mm b_{j_0+1} \). Thus the claim is true, which
  completes the proof.
\end{proof}

\begin{proof}[Proof of \Cref{thm:sink'_theorem}]
  \Cref{thm:E=G=S_intro} says that
  \[
    \sum_{\lambda\vdash n} c'_\lambda(q) e_\lambda(x)
      = \sum_{\lambda\vdash n} \sum_{T\in\PT'_\mm(\lambda)} q^{\inv_\mm(T)} s_\lambda(x).
  \]
  Hence applying \Cref{lem:sum_c_eq_sum_a}, we have that for \( \ell\ge 1 \),
  \begin{equation} \label{eq:sum_c=sum_binom_PT}
    \sum_{\ell(\lambda)=\ell} c'_\lambda(q) =
      \sum_{k\ge \ell} (-1)^{k-\ell} \binom{k-1}{\ell-1}
        \sum_{T\in \PT'_\mm((k,1^{n-k}))} q^{\inv_\mm(T)}.
  \end{equation}
  Since \( \theta(T) \) has at least \( k \) sinks for \( T\in \PT'_\mm((k,1^{n-k})) \),
  we can decompose \( \PT'_\mm((k,1^{n-k})) \) as
  \[
    \PT'_\mm((k,1^{n-k})) = \bigsqcup_{\substack{\theta\in\AO'_\mm \\ \sink(\theta)\ge k}}
    \{T\in \PT'_\mm((k,1^{n-k})) : \theta(T) = \theta \}.
  \]
  Then, by applying \Cref{lem:acyclic_orientation_PTab} with the decomposition to the right-hand
  side of \eqref{eq:sum_c=sum_binom_PT}, we deduce
  \begin{align*}
    \sum_{\ell(\lambda)=\ell} c'_\lambda(q)
      &= \sum_{k\ge \ell} (-1)^{k-\ell} \binom{k-1}{\ell-1}
          \sum_{\substack{\theta\in\AO'_\mm \\ \sink(\theta)\ge k}}
            \binom{\sink(\theta)-1}{k-1} q^{\asc_\mm(\theta)}, \\
      &= \sum_{\substack{\theta\in\AO'_\mm \\ \sink(\theta)\ge \ell}}
          \sum_{k=\ell}^{\sink(\theta)} (-1)^{k-\ell} \binom{\sink(\theta)-1}{k-1}
            \binom{k-1}{\ell-1} q^{\asc_\mm(\theta)} \\
      &= \sum_{\substack{\theta\in\AO'_\mm \\ \sink(\theta)= \ell}} q^{\asc_\mm(\theta)},
  \end{align*}
  since \( \sum_{i=b}^{a} (-1)^{i-b} \binom{a}{i} \binom{i}{b} = \delta_{a,b} \).
  This completes the proof.
\end{proof}

\section*{Acknowledgment}
The authors are grateful to Alex Abreu and Antonio Nigro for insightful discussions on their function \( g \) and the underlying geometry. This research was initiated during the workshop `Combinatorics on Flag Varieties and Related Topics 2025', and the authors want to thank the organizers for creating a stimulating and productive environment.

\appendix

\section{Comparison with Hikita's model}
\label{sec:comp-with-hikit}

In this section, we compare our definition of
\( \psi_k^{(r)}(T ; q) \) and Hikita's \( \varphi_k^{(r)}(T ; q) \).

Let
  \( \vec \delta =\left(1^{b_0}, 0^{a_1}, 1^{b_1}, \ldots, 0^{a_l},
    1^{b_l}, 0^{a_{l+1}}\right) \), where
  \( a_1,\dots,a_l, b_1,\dots,b_l \) are positive integers, and
  \( b_0 \) and \( a_{l+1} \) are nonnegative integers. For
  \( 0\le k\le l \), we define
  \[
    \varphi_k(\vec\delta)=q^{\sum_{i=1}^k a_i} \prod_{i=1}^k
    \frac{\left[\sum_{j=i+1}^k a_j+\sum_{j=i}^k
      b_j\right]_q}{\left[\sum_{j=i}^k a_j+\sum_{j=i}^k
      b_j\right]_q} \prod_{i=k+1}^l \frac{\left[\sum_{j=k+1}^i
      a_j+\sum_{j=k+1}^{i-1} b_j\right]_q}{\left[\sum_{j=k+1}^i
      a_j+\sum_{j=k+1}^i b_j\right]_q} ,
  \]
and
\begin{equation}\label{eq:29}
  \psi_k(\vec\delta)= q^{\sum_{i=k+1}^l b_i-\sum_{i=1}^k a_i} \varphi_k(\vec\delta).
\end{equation}
Hikita \cite{Hikita2024} showed that \( \varphi_k(\vec\delta) \) can
be understood as a probability in the following sense.

\begin{lem} \cite[Lemma~6]{Hikita2024}
  \label{lem:Hikita}
  Let
  \( \vec \delta =\left(1^{b_0}, 0^{a_1}, 1^{b_1}, \ldots, 0^{a_l},
    1^{b_l}, 0^{a_{l+1}}\right) \), where
  \( a_1,\dots,a_l,a_{l+1}, b_1,\dots,b_l \) are positive integers and
  \( b_0 \) is a nonnegative integer. Then,
  \[
    \sum_{k=0}^l\varphi_k(\boldsymbol{\delta})=1.
  \]
\end{lem}

Note that, for any sequence
\( \boldsymbol{\delta} =\left(1^{b_0}, 0^{a_1}, 1^{b_1}, \ldots,
  0^{a_l}, 1^{b_l}, 0^{a_{l+1}}\right) \), we have
\(
\psi_k(\boldsymbol{\delta})=\varphi_{l-k}(\overline{\boldsymbol{\delta}})\),
where
\( \overline{\boldsymbol{\delta}}= (1^{a_{l+1}}, 0^{b_l},
1^{a_l},\dots,0^{b_0}) \). Thus, \Cref{lem:Hikita} implies the
following.

\begin{lem}
  \label{lem:modified_p}
  Let
  \( \vec \delta =\left(1^{b_0}, 0^{a_1}, 1^{b_1}, \ldots, 0^{a_l},
    1^{b_l}, 0^{a_{l+1}}\right) \), where
  \( a_1,\dots,a_l, b_1,\dots,b_l \) are positive integers, and
  \( b_0 \) and \( a_{l+1} \) are nonnegative integers. Then,
  \[
    \sum_{k=0}^l \psi_k(\boldsymbol{\delta})=1.
  \]
\end{lem}

\begin{defn}\label{defn:Hikita_pT}
  Let \( \mm\in \HH_n \) and \( T\in \SYT(n) \). Let
  \( T'\in\SYT(n-1) \) be the tableau obtained from \( T \) by
  removing the entry \( n \). Let \( \mm'\in \HH_{n-1} \) be
  the Hessenberg function given by \( \mm'(i) = \mm(i+1)-1 \) for
  \( i\in [n-1] \) and let \( r= n-\mm(1) \). We define
  \( \overline{p}_\mm(T;q) \) recursively by \( \overline{p}_\emptyset(\emptyset;q)=1 \) for
  \( n=0 \) and
  \[
    \overline{p}_\mm(T;q)=
    \begin{cases}
       \varphi_{k}^{(r)}(T';q) \overline{p}_{\mm'}(T';q)
      & \mbox{if \( T=f_{k}^{(r)}(T') \) for some \( 0\le k\le l \),}\\
      0 & \mbox{otherwise,}
    \end{cases}
  \]
  where
  \[
    \varphi_k^{(r)}(T' ; q)= \varphi_k(\vec\delta^{(r)}(T')).
  \]
\end{defn}

\begin{lem}\label{lem:1}
  Let \( \mm \in \HH_n \) and \( T\in \SYT(n) \) such that 
  \( \overline{p}_\mm(T;q)\ne0 \). Let \( T'\in\SYT(n-1) \) be the tableau
  obtained from \( T \) by removing \( n \) and suppose
  \( \vec \delta^{(r)}(T') =\left(1^{b_0}, 0^{a_1}, 1^{b_1}, \ldots,
    0^{a_l}, 1^{b_l}, 0^{a_{l+1}}\right) \), where \( r= n-\mm(1) \).
  Then the following hold:
  \begin{enumerate}
  \item Every column of \( T' \) contains at most one integer greater than \( r \).
  \item We have \( \sum_{i=0}^l b_i = \mm(1)-1 \).
  \end{enumerate}
\end{lem}

\begin{proof}
  For the first statement, we use induction on \( n \), where the
  cases \( n=0 \) and \( n=1 \) are trivial. Suppose \( n\ge2 \) and
  the statement holds for \( n-1 \). Let
  \( \mm'\in \HH_{n-1} \) and \( \mm''\in \HH_{n-2} \)
  be the Hessenberg functions given by \( \mm'(i) = \mm(i+1)-1 \) for
  \( i\in [n-1] \) and \( \mm''(i) = \mm(i+2)-2 \) for
  \( i\in [n-2] \). Let \( T''\in\SYT(n-1) \) be the tableau obtained
  from \( T' \) by removing the entry \( n-1 \) and suppose
  \( \vec \delta^{(r')}(T'') =\left(1^{b'_0}, 0^{a'_1}, 1^{b'_1},
    \ldots, 0^{a'_{l'}}, 1^{b'_{l'}}, 0^{a'_{l'+1}}\right) \), where
  \( r'= n-1-\mm'(1) = n-\mm(2) \).
  Since \( \overline{p}_\mm(T;q) \ne 0 \), we have \( T=f_{k}^{(r)}(T') \) for
  some \( 0\le k\le l \). Then
  \( \overline{p}_\mm(T;q) = \varphi^{(r)}_k(T';q) \overline{p}_{\mm'}(T;q) \), so
  \( \overline{p}_{\mm'}(T';q) \ne 0 \). Thus, by the induction hypothesis, every
  column of \( T' \) has at most one integer greater than \( r' \).
  Moreover, by the construction of \( T=f_{k}^{(r)}(T') \), the column
  of \( T \) containing \( n \) does not contain any integer greater
  than \( r' \), except for \( n \). Since
  \( r=n-\mm(1)\ge n-\mm(2)=r' \), we obtain that every column of
  \( T \) has at most one integer greater than \( r \), as desired.

  For the second statement, observe that \( \sum_{i=0}^l b_i \) is the
  number of columns of \( T' \) that contain an integer greater than
  \( r \). By the first statement, this is exactly the number of
  integers greater than \( r \) in \( [n] \), which is
  \( n+1-r=\mm(1)-1 \). This completes the proof.
\end{proof}

\begin{prop}
  \label{prop:modified_p}
  Let \( \mm\in \HH_n \).
  For \( \lambda\vdash n \) and \( T\in\SYT(\lambda) \),
  \begin{equation}\label{eq:8}
    p_\mm(T;q) = q^{\area(\mm)-\sum_{j=1}^{\ell(\lambda)} \binom{\lambda_j}{2}}\overline{p}_\mm(T;q).
  \end{equation}
\end{prop}

\begin{proof}
  We proceed by induction on \( n \), where the case \( n=0 \) is
  trivial. Let \( n\ge1 \) and suppose that \eqref{eq:8} holds for
  \( n-1 \). Let \( \mm'\in \HH_{n-1} \) be and
  let \( T'\in\SYT(n-1) \) be the tableau obtained
  from \( T \) by removing the entry \( n \). We may assume that
  \( \overline{p}_\mm(T;q) \ne 0 \), which is equivalent to
  \( p_\mm(T;q) \ne 0 \), because otherwise
  \( \overline{p}_\mm(T;q) = p_\mm(T;q) = 0 \), and thus \eqref{eq:8}
  holds.

  Suppose
  \( \vec\delta^{(r)}(T') =\left(1^{b_0}, 0^{a_1}, 1^{b_1}, \ldots,
    0^{a_l}, 1^{b_l}, 0^{a_{l+1}}\right) \), where \( r=n-\mm(1) \).
  Since \( \overline{p}_\mm(T;q) \ne 0 \), we have \( T=f_{k}^{(r)}(T') \) for
  some \( 0\le k\le l \). By the induction hypothesis,
  \[
    p_\mm(T;q)
    = \psi_{k}^{(r)}(T';q) p_{\mm'}(T;q)
    = \psi_{k}^{(r)}(T';q) \cdot
    q^{\area(\mm')-\sum_{j=1}^{\ell(\mu)} \binom{\mu_j}{2}}
    \overline{p}_{\mm'}(T;q),
  \]
  where \( \mu \) is the shape of \( T' \). Since
  \( \overline{p}_\mm(T;q) = \varphi_{k}^{(r)}(T';q) \overline{p}_{\mm'}(T';q) \),
  in order to prove \eqref{eq:8}, it suffices to show that
  \[
    \varphi_{k}^{(r)}(T';q) \cdot
    q^{\area(\mm)-\sum_{j=1}^{\ell(\lambda)} \binom{\lambda_j}{2}}
   = \psi_{k}^{(r)}(T';q) \cdot
    q^{\area(\mm')-\sum_{j=1}^{\ell(\mu)} \binom{\mu_j}{2}}.
  \]
  By \eqref{eq:29}, the above identity is
  equivalent to
  \begin{equation}\label{eq:16}
    \sum_{i=1}^k a_i + \area(\mm)
    -\sum_{j=1}^{\ell(\lambda)} \binom{\lambda_j}{2}
    = \sum_{i=k+1}^l b_i + \area(\mm')
    -\sum_{j=1}^{\ell(\mu)} \binom{\mu_j}{2}.
  \end{equation}

  Note that \( \area(\mm)=\sum_{i=1}^{n}\left( \mm(i)-i \right)=
  \area(\mm') +\mm(1)-1 \). Since
  \( T=f_{k}^{(r)}(T') \), the unique cell in \( \lambda/\mu \) is in
  column \( c+1 \), where
  \( c = \sum_{j=1}^{k} a_j + \sum_{j=0}^{k} b_j \). Thus,
  \( \sum_{j=1}^{\ell(\lambda)}\binom{\lambda_j}{2} =
  \sum_{j=1}^{\ell(\mu)}\binom{\mu_j}{2} + c \). Combining these
  observations, we can rewrite \eqref{eq:16} as
  \( \mm(1)-1 = \sum_{i=0}^l b_i \), which is proved in \Cref{lem:1}.
  This completes the proof.
\end{proof}

\bibliographystyle{abbrv}

\begin{thebibliography}{10}

\bibitem{Abreu2021b}
A.~Abreu and A.~Nigro.
\newblock {A symmetric function of increasing forests}.
\newblock {\em Forum Math. Sigma}, 9:21, 2021.
\newblock Id/No e35.

\bibitem{Abreu2021}
A.~Abreu and A.~Nigro.
\newblock {Chromatic symmetric functions from the modular law}.
\newblock {\em J. Combin. Theory Ser. A}, 180:Paper No. 105407, 30, 2021.

\bibitem{Abreu2023}
A.~Abreu and A.~Nigro.
\newblock {Splitting the cohomology of Hessenberg varieties and e-positivity of
  chromatic symmetric functions}.
\newblock {\it Preprint},
  \href{https://arxiv.org/abs/2304.10644}{arXiv:2304.10644}, 2023.

\bibitem{Beilinson}
A.~A. Beilinson, J.~Bernstein, and P.~Deligne.
\newblock Faisceaux pervers.
\newblock {Analysis and topology on singular spaces, I (Luminy, 1981),
  Ast\'erisque, vol. 100, Soc. Math. France, Paris, 1982, pp. 5--171}.

\bibitem{Brosnan2018}
P.~Brosnan and T.~Y. Chow.
\newblock {Unit interval orders and the dot action on the cohomology of regular
  semisimple {H}}essenberg varieties.
\newblock {\em Adv. Math.}, 329:955--1001, 2018.

\bibitem{Chow0}
T.~Y. Chow.
\newblock A note on a combinatorial interpretation of the \(e\)-coefficients of
  the chromatic symmetric function.
\newblock {\it Unpublished manuscript},
  \url{https://timothychow.net/ecoeff.pdf}.

\bibitem{Gasharov1996}
V.~Gasharov.
\newblock {Incomparability graphs of {$(3+1)$}}-free posets are {$s$}-positive.
\newblock In {\em Proceedings of the 6th {C}onference on {F}ormal {P}ower
  {S}eries and {A}lgebraic {C}ombinatorics ({N}ew {B}runswick, {NJ}, 1994)},
  volume 157, pages 193--197, 1996.

\bibitem{Griffin2025}
S.~T. Griffin, A.~Mellit, M.~Romero, K.~Weigl, and J.~J. Wen.
\newblock {On Macdonald expansions of $q$-chromatic symmetric functions and the
  Stanley-Stembridge Conjecture}.
\newblock {\it Preprint},
  \href{https://arxiv.org/abs/2504.06936}{arXiv:2504.06936}, 2025.

\bibitem{Guay-Paquet2013}
M.~Guay-Paquet.
\newblock {A modular relation for the chromatic symmetric functions of
  (3+1)-free posets}.
\newblock {\it Preprint},
  \href{https://arxiv.org/abs/1306.2400}{arXiv:1306.2400}, 2013.

\bibitem{Guay-Paquet2016}
M.~Guay-Paquet.
\newblock {A second proof of the Shareshian--Wachs conjecture, by way of a new
  Hopf algebra}.
\newblock {\it Preprint},
  \href{https://arxiv.org/abs/1601.05498}{arXiv:1601.05498}, 2016.

\bibitem{Hikita2024}
T.~Hikita.
\newblock {A proof of the Stanley-Stembridge conjecture}.
\newblock {\it Preprint},
  \href{https://arxiv.org/abs/2410.12758}{arXiv:2410.12758}, 2024.

\bibitem{Kim2024c}
D.~Kim and J.~Oh.
\newblock {Extending the science fiction and the Loehr--Warrington formula}.
\newblock {\it Preprint},
  \href{https://arxiv.org/abs/2409.01041}{arXiv:2409.01041}, 2024.

\bibitem{Shareshian2016}
J.~Shareshian and M.~L. Wachs.
\newblock {Chromatic quasisymmetric functions}.
\newblock {\em Adv. Math.}, 295:497--551, 2016.

\bibitem{Stanley1995}
R.~P. Stanley.
\newblock {A symmetric function generalization of the chromatic polynomial of a
  graph}.
\newblock {\em Adv. Math.}, 111(1):166--194, 1995.

\bibitem{Stanley1993}
R.~P. Stanley and J.~R. Stembridge.
\newblock {On immanants of {J}}acobi-{T}rudi matrices and permutations with
  restricted position.
\newblock {\em J. Combin. Theory Ser. A}, 62(2):261--279, 1993.

\bibitem{Tom2025}
F.~Tom and A.~Vailaya.
\newblock {The chromatic symmetric function of graphs glued at a single
  vertex}.
\newblock {\it Preprint},
  \href{https://arxiv.org/abs/2503.19344}{arXiv:2503.19344}, 2025.

\bibitem{Tymoczko2008}
J.~S. Tymoczko.
\newblock {Permutation actions on equivariant cohomology of flag varieties}.
\newblock In {\em Toric topology}, volume 460 of {\em Contemp. Math.}, pages
  365--384. Amer. Math. Soc., Providence, RI, 2008.

\end{thebibliography}

\end{document}